\documentclass[a4paper, 12pt, twoside, reqno]{amsart}

\usepackage[utf8]{inputenc}
\usepackage[T1]{fontenc}

\usepackage{geometry}
\usepackage{amsmath}
\usepackage{amsthm}
\usepackage{amssymb}
\usepackage{amsfonts}
\usepackage{mathtools}
\usepackage[inline]{enumitem}
\usepackage{fancyhdr}
\usepackage{physics}

\usepackage{color}
\usepackage{graphicx}

\usepackage{tikz}
\usepackage[most]{tcolorbox}
\usepackage[colorinlistoftodos]{todonotes}
\usetikzlibrary{tikzmark}

\usepackage{xspace}

\allowdisplaybreaks

\numberwithin{equation}{section}

\newtheorem{theorem}{Theorem}[section]
\newtheorem{proposition}[theorem]{Proposition}
\newtheorem{corollary}[theorem]{Corollary}
\newtheorem{lemma}[theorem]{Lemma}

\theoremstyle{definition}

\theoremstyle{remark}
\newtheorem{remark}[theorem]{Remark}
\newtheorem{claim}{Claim}

\usepackage{hyperref}
\hypersetup{%
colorlinks=true, %
linkcolor=blue, %
citecolor=blue, %
filecolor=blue, %
urlcolor=blue, %
pdfauthor={Ronaldo B. Assun\c{c}\~{a}o,
Ol\'{\i}mpio H. Miyagaki, and
Rafaella F. S. Siqueira},
pdftitle={Fractional Sobolev-Chochard critical equation with Hardy term and weighted singularities}, %
pdfsubject={Elliptic partial differential equations}, %
pdfkeywords={Fractional p-Laplacian operator, doubly critical singular problem, variational
methods, weighted Sobolev spaces, weighted Morrey spaces, Caffarelli-Kohn-Nirenberg inequality.}, %
pdfproducer={Latex}, %
pdfcreator={ps2pdf}}

\usepackage[hyperpageref]{backref}

\title[Fractional singular Sobolev-Choquard
critical equation]{Fractional Sobolev-Choquard
critical equation with Hardy
term and weighted
singularities}

\author[Assun\c{c}\~{a}o]{Ronaldo B. Assun\c{c}\~{a}o}
\address{Ronaldo B. Assun\c{c}\~{a}o \hfill\break\indent
Departamento de Matem\'{a}tica\,---\,
Universidade Federal de Minas Gerais, 
UFMG \hfill\break\indent
Av.~Ant\^{o}nio Carlos, 6627\,---
CEP 30161-970\,---\,Belo Horizonte, MG, Brasil}
\email{ronaldo@mat.ufmg.br}

\author[Miyagaki]{Ol\'{\i}mpio H. Miyagaki}
\thanks{Ol\'{\i}mpio H. Miyagaki was supported by Grant 2019/24901-3\,---\,S\~{a}o Paulo Research Foundation (FAPESP) and Grant 303256/2022-2\,---\,CNPq/Brazil.}
\address{Ol\'{\i}mpio H. Miyagaki \hfill\break\indent
Departamento de Matem\'{a}tica\,---\,
Universidade Federal de S\~{a}o Carlos, 
UFSCar \hfill\break\indent
CEP 13565-905\,---\,S\~{a}o Carlos, SP, Brasil} 
\email{ohmiyagaki@gmail.com}

\author[Siqueira]{Rafaella F. S. Siqueira}
\thanks{Rafaella F. S. Siqueira was partially supported by CNPq/Brazil.}
\address{Rafaella F. S. Siqueira 
\hfill\break\indent
Departamento de Matem\'{a}tica\,---\,
Instituto Federal Fluminense, 
IFF \hfill\break\indent
CEP 28360-000\,---\,Bom Jesus do Itabapoana, RJ, Brasil}
\email{rafaella.siqueira@iff.edu.br}

\date{Belo Horizonte, \today}



\begin{document}
\begin{abstract}
In this paper we consider a fractional $p$-Laplacian equation in the entire space $\mathbb{R}^{N}$ with doubly critical singular nonlinearities involving a local critical Sobolev  term together with a nonlocal Choquard critical term; the problem also includes a homogeneous singular Hardy term.
More precisely, we deal with the problem
\begin{align*}
  \begin{cases}
  (-\Delta)^{s}_{p,\theta} u 
   -\gamma \dfrac{|u|^{p-2}u}{|x|^{sp+ \theta}} 
   = \dfrac{|u|^{p^*_s(\beta,\theta)-2}u}%
     {|x|^{\beta}} 
     + \left[
       I_{\mu} \ast F_{\delta,\theta,\mu}(\cdot, u) 
       \right](x)f_{\delta,\theta,\mu}(x,u)\\
  u \in \dot{W}^{s,p}_{\theta}(\mathbb{R}^N)
  \end{cases}
\end{align*}
where 
$0 < s  < 1$; 
$0 < \alpha, \,\beta < sp + \theta < N$; 
$0 < \mu < N$;
$2\delta + \mu < N$;
$\gamma < \gamma_{H}$
with the best fractional Hardy constant $\gamma_{H}$;
the Hardy-Sobolev and Stein-Weiss upper critical fractional exponents (this latter also called Hardy-Littlewood-Sobolev upper critical exponent) are
respectively defined by
$p^*_s(\beta,\theta)
\coloneqq p(N-\beta)/(N-sp-\theta)$,
and $p^\sharp_s(\delta,\theta,\mu)
\coloneqq p(N-\delta-\mu/2)/(N-sp-\theta)$.
Moreover,
$I_{\mu}(x) =|x|^{-\mu}$ is the Riesz potencial; 
$f_{\delta,\theta,\mu}(x,t)
\coloneqq |x|^{-\delta}
|t|^{p^{\sharp}_{s}(\delta,\theta,\mu)-2}t$ 
    and
$F_{\delta,\theta,\mu}(x,t)
\coloneqq |x|^{\delta}
|t|^{p^{\sharp}_{s}(\delta,\theta,\mu)}$;
and the term with convolution integral is known as Choquard type nonlinearity. 
To prove the main result we have to show new embeddings involving the weighted Morrey spaces.
For $1 \leqslant q <  p^*_s(\alpha,\theta)$ and
    $r= \frac{\alpha}{p^*_s(\alpha,\theta)}$, it holds
\begin{align*}
\dot{W}^{s,p}_{\theta}(\mathbb{R}^N) \hookrightarrow L^{p^*_s(\alpha,\theta)}(\mathbb{R}^N, |y|^{-\alpha}) \hookrightarrow L_{M}^{q, \frac{(N-sp-\theta)q}{p}+qr}(\mathbb{R}^N, |y|^{-pr})
\end{align*}
and the norms in these spaces share the same dilation invariance.
We also have to prove a version of the Caffarelli-Kohn-Nirenberg inequality.
Let $s\in (0,1)$ and $0< \beta < sp +\theta< N$. Then there exists $C = C(N, s, \beta)>0$ such that for any $\zeta \in (\Bar{\zeta}, 1)$ and for any $q\in [1, p^*_s(\beta,\theta))$, it holds
    \begin{align*}
\left(\int_{\mathbb{R}^N}\frac{|u(y)|^{p^*_s(\beta,\theta)}}{|y|^{\beta}}\dd y\right)^{\frac{1}{p^*_s(\beta,\theta)}} \leqslant \|u\|_{\Dot{W}^{s,p}_{\theta}}^{\zeta} \|u\|^{1-\zeta}_{L_M^{q,\frac{N-sp-\theta}{p}q+qr}(\mathbb{R}^N, |y|^{-qr})}
    \end{align*}
for all $u \in \Dot{W}^{s,p}_{\theta}(\mathbb{R}^N)$,
    where $\Bar{\zeta} = \max 
    \{p/p^*_s(\beta,\theta), (p^*_s(0,\theta)-1)/p^*_s(\beta,\theta)\}>0$ and $r=\beta/p^*_s(\beta,\theta)$.
    With the help of these new embedding results, we provide sufficient conditions under which a weak nontrivial solution to the problem exists via variational methods.
\end{abstract}

\maketitle

\tableofcontents 

\section{Model problems and main results}
\label{cap:1}

\subsection{Introduction}
In the present paper, we consider the following fractional $p$-Laplacian equation in the entire space $\mathbb{R}^{N}$ featuring doubly critical nonlinearities, involving a local critical Sobolev term together with a nonlocal Choquard critical term; the problem also includes a homogeneous Hardy term; additionaly, all terms have critical singular weights.
More precisely, we deal with the problem
\begin{align}
  \label{problema0.1}
  \begin{cases}
  (-\Delta)^{s}_{p,\theta} u 
   -\gamma \dfrac{|u|^{p-2}u}{|x|^{sp+ \theta}} 
   = \dfrac{|u|^{p^*_s(\beta,\theta)-2}u}%
     {|x|^{\beta}} 
     + \left[
       I_{\mu} \ast F_{\delta,\theta,\mu}(\cdot, u) 
       \right](x)f_{\delta,\theta,\mu}(x,u)\\
  u \in \dot{W}^{s,p}_{\theta}(\mathbb{R}^N)
  \end{cases}
\end{align}
where 
$0 < s  < 1$; 
$0 < \alpha, \,\beta < sp + \theta < N$; 
$0 < \mu < N$;
$2\delta + \mu < N$;
$\gamma < \gamma_{H}$
with the best fractional Hardy constant $\gamma_{H}$ 
to be defined below;
the Hardy-Sobolev and Stein-Weiss upper critical fractional exponents (this latter also called Hardy-Littlewood-Sobolev upper critical exponent) are
respectively defined by
\begin{alignat*}{2}
p^*_s(\beta,\theta)
& = \frac{p(N-\beta)}{N-sp-\theta},
& \qquad
p^\sharp_s(\delta,\theta,\mu)
& = \frac{p(N-\delta-\mu/2)}{N-sp-\theta}.
\end{alignat*}
Moreover,
$I_{\mu}(x) =|x|^{-\mu}$ is the Riesz potential of order $\mu$; 
the functions 
$f_{\delta,\theta,\mu}, 
 F_{\delta,\theta,\mu} \colon \mathbb{R}^{N} \times \mathbb{R} \to \mathbb{R}$ are respectively defined by
\begin{alignat}{2}
\label{def:f}
f_{\delta,\theta,\mu}(x,t)
& = \frac{|t|^{p^{\sharp}_{s}(\delta,\theta,\mu)-2}t}%
    {|x|^{\delta}},
& \qquad
F_{\delta,\theta,\mu}(x,t)
& = \frac{|t|^{p^{\sharp}_{s}(\delta,\theta,\mu)}}%
         {|x|^{\delta}},
\end{alignat}
that is, $F_{\delta,\theta,\mu}(x,t)=p^{\sharp}_{s}(\delta,\theta,\mu)
\int_{0}^{|t|}
f_{\delta,\theta,\mu}(x,\tau)\dd{\tau}$;
and the term with convolution integral, 
\begin{align*}
[I_{\mu}\ast F_{\delta,\theta,\mu}(\cdot,u)](x)
& \coloneqq \int_{\mathbb{R}^{N}} \dfrac{|u(y)|^{p^\sharp_s(\delta,\theta,\mu)}}{|x-y|^{\mu}|y|^{\delta}}\dd{y},
\end{align*}
is known as Choquard type nonlinearity. Intuitively, problem~\eqref{problema0.1} is 
understood as showing the existence of a function 
$u \in \dot{W}^{s,p}_{\theta}(\mathbb{R}^N)$ such that
\begin{align*}
  (-\Delta)^{s}_{p,\theta} u 
   -\gamma \dfrac{|u|^{p-2}u}{|x|^{sp+ \theta}} 
   = \dfrac{|u|^{p^*_s(\beta,\theta)-2}u}{|x|^{\beta}} 
     + \biggl( \displaystyle
      \int_{\mathbb{R}^{N}}
      \dfrac{|u(y)|^{p^{\sharp}_{s}(\delta,\theta,\mu)}}{|x-y|^{\mu}|y|^{\delta}}\dd{y}
     \biggr)
     \dfrac{|u(x)|^{p^{\sharp}_{s}(\delta,\theta,\mu)-2}u(x)}{|x|^{\delta}}
\end{align*}
where the fractional $p$-Laplacian operator is defined 
for $\theta=\theta_{1}+\theta_{2}$, $x\in\mathbb{R}^{N}$, 
and any function $u \in C_{0}^{\infty}(\mathbb{R}^{N})$, as
\begin{align*}
(-\Delta)_{p,\theta}^{s}u(x)
&\coloneqq
\textup{p.v.\xspace}
\int_{\mathbb{R}^{N}} 
\dfrac{|u(x)-u(y)|^{p-2}(u(x)-u(y))}{|x|^{\theta_{1}}
|x-y|^{N+sp}|y|^{\theta_{2}}}\dd{y},
\end{align*} 
and p.v.\xspace\ is the
Cauchy's principal value.

Let us now introduce the spaces of functions that are meaningful to our considerations. Throughout this work, 
we denote the norm of the weighted Lebesgue space
$L^{p}(\mathbb{R}^{N},|x|^{-\lambda})$ by
\begin{align*}
\|u\|_{L^{p}(\mathbb{R}^{N};|x|^{-\lambda})}
& \coloneqq
 \biggl(
\int_{\mathbb{R}^{N}} \dfrac{|u|^{p}}{|x|^{\lambda}}\dd{x} \biggr)^{\frac{1}{p}}
\end{align*}
for any 
$0\leqslant \lambda < N$ and $1 \leqslant p < +\infty$. 

We say that a Lebesgue measurable function 
$u \colon \mathbb{R}^{N} \to \mathbb{R}$ 
belongs to the weighted Morrey space 
$L_{M}^{p,\gamma+\lambda}(\mathbb{R}^{N},|x|^{-\lambda})$ 
if 
\begin{align*}
\|u\|_{L_{M}^{p,\gamma+\lambda}(\mathbb{R}^{n},|x|^{-\lambda})}
& \coloneqq
\sup_{x\in\mathbb{R}^{N},\: R\in \mathbb{R}_{+}}
\biggl\{
\biggl(
R^{\gamma+\lambda-N}
\int_{B_{R}(x)} \dfrac{|u|^{p}}{|x|^{\lambda}}\dd{x}
\biggr)^{\frac{1}{p}} \biggr\} 
< +\infty,
\end{align*}
where $1 \leqslant p < +\infty$;
$\gamma, \, \lambda \in \mathbb{R}_{+}$,
and
$0 < \gamma+\lambda < N$.

Our concerns involve the homogeneous fractional Sobolev-Slobodeckij space $\dot{W}^{s,p}_{\theta}(\mathbb{R}^N)$ 
defined as the completion of the space 
$C_{0}^{\infty}(\mathbb{R}^{N})$
with respect to the Gagliardo seminorm given by
\begin{align*}
u \mapsto
[u]_{\dot{W}^{s,p}_{\theta}(\mathbb{R}^N)}
& \coloneqq
\biggl(
\int_{\mathbb{R}^N} \int_{\mathbb{R}^N} 
\frac{|u(x)-u(y)|^{p}}%
{|x|^{\theta _1}|x-y|^{N+sp}|y|^{\theta _2}}
\dd x \dd y
\biggr)^{\frac{1}{p}},
\end{align*}
i.e.,\xspace
$\dot{W}^{s,p}_{\theta}(\mathbb{R}^N) = 
\overline{C_{0}^{\infty}(\mathbb{R}^{N})}^{[\,\cdot\,]} $.
We can equip the homogeneous fractional Sobolev space 
$\dot{W}^{s,p}_{\theta}(\mathbb{R}^N)$ 
with the norm 
\begin{align*}
\|u \|_{\dot{W}_{\theta}^{s,p} (\mathbb{R}^N)} 
& \coloneqq
\biggl(
\int_{\mathbb{R}^N} \int_{\mathbb{R}^N} 
\dfrac{|u(x)-u(y)|^{p}}%
{|x|^{\theta _1}|x-y|^{N+sp}|y|^{\theta _2}}
\dd x \dd y
- \gamma
\int_{\mathbb{R}^N}  
\dfrac{|u|^{p}}{|x|^{sp+\theta}}
\dd x
\biggr)^{\frac{1}{p}}\\
& \coloneqq \bigl(
[u]_{\dot{W}_{\theta}^{s,p} (\mathbb{R}^n)}^{p}
-\gamma \|u \|_{L^p (\mathbb{R}^N; |x|^{-sp-\theta})}^{p}
\bigr)^{\frac{1}{p}}.
\end{align*} 
Here, we assume that $\gamma < \gamma_{H}$, where the
best fractional Hardy constant is defined by
 \begin{align*}
\gamma_{H}
& \coloneqq 
\inf_{\substack{u \in \dot{W}^{s,p}_{\theta}(\mathbb{R}^n) \\ u \neq 0}}
\dfrac{[u]_{\dot{W}^{s,p}_{\theta}(\mathbb{R}^n)}^{p}}%
{\|u\|_{L^{p}(\mathbb{R}^{n};|x|^{-sp-\theta})}^{p}}.
\end{align*}
This turns the space 
$\dot{W}_{\theta}^{s,p} (\mathbb{R}^N)$
into a Banach space; moreover, this space is uniformly convex; in particular, it is reflexive and separable.

Our main goal in this work is to show that 
problem~\eqref{problema0.1} admits at least a weak solution, by which term we mean a function 
$u \in \dot{W}^{s,p}_{\theta}(\mathbb{R}^N)$ such that
\begin{align*}
\lefteqn{
\int_{\mathbb{R}^{N}}
\int_{\mathbb{R}^{N}} 
\dfrac{|u(x)-u(y)|^{p-2}
(u(x)-u(y))(\phi(x)-\phi(y))}{|x|^{\theta_1}|x-y|^{N+sp}
|y|^{\theta_2}}
\dd{x}\dd {y}
 -\gamma 
 \int_{\mathbb{R}^{N}} 
 \dfrac{|u|^{p-2}u \phi}{|x|^{sp+\theta}}\dd{x}} \\
& = \int_{\mathbb{R}^{N}} 
\dfrac{|u|^{p_{s}^{\ast}(\beta,\theta)-2}u \phi}%
{|x|^{\beta}}\dd{x} 
+\int_{\mathbb{R}^{N}}
 \int_{\mathbb{R}^{N}}
 \dfrac{|u(x)|^{p_{s}^{\sharp}(\delta,\theta,\mu)-2}
 u(x) \phi(x)
        |u(y)|^{p_{s}^{\sharp}(\delta,\theta,\mu)-2}
        u(y) \phi(y)}%
       {|x|^{\delta}
        |x-y|^{\mu}
        |y|^{\delta}}
        \dd{x}\dd{y}
\end{align*}
for any test function $\phi \in \dot{W}^{s,p}_{\theta}(\mathbb{R}^N)$. 

Now we define the energy functional 
$I \colon \dot{W}^{s,p}_{\theta}(\mathbb{R}^N)
\to \mathbb{R}$ by
\begin{align*}
I(u)
& \coloneqq
\dfrac{1}{p}
\int_{\mathbb{R}^{N}}
\int_{\mathbb{R}^{N}} 
\dfrac{|u(x)-u(y)|^{p}}%
{|x|^{\theta_{1}}|x-y|^{N+sp}|y|^{\theta_{2}}} \dd{x} \dd{y}
-\dfrac{\gamma}{p}
\int_{\mathbb{R}^{N}} 
\dfrac{|u|^{p}}{|x|^{sp+\theta}} \dd{x}\\
& \quad
-\dfrac{1}{p_{s}^{\ast}(\beta,\theta)}
\int_{\mathbb{R}^{N}} 
\dfrac{|u|^{p_{s}^{\ast}(\beta,\theta)}}{|x|^{\beta}} \dd{x}
-\dfrac{1}{2 p_{s}^{\sharp}(\delta,\theta,\mu)}
\int_{\mathbb{R}^{N}}
\int_{\mathbb{R}^{N}}
\dfrac{|u(x)|^{p_{s}^{\sharp}(\delta,\theta,\mu)}
       |u(y)|^{p_{s}^{\sharp}(\delta,\theta,\mu)}}%
       {|x|^{\delta}
        |x-y|^{\mu}
        |y|^{\delta}}
        \dd{x}\dd{y}.
\end{align*}
For the parameters in the previously specified intervals, the energy functional $I$ is well defined and is continuously differentiable, 
i.e.,\xspace $I \in C^{1}(\dot{W}^{s,p}_{\theta}(\mathbb{R}^N);\mathbb{R})$; moreover, a nontrivial critical point of the energy functional $I$ is a nontrivial weak solution to problem~\eqref{problema0.1}.

\begin{theorem}
\label{teo:1.1chineses}
Problem~\eqref{problema0.1} has at least a nontrivial weak solution provided that 
$0 < s < 1$;
$0 < \alpha, \, \beta < sp + \theta < N$; 
$0 < \mu < N$; and $\gamma < \gamma _H$.
\end{theorem}

In this work we also consider the following variants of problem~\eqref{problema0.1}, namely one problem with a Hardy potential and double Sobolev type nonlinearities,
\begin{align}
  \label{problema0.2}
  \begin{cases}
  (-\Delta)^{s}_{p,\theta} u 
   -\gamma \dfrac{|u|^{p-2}u}{|x|^{sp+ \theta}} 
   = \displaystyle\sum_{k=1}^{k=2}
   \dfrac{|u|^{p^*_s(\beta_{k},\theta)-2}u}%
     {|x|^{\beta_{k}}} \\
  u \in \dot{W}^{s,p}_{\theta}(\mathbb{R}^N);
  \end{cases}
\end{align}
and another one with a Hardy potential and double Choquard type nonlinearities,
\begin{align}
\label{problema0.3}
\begin{cases}
  (-\Delta)^{s}_{p,\theta} u 
   -\gamma \dfrac{|u|^{p-2}u}{|x|^{sp+ \theta}} 
   = \displaystyle\sum_{k=1}^{k=2}\left[
       I_{\mu_{k}} \ast F_{\delta_{k},\theta,\mu_{k}}(\cdot, u) 
       \right](x)
       f_{\delta_{k},\theta,\mu_{k}}(x,u) \\
  u \in \dot{W}^{s,p}_{\theta}(\mathbb{R}^N).
  \end{cases}
\end{align}
The notion of weak solution to problems~\eqref{problema0.2}  
and~\eqref{problema0.3} 
can be defined in the same way as that for problem~\eqref{problema0.1}, i.e.,\xspace
we multiply the differential equations by test functions and use a kind of integration by parts. Then we recognize these expressions as the derivatives of an energy functionals which, under the appropriate hypotheses on the parameters, are continuously differentiables. This means that weak solutions to these problems are critical points of the appropriate energy functionals. By adaptting the method used in the proof of Theorem~\ref{teo:1.1chineses} we deduce the following result.
\begin{theorem}
\label {teo:1.2chineses}
Problems~\eqref{problema0.2} and~\eqref{problema0.3} have at least a nontrivial weak solution
under similar assumptions as in Theorem~\ref{teo:1.1chineses}, i.e.,\xspace
$0 < s < 1$;
$\gamma < \gamma _H$;
$0 < \alpha_{k}, \,\beta_{k} < sp + \theta < N$; 
and $0 < \mu_{k} < N$
for $k \in \{1,2\}$.
\end{theorem}

\begin{remark}
\label{rem:li_yang}
At this point let us mention that, in the past, other authors have attempted to prove existence results for this class of fractional elliptic problems. To be precise, 
Li \& Yang~\cite{Li_2020} claimed to have established the existence of solution to problem~\eqref{problema0.1} in the case $p=2$ for the fractional Laplacian, but without singularities in the operator, i.e., $\theta_{1} = \theta_{2} = 0$, and in the unweighted Choquard term, i.e.,\xspace
$\alpha=0$, or in the unweighted Sobolev term, i.e.,\xspace $\beta=0$. Their proofs rely on a related minimization problem.
However, we could not check the arguments
on which the proof is based; 
see Yang \& Wu~\cite[inequality~(2.8)]{yang2017doubly};
Yang~\cite[inequality~(3.2)]{yang}.
In this way, we believe that their results in these cases are still open problems; see De N\'{a}poli, Drelichman \& Salort~\cite{salort}.

\end{remark}

\subsubsection*{Notation.}
For $\rho \in \mathbb{R}_{+}$, we define 
$B_{\rho}(x)\coloneqq
\{y\in\mathbb{R}^{N} \colon |x-y|<\rho\}$, the open ball centered at $x$ with radius $\rho$. The constant $\omega_{N}$ denotes the volume of the unit ball in $\mathbb{R}^{N}$.
The arrows $\to$ and $\rightharpoonup$ denote the strong convergence and the weak convergence, respectively.
Given the functions $f,g \colon \mathbb{R}^{N}\to\mathbb{R}$, we recall that $f=O(g)$ if there is a constant $C\in \mathbb{R}_{+}$ such that 
$|f(x)|\leqslant C|g(x)|$ for all
$x \in \mathbb{R}^{N}$;
and $f=o(g)$ as $x\to x_{0}$ if 
$\lim_{x\to x_{0}} |f(x)|/|g(x)|=0$.
The pair $r$ and $r'$ denote H\"{o}lder conjugate exponents, i.e.,\xspace 
$1/r + 1/r'=1$ or $r+r'=r r'$.
The positive and negative parts of a function $\phi$ are denoted by $\phi_{\pm}\coloneqq \max\{\pm\phi,0\}$.
Finally, $C\in\mathbb{R}_{+}$ denotes a universal constant that may change from line to line; when it is relevant, we will add subscripts to specify the dependence of certain parameters.

\subsection{Historical background}
Reasons for the recent interest in this class of nonlinear elliptic problems reside in the merits of the subject itself and also in the number and variety of phenomena occurring in real-world applications that can be modeled by these equations.
For example, fractional and nonlocal differential operators arise in a quite natural way in many different problems that involve long-range interactions, such as anomalous diffusion, dislocations in crystals, water waves, phase transitions, stratified materials, semipermeable membranes, flame propagation, non-Newtonian fluid theory in a porous medium, financial mathematics, phase transition phenomena, population dynamics, minimum surfaces, game theory, image processing, etc.\@\xspace
In particular,
there are some
remarkable mathematical models involving the fractional $p$-Laplacian, such as the fractional Schr\"{o}dinger equation, the fractional Kirchhoff equation, the
fractional porous medium equation, etc.\@\xspace
For more information, see the excelent survey papers by 
Di Nezza, Palatucci \& Valdinoci~\cite{MR2944369},
Moroz \& Van Schaftingen~\cite{MR3625092} and
Mukherjee \& Sreenadh~\cite{MR4300845}
and the references they contain.
We can also mention the diversity of tools used in their study, mainly critical point theory and variational and topological methods.

\subsubsection*{The fractional Laplacian}
There are many equivalent definitions of the fractional Laplacian. 
In our case, on the Euclidean space $\mathbb{R}^{N}$ of dimension $N \geqslant 1$, for $\theta=\theta_{1}+\theta_{2}$ and the above specified intervals for the parameters, we define the nonlocal elliptic $p$-Laplacian operator with the help of the Cauchy's principal value integral as
\begin{align*}
(-\Delta)_{p,\theta}^{s}u(x)
&\coloneqq
\textup{p.v.\xspace}
\int_{\mathbb{R}^{N}} 
\dfrac{|u(x)-u(y)|^{p-2}(u(x)-u(y))}{|x|^{\theta_{1}}
|x-y|^{N+sp}|y|^{\theta_{2}}}\dd{y}
\\
&\coloneqq 
2
\lim_{\varepsilon \to 0}
\int_{\mathbb{R}^N\backslash 
B_{\varepsilon}(x)}
\dfrac{|u(x)-u(y)|^{p-2}(u(x)-u(y))}{|x|^{\theta_{1}}
|x-y|^{N+sp}|y|^{\theta_{2}}}\dd{y}
\end{align*} 
for $x\in\mathbb{R}^{N}$ and any function $u \in C_{0}^{\infty}(\mathbb{R}^{N})$. The usual definition of the fractional $p$-Laplacian carries a normalizing constant dependent on $N$, $s$, $p$, and $\theta$ in front of the integral. This constant is irrelevant for our purposes; so, for the sake of clarity, we omit it in the definition and in the formulation of the results. The limit operator (up to suitable normalizing constant) as $s \to 1^{-}$ and $\theta_{1}=\theta_{2}=0$ is the so called $p$-Laplacian defined as $\Delta_{p}u(x)=\nabla\cdot(|\nabla u(x)|^{p-2}\nabla u(x))$. Also, if $p=2$ and $\theta_{1}=\theta_{2}=0$, then the usual notation is $(-\Delta)^{s}u$ and the definition is consistent, up to a normalizing constant, to the fractional Laplacian defined by the
Fourier multiplier
$\mathcal{F}[(-\Delta)^{s}u(x)](\xi)
=2\pi|\xi|^{2s}\mathcal{F}[u(x)](\xi)$ 
for $x, \, \xi \in \mathbb{R}^{N}$.
In this formula, 
$\mathcal{F}[u(x)](\xi)=\int_{\mathbb{R}^{N}} \exp(-2\pi i x \cdot \xi) u(x) \dd{x}$ denotes the Fourier transform of $u \in \mathcal{S}(\mathbb{R}^{N})$, the Schwartz's space of rapid decaying functions defined by
$\mathcal{S}(\mathbb{R}^{N})
\coloneqq \{ g \in C^{\infty} \colon
\sup_{x\in\mathbb{R}^{N}} |x^{\eta} \partial_{\kappa} g(x)|<+\infty \}$ where the supremum is taken over the multi-indices $\kappa, \eta \in \mathbb{N}_{0}^{N}$.

Consider the integral functional
defined by 
\begin{align*}
u \mapsto E(u)
& \coloneqq
\frac{1}{p}\int_{\mathbb{R}^{N}}
\int_{\mathbb{R}^{N}} 
\frac{|u(x)-u(y)|^p}{|x|^{\theta_{1}}|x-y|^{N+sp}|y|^{\theta_{2}}} \dd{x}\dd{y}.
\end{align*}
The G\^{a}teaux derivative of the functional $E$ at $u$ in the direction $\varphi$, also called the first variation of the functional, is computed as
\begin{align*}
\dv{}{\tau}[E(u+\tau \varphi)]\Big|_{\tau=0}
&= \frac{1}{p} 
\int_{\mathbb{R}^{N}}
\int_{\mathbb{R}^{N}}
\dv{}{\tau} \Bigl[
\dfrac{|u(x)-u(y)+\tau (\varphi(x)-\varphi(y))|^p}{|x|^{\theta_{1}}|x-y|^{N+sp}|y|^{\theta_{2}}} \Bigr] \dd{x}\dd{y} \Big|_{\tau=0} \\
&=\int_{\mathbb{R}^{N}}
\int_{\mathbb{R}^{N}}
\dfrac{|u(x)-u(y)|^{p-2}(u(x)-u(y))(\varphi(x)-\varphi(y))}{|x|^{\theta_{1}}|x-y|^{N+sp}|y|^{\theta_{2}}}\dd{x}\dd{y}. 
\end{align*}
This implies that 
$(-\Delta)_{p,\theta}^{s} u$ is the gradient vector field 
$\nabla E(u)$, 
i.e.,\xspace
$\nabla E(u)
=(-\Delta)^{s}_{p,\theta} u
$; hence, 
$(-\Delta)^{s}_{p,\theta} u(x)$
is interpreted as a nonlinear generalization of the usual Laplacian operator.
For a variety of interesting problems, their results and the progress of research on the fractional operator see, e.g.,\xspace
Di~Nezza, Palatucci \& Valdinoci~\cite{MR2944369};
Molica Bisci, R\u{a}dulescu \& Servadei~\cite{MR3445279},
Kwa\'{s}nicki~\cite{MR3613319};
Kuusi \& Palatucci~\cite{MR3824216}; 
del~Teso, Castro-G\'{o}mez \& V\'{a}zquez~\cite{MR4303657},
and Lischke et al.~\cite{lischke}.

\subsubsection*{The Riesz potential} 
On the Euclidean space $\mathbb{R}^{N}$ of dimension $N \geqslant 1$, for $0< \mu < N$ and for each point $x \in \mathbb{R}^{N} \backslash \{0\}$, we set $I_{\mu}(x)=|x|^{-\mu}$. 
The Riesz potential of order $\mu$ of a function $f\in L^1_{\mathrm{loc}}(\mathbb{R}^{N})$ is defined as
\begin{align*}
[I_\mu \ast f](x)
&\coloneqq
\int_{\mathbb{R}^{N}}\frac{f(y)}{\abs{x-y}^{\mu}}dy,
\end{align*}
where the convolution integral is understood in the sense of the Lebesgue integral. 

The usual definition of this potential carries a normalizing constant dependent on $N$ and $\mu$ in front of the integral. This constant is chosen to ensure the semigroup property, $I_{\mu} \ast I_{\nu} = I_{\mu+\nu}$ for $\mu, \nu \in \mathbb{R}_{+}$ such that $\mu+\nu<N$ but it is not considered in this work to simplify the formulation of the results.

The Riesz potential $I_\mu \colon L^{q}(\mathbb{R}^{N} \to L^{r}(\mathbb{R}^{N})$ is well-defined whenever $1 < q < N/(N-\mu)$
and
$1/q - 1/r = (N-\mu)/N$

\subsubsection*{The Choquard equation}
On the Euclidean space $\mathbb{R}^{N}$ of dimension $N \geqslant 1$ and for $x\in \mathbb{R}^{N}$, the equation
\begin{align*}
-\Delta u + V(x) u & = (I_{\mu} \ast |u|^{q}) |u|^{q-2}u
\end{align*}
was introduced by Choquard in the case $N=3$ and $q=2$ to model one-component plasma. It had appeared earlier in the model of the polaron by Fr\"{o}lich and Pekar, where free electrons interact with the polarisation that they create on the medium. When $V(x) \equiv 1$, the groundstate solutions exist if 
$2^{\flat}\coloneqq 2(N-\mu/2))/N < q < 2(N-\mu/2)/(N-2s) \coloneqq 2^{\sharp}$ due to the mountain
pass lemma or the method of the Nehari manifold, 
while there are no nontrivial solution if $q = 2^{\flat}$ or if $q = 2^{\sharp}$ as
a consequence of the Pohozaev identity.
The Choquard equation is also known as the Schr\"{o}dinger-Newton equation in models coupling the Schr\"{o}dinger equation of quantum physics together with Newtonian gravity. This equation is related to several other partial differential equations with nonlocal interactions.

In general, the associated Schr\"{o}dinger-type evolution equation
$i \partial_{t} \psi 
= \Delta \psi
+ \big( I_{\mu} \ast |\psi|^{2} \big) \psi$
is a model for large systems of atoms with an attractive interaction that is weaker and has a longer range than that of the nonlinear Schr\"{o}dinger equation. 
Standing wave solutions of this equation
are solutions to the Choquard equation. 
For more information on the various results related to the non-fractional
Choquard-type equations and their variants see Moroz \& Van Schaftingen~\cite{MR3625092}.

\subsubsection*{The fractional Sobolev space}
In the last years, for pure mathematical research and concrete real-world applications, the fractional $p$-Laplacian operator has been studied on the fractional Sobolev space
$\dot{W}_{\theta}^{s,p} (\mathbb{R}^N)$.
It is the natural fractional counterpart of the homogeneous Sobolev space 
$\mathcal{D}_{0}^{1,p}(\Omega)$, defined as the completion of the space $C_{0}^{\infty}(\mathbb{R}^{N})$ with respect to the norm 
$u \mapsto 
(\int_{\mathbb{R}^{N}} |\nabla u|^{p} \dd{x})^{1/p}$.
Additionally, in the same way that
$\mathcal{D}_{0}^{1,p}(\mathbb{R}^{N})$
is the natural setting for studying variational problems of the type
$\inf \{ (1/p)\int_{\Omega} |\nabla u|^{p}\dd{x}
-\int_{\Omega} f u \dd{x} \}$,
supplemented with Dirichlet boundary conditions (in the absence of regularity assumptions on the boundary $\partial \Omega$), the space 
$\dot{W}^{s,p}_{\theta}(\mathbb{R}^N)$ 
is the natural framework for studying minimization problems containing functionals of the type
\begin{align*}
u & \mapsto 
\dfrac{1}{p}
\int_{\mathbb{R}^{N}}
\int_{\mathbb{R}^{N}}
\dfrac{|u(x)-u(y)|^{p}}{|x|^{\theta_{1}}
|x-y|^{N+sp}|y|^{\theta_{2}}}
\dd{x} \dd{y}
- \int_{\Omega} f u \dd{x},
\end{align*}
in the presence of nonlocal Dirichlet boundary conditions, i.e.,\xspace
the values of $u$ prescribed on the whole complement $\mathbb{R}^{N}\backslash\Omega$, which takes into account long range interactions.

The dual space of 
$\dot{W}^{s,p}_{\theta}(\mathbb{R}^N)$ 
is denoted by $\dot{W}^{s,p}_{\theta}(\mathbb{R}^N)'$
or by $\dot{W}^{s,-p}_{\theta}(\mathbb{R}^N)$.

\subsubsection*{The Stein-Weiss inequality} 

Here we recall a generalization of the Hardy-Littlewood-Sobolev, also called the doubly weighted inequality or the Stein-Weiss inequality.
See, e.g.,\xspace
Stein \& Weiss~\cite{MR0098285};
Lieb \& Loss~\cite[Theorem~4.3]{MR1817225};
Yuan, R\u{a}dulescu, Chen, Wen~\cite[Proposition~1]{MR4564507},
and 
Han, Lu \& Zhu~\cite{MR2921990}.

\begin{proposition}[Doubly weighted Hardy-Littlewood-Sobolev inequality]
\label{prop:2.4chineses:bis}
Let $1 < r, \; t < +\infty$ and $0 < \mu < N$ with 
$1/t+\mu/N+1/r=2$;
let $f \in L^t(\mathbb{R}^N)$ and 
$h \in L^r(\mathbb{R}^N)$. Then there exists a sharp constant $C(N, \mu, r, t)$, independent on $f$ and $h$, such that
\begin{align}
    \label{des:2.5chineses:bis}
\bigg|\int_{\mathbb{R}^N}
      \int_{\mathbb{R}^N} 
\frac{\overline{f(x)}h(y)}{|x-y|^\mu}\dd x \dd y\bigg| 
& \leqslant C(N, \mu, r, t) \|f\|_{L^t(\mathbb{R}^N)} \|h\|_{L^r(\mathbb{R}^N)}.
\end{align}
\end{proposition}

This inequality was introduced by Hardy and Littlewood in $\mathbb{R}^{1}$ and generalized by Sobolev to $\mathbb{R}^{N}$. However, none of them is in its sharp form; namely, neither the sharp constant $C(N, \mu, r, t)$ 
nor the extremal function such that the inequality~\eqref{des:2.5chineses:bis} holds with the sharp constant was known. For a special case when $r=t=2N/(2N-\mu)$, 
Lieb~\cite{MR717827} gave the sharp version of inequality~\eqref{des:2.5chineses:bis}, i.e.,\xspace
the inequality with the best constant 
\begin{align*}
C(t, N, \mu, r) = C(N, \mu)
& = \pi^{\frac{\mu}{2}}\, 
\frac{\Gamma (\frac{n}{2}-\frac{\mu}{2})}{\Gamma (n-\frac{\mu}{2})}
\Bigl\{\frac{\Gamma (\frac{n}{2})}{\Gamma(n)} \Bigr\}^{-1+\frac{\mu}{n}}
\end{align*}
and showed that its extremals, functions for which the inequality~\eqref{des:2.5chineses:bis} is valid with the smallest constant $C(N, \mu, r, t)$, are such that
$f$ is a constant multiple of the function $h$, which must be of the form 
$h(x) = A/(\epsilon^2+|x-a|^2)^{N-\mu/2}$
for some parameters 
$A \in \mathbb{C}$, 
$\epsilon \in \mathbb{R}\backslash\{0\}$,
and $a \in \mathbb{R}^{N}$.
For the general case when $r\neq t$, neither the sharp constant $C(N, \mu, r, t)$ nor the extremals are known yet.

\begin{corollary}
\label{cor:stein-weiss}
Let $0 < s < 1$;
$ 0 \leqslant \alpha < sp + \theta < N$;
$0 < \mu < N$;
given a function
$ u \in \Dot{W}^{s,p}_{\theta}(\mathbb{R}^N)$ 
consider  
$2\delta + \mu \leqslant N$; 
$t=r=N/(N - \delta -\mu/2)$; 
and 
$f(x) = h(x) 
 = |u(x)|^{p^{\sharp}_{s}(\delta,\theta,\mu)}/|x|^{\delta}$.  Then  
$f, h \in L^{\frac{N}{N-\delta-\mu/2}}(\mathbb{R}^{N};|x|^{-\delta})$ and
\begin{align}
    \label{des:2.6chineses:bis}
    \int_{\mathbb{R}^N}
    \int_{\mathbb{R}^N} \frac{|u(x)|^{p^{\sharp}_{s}(\delta,\theta,\mu)}
    |u(y)|^{p^{\sharp}_{s}(\delta,\theta,\mu)}}{|x|^{\delta}|x-y|^{\mu}|y|^{\delta}}\dd x\dd y
    & \leqslant C(N,\delta,\theta,\mu)
    \biggl(\int_{\mathbb{R}^N}
     |u|^{p^*_s(0,\theta)}\dd x\biggr)^{\frac{2(N-\delta-\mu/2)}{N}} 
\end{align}
\end{corollary}

In general, 
the map
\begin{align*}
    u & \mapsto \int_{\mathbb{R}^{N}}
    \int_{\mathbb{R}^{N}}
    \dfrac{|u(x)|^{q}|u(y)|^{q}}{|x|^{\delta}|x-y|^{\mu}|y|^{\delta}} \dd{x} \dd{y}
\end{align*}
is well defined if
\begin{align*}
    p^{\flat}_{s}(\delta,\mu)
    \coloneqq \dfrac{p(N-\delta-\mu/2)}{N} < q < 
    \dfrac{p_{s}^{*}(0,\theta)(N-\delta-\mu/2)}{N}
    \eqqcolon p_{s}^{\sharp}(\delta,\theta,\mu).
\end{align*}

Consider the integral functional
defined by 
\begin{align*}
u \mapsto J(u)
& \coloneqq
\frac{1}{qr}
\int_{\mathbb{R}^{N}}\biggl(
\int_{\mathbb{R}^{N}} \frac{|u(x)|^{q}|u(y)|^r}{|x|^{\delta}|x-y|^{\mu}|y|^{\delta}} \dd{y}\biggr)\dd{x}.
\end{align*}
The G\^{a}teaux derivative of the functional $J$ at $u$ in the direction $\varphi$, also called the first variation of the functional, is computed as
\begin{align*}
\dv{}{\tau}[J(u+\tau \varphi)]\bigg|_{\tau=0}
&= \frac{1}{qr} 
\int_{\mathbb{R}^{N}} \biggl(
\int_{\mathbb{R}^{N}}
\dv{}{\tau} \biggl[
\dfrac{|u(x)+\tau \varphi(x)|^{q}|u(y)+\tau \varphi(y)|^r}{|x|^{\delta}|x-y|^{\mu}|y|^{\delta}} \biggr] \dd{y}\Bigr)\dd{x} \bigg|_{\tau=0} \\
&=\dfrac{1}{q}\int_{\mathbb{R}^{N}} \biggl(
\int_{\mathbb{R}^{N}}
\dfrac{|u(x)|^{q}|u(y)|^{r-2}u(y)\varphi(y)}{|x|^{\delta}|x-y|^{\mu}|y|^{\delta}}\dd{y}\biggr)\dd{x} \\
& \quad + \dfrac{1}{r}\int_{\mathbb{R}^{N}}\biggl(
\int_{\mathbb{R}^{N}}
\dfrac{|u(x)|^{q-2}u(x)\varphi(x)|u(y)|^{r}}{|x|^{\delta}|x-y|^{\mu}|y|^{\delta}}\dd{y}\biggr)\dd{x}.
\end{align*}

The Stein-Weiss inequality provides quantitative information to characterize the integrability for the integral operators present in the energy functional. It is intrinsically determined by their weighted scaling invariance; however, the appearance of the Stein-Weiss convolution integral generates the lack of translation invariance. The study of this inequality have aroused an increasing interest by many authors due to its application in partial differential equations; in particular, 
in the study of the regularity properties of solutions. They are crucial in the analysis developed in this work.

\subsubsection*{The Morrey spaces}
The study of Morrey spaces is motivated by many reasons.
Initially, these
spaces were introduced by Morrey in order to understand the regularity of solutions to elliptic partial
differential equations. 
Regularity theorems, which allow one to conclude higher regularity of a function that is a solution of a differential equation together with a lower regularity of that function, play a central role in the theory of partial differential equations. 
One example of this kind of regularity theorem is a version of the Sobolev embedding theorem which states that 
$W^{j+m,p}(\Omega) \subset C^{j,\lambda}(\overline{\Omega})$ for $0 < \lambda \leqslant m - N/p$, where $j \in \mathbb{N}$ and $\Omega \subset \mathbb{R}^{N}$ is a Lipschitz domain.

Morrey spaces can complement the boundedness properties of operators that Lebesgue spaces can not handle. In line with this, many authors study the boundedness of various
integral operators on Morrey spaces. The theory of Morrey spaces may come in useful when the Sobolev embedding theorem is not readily available. The main results about Morrey spaces are summarized at Appendix~\ref{appendix:properties:morrey}.

\subsubsection*{The variational setting}
The variational structure of 
problem~\eqref{problema0.1} as well as that of
problems~\eqref{problema0.2} and~\eqref{problema0.3} 
can be established with the help of several inequalities.
To ensure the well-definiteness of the energy functional, first we deal with the Hardy potential with a singularity.

\begin{lemma}[Fractional Hardy inequality]
\label{lemma:2.1chineses:bis}
Let $s \in (0,1)$ and $N>sp+\theta$. Then the best fractional Hardy constant $\gamma_{H}$ is attained, where
 \begin{align*}
\gamma_{H}
& \coloneqq 
\inf_{\substack{u \in \dot{W}^{s,p}_{\theta}(\mathbb{R}^n) \\ u \neq 0}}
\dfrac{[u]_{\dot{W}^{s,p}_{\theta}(\mathbb{R}^n)}^{p}}%
{\|u\|_{L^{p}(\mathbb{R}^{n};|x|^{-sp-\theta})}^{p}}.
\end{align*}
\end{lemma}
\begin{proof}
    See Abdellaoui \& Bentifour ~\cite[Lemma $2.7$]{abdellaoui2017caffarelli}; see also Franck \& Seiringer~\cite{MR2469027}.
\end{proof}
 
Next, we use the following versions of the fractional Hardy-Sobolev and Caffarelli-Kohn-Nirenberg inequalities; see Nguyen \& Squassina~\cite[Theorem~1.1]{MR3771839}; see also Abdellaoui \& Bentifour~\cite{MR3626031}.

\begin{lemma}
\label{lemma:2.2chineses:bis}
Let $N \geqslant 1, p \in (1,+\infty), s \in (0,1), 0\leqslant \alpha \leqslant sp+\theta<N, \theta, \theta_1,\theta_2,\beta \in \mathbb{R}$ be such that $\theta = \theta _1 + \theta _2$. If $1/p^*_s(\alpha,\theta) - \alpha / Np^*_s(\alpha,\theta)>0$, then there exists a positive constant $C(N,\alpha,\theta)$ such that
\begin{align}
    \label{des:2.2chineses:bis}
\biggl(\int_{\mathbb{R}^N}\frac{|u|^{p^*_s(\alpha,\theta)}}{|x|^{\alpha}}\dd x\biggr)^{\frac{p}{p^*_s(\alpha,\theta)}} \leqslant C(N,\alpha,\theta) \int_{\mathbb{R}^N} \int_{\mathbb{R}^N} \frac{|u(x)-u(y)|^p}{|x|^{\theta _1}|x-y|^{N+sp}|y|^{\theta _2}} \dd x \dd y
\end{align}
for all $u \in \Dot{W}^{s,p}_{\theta}(\mathbb{R}^N)$.
\end{lemma}

Using inequality~\eqref{des:2.6chineses:bis}
from Lemma~\ref{cor:stein-weiss}
together with H\"{o}lder's inequality and Lemma~\ref{lemma:2.2chineses:bis} we can deduce another useful inequality.

\begin{corollary}
Under the hypotheses of Lemma~\ref{lemma:2.2chineses:bis} we have
\begin{align}
    \int_{\mathbb{R}^N}
    \int_{\mathbb{R}^N} \frac{|u(x)|^{p^{\sharp}_{s}(\delta,\theta,\mu)}
    |u(y)|^{p^{\sharp}_{s}(\delta,\theta,\mu)}}{|x|^{\delta}|x-y|^{\mu}|y|^{\delta}}\dd x\dd y
     & \leqslant 
    C(N,\delta,\mu)
    \|u\|^{2 p^{\sharp}_{s}(\delta,\theta,\mu)}_{\Dot{W}^{s,p}_{\theta}(\mathbb{R}^N)}.
\end{align}
\end{corollary}

Based on the embeddings~\eqref{imersoes} 
we establish the following improved weighted fractional Caffarelli-Kohn-Nirenberg inequality whose proof is in section~\ref{proof:prop:1.3chineses:bis}; see also~\cite{MR3216834}.

\begin{lemma}
[Fractional Caffarelli-Kohn-Nirenberg inequality]
\label{prop:1.3chineses:bis}
    Let $s\in (0,1)$ and $0< \beta < sp +\theta< N$. Then there exists $C = C(N, s, \beta)>0$ such that for any $\zeta \in (\Bar{\zeta}, 1)$ and for any $q\in [1, p^*_s(\beta,\theta))$, it holds
    \begin{align}
    \label{des:1.10chineses:bis}   \biggl(\int_{\mathbb{R}^N}\frac{|u(y)|^{p^*_s(\beta,\theta)}}{|y|^{\beta}}\dd y\bigg)^{\frac{1}{p^*_s(\beta,\theta)}} \leqslant \|u\|_{\Dot{W}^{s,p}_{\theta}}^{\zeta} \|u\|^{1-\zeta}_{L_M^{q,\frac{N-sp-\theta}{p}q+qr}(\mathbb{R}^N, |y|^{-qr})}
    \end{align}
for all $u \in \Dot{W}^{s,p}_{\theta}(\mathbb{R}^N)$,
    where $\Bar{\zeta} = \max 
    \{p/p^*_s(\beta,\theta), (p^*_s(0,\theta)-1)/p^*_s(\beta,\theta)\}>0$ and $r=\beta/p^*_s(\beta,\theta)$.
    \end{lemma}

\subsubsection*{Related works on fractional elliptic operators}
Problems with one or two nonlinearities involving the $p$-Laplacian and the fractional $p$-Laplacian have been studied by many authors. 
Filippucci, Pucci \& Robert~\cite{filippucci2009p} proved that there exists a positive solution for a $p$-Laplacian problem with
critical Sobolev and Hardy–Sobolev terms,
i.e.,\xspace
problem~\eqref{problema0.2} with $s=1$, $p=2$, $\theta_1 = \theta_2 = 0$, 
$\beta_1 = 0$ and no singularity in the Hardy potential. 
As is well known, to show existence results it is natural to consider variants of Lions's concentration–compactness principle for critical problems. However, due to the nonlocal feature of the fractional $p$-Laplacian, it is difficult to use the concentration–compactness principle directly, since one needs to estimate commutators of the fractional
Laplacian and smooth test functions. A possible strategy, which is known as $s$-harmonic extension, is to transform the nonlocal problem in $\mathbb{R}^{N}$ into a local problem in $\mathbb{R}^{N+1}_{+}$ with Neumann boundary condition, as performed by Caffarelli \& Silvestre~\cite{MR2354493}. Since that, many interesting results in the classical elliptic problems have
been extended to the setting of the fractional Laplacian. For example, Ghoussoub \& Shakerian~\cite{MR3366777} considered problem~\eqref{problema0.2} with 
$p=2$, $\theta_1 = \theta_2 = 0$, 
$\beta_1 = \beta_{2} = 0$ and no singularity in the Hardy potential; 
Chen~\cite{MR3762809} also studied problem~\eqref{problema0.2} and extended this result to the case $p=2$, $\theta_1 = \theta_2 = 0$ but with $\beta_{1} \neq 0$ and $\beta_{2} \neq 0$. In both papers, the authors combined the $s$-harmonic extension with the concentration-compactness principle to investigate the existence of solutions for a doubly critical problem involving the fractional Laplacian. 
Assun\c{c}\~{a}o, Miyagaki \& Silva~\cite{assunccao2020fractional} 
considered problem~\eqref{problema0.2} with no singularity in the Hardy potential, that is, $\theta_{1}=\theta_{2}=0$ and $\beta_1 \neq 0$ and $\beta_{2}\neq 0$.
Li \& Yang~\cite{Li_2020} studied problem~\eqref{problema0.1} involving the fractional Laplacian with a Hardy potential and two nonlinearities, one of Sobolev type and the other of Choquard type. More precisely, they considered problem~\eqref{problema0.1} with $p=2$, $\theta_1 = \theta_2 = 0$. The proof of the  existence result is achieved in the setting of Morrey spaces to avoid the use of the concentration-compactness principle. They also studied problems~\eqref{problema0.2} and~\eqref{problema0.3} in the case $p=2$, $\theta_1 = \theta_2 = 0$ and the proof follows basically the same steps.
They claim to have considered also the cases $\alpha = 0$ or $\beta = 0$; however, their proof is based on a flawed argument;
see Remark~\ref{rem:li_yang}. Recently, Su~\cite{MR4554065} considered the general $p$ and $\theta_{1}=\theta_{2}=0$ and proved existence, decaying and regularity results for problem~problem~\eqref{problema0.1} with a general condition $0 < sp < N$ and $\theta_1 = \theta_2 = 0$.

\subsubsection*{Our contribution to the problem}
Inspired by the previously mentioned papers, we mainly extend the results by Li \& Yang~\cite{Li_2020}. We consider the general fractional $p$-Laplacian with 
$p > 1$ and
$\theta = \theta_1 + \theta_2 $ not necessarily zero. 
By establishing new embedding results involving weighted Morrey norms in the homogeneous fractional Sobolev space, we provide sufficient conditions under which a weak nontrivial solution to the problem exists via variational methods.

\subsection{Some of the difficulties to prove the theorems}
In the process of proving 
Theorem~\ref{teo:1.1chineses}, there are several technical and substantial difficulties. 

\subsubsection*{The nonlocality of the operator}
First, we mention that 
the procedure based upon the Caffarelli \& Silvestre approach through $s$-harmonic extension 
can overcome the difficulty of the nonlocality of the operator only in the case $p=2$ for the fractional Laplacian operator $(-\Delta)^{s}$.

\subsubsection*{Estimatives of mixed terms}
Second, the truncation technique adopted by Filipucci, Pucci \& Robert~\cite{filippucci2009p} is not suitable when we work in the homogeneous Sobolev space 
$\dot{W}_{\theta}^{s,p}(\mathbb{R}^{N})$
to consider the nonlocal operator $(-\Delta)_{p,\theta}^{s}$. 

\subsubsection*{The structure of Palais-Smale sequences}
Third, since we consider problems with critical nonlinearities in the entire space $\mathbb{R}^{N}$, the compactness of the corresponding Palais-Smale sequences can not hold for any energy level $c>0$ since the problem is invariant under the scaling
$u(x) \mapsto \lambda^{(N-sp-\theta)/p}u(\lambda x)$.

\subsubsection*{The auxiliary minimization problems}
Fourth, we have to show 
that the best constants in two auxiliary minimization problems
are attained; this is a crucial step in our work. More precisely, we consider a minimization problem involving the Choquard convolution integral
\begin{align}
    \label{probminimizacao1:yang:1.16}
    S_\mu (N, s, p, \theta,\gamma, \alpha)= \inf\limits_{u \in \Dot{W}^{s,p}_{\theta}(\mathbb{R}^N)\setminus \{0\}} \frac{\|u\|^p}{Q^{\sharp} (u,u)^{\frac{p}{2p^{\sharp}_{\mu}(\alpha, \theta)}}}
\end{align}
where the quadratic form 
$Q^{\sharp} \colon 
\Dot{W}^{s,p}_{\theta}(\mathbb{R}^N) \times\Dot{W}^{s,p}_{\theta}(\mathbb{R}^N)
\to \mathbb{R}$
is defined by
\begin{align}
\label{def:qsharp}
    Q^{\sharp}(u,v)=\int_{\mathbb{R}^N}\int_{\mathbb{R}^N} \frac{|u(x)|^{p^\sharp_s(\delta,\theta,\mu) }|v(y)|^{p^\sharp_s(\delta,\theta,\mu) }}{|x|^{\delta}|x-y|^{\mu}|y|^{\delta}}\dd x\dd y.
\end{align}
We also consider another minimization problem involving the Sobolev term,
\begin{align}    
\label{probminimizacao2:yang:1.17}
    \Lambda (N,s,p,\theta, \gamma, \beta) = \inf\limits_{u \in \Dot{W}^{s,p}_{\theta}(\mathbb{R}^N)\setminus \{0\}} \frac{\|u\|^p}{\|u\|^p_{L_s^{p^*_s(\beta,\theta)}(\mathbb{R}^N, |x|^{-\beta})}}.
\end{align}
The best constants of these minimization problems are attained as claimed in the following proposition, whose proof is in section~\ref{proof:prop:4.1chineses}.
\begin{proposition}
\label{prop:4.1chineses}
    For $s \in (0,1)$ the best constants $S_\mu (N, s, \gamma, \alpha)$ and $\Lambda (N, s, \gamma, \beta)$ verify  the following items.
    \begin{enumerate}
        \item If $0<\alpha<sp+\theta<N, \mu \in (0,N)$ and $\gamma < \gamma _H$, then $S_\mu (N, s, \gamma, \alpha)$ is attained in $\Dot{W}^{s,p}_{\theta}(\mathbb{R}^N)$;
        \item If $0<\beta<sp+\theta<N$ and $\gamma < \gamma _H$, then $\Lambda (N, s, \gamma, \beta)$ is attained in $\Dot{W}^{s,p}_{\theta}(\mathbb{R}^N)$;
          \item If $N>sp+\theta, \mu \in (0,N)$ and $0\leqslant \gamma < \gamma _H$, then $S_\mu (N, s, \gamma, 0)$ is attained in $\Dot{W}^{s,p}_{\theta}(\mathbb{R}^N)$;
           \item If $N>sp+\theta$ and $0\leqslant \gamma < \gamma _H$, then $\Lambda (N, s, \gamma, 0)$ is attained in $\Dot{W}^{s,p}_{\theta}.(\mathbb{R}^N)$
    \end{enumerate}
\end{proposition}

To simplify the notation, hereafter we simply denote $S_{\mu}=S_{\mu}(N, s, p, \theta, \gamma, \alpha)$ and
$\Lambda = \Lambda (N, s, p, \theta, \gamma, \beta)$. To show that $S_{\mu}$ and $\Lambda$ are attained we have to proof a version of the Caffarelli-Kohn-Nirenberg's inequality, which estimates the norm of a function in the critical weighted Sobolev space and the norm of the same function in the fractional Sobolev and Morrey spaces; see Lemma~\ref{prop:1.3chineses:bis}. In our setting, that is, $1<p<+\infty$, we have to use a version of the Caffarelli-Silvestre extension as given by del~Teso, Castro-G\'{o}mez \& V\'{a}zquez~\cite[Theorem~3.1]{MR4303657}. 

Additionaly, we have to consider
the fractional Hardy type potential, which is related to the best constant in the fractional Hardy inequality.

\subsubsection*{The asymptotic competition}
Finally, as it was
already mentioned by 
Filippucci, Pucci 
\& Robert~\cite{filippucci2009p} and in several other papers, we observe here the main difficulty is that there is an asymptotic
competition between the energy carried by two critical nonlinearities. If one dominates the
other, then there is vanishing of the weakest one and we obtain solution of an equation with
only one critical nonlinearity. Therefore the crucial step in the proof is to avoid the dominance of one term over the other. To overcome this difficulty, we choose the Palais-Smale sequence at suitable energy level and make a careful analysis of the concentration; afterwards,
we show that there is a balance between the energies of the two nonlinearities mentioned
above, and therefore none can dominate the other. Moreover, we can make the full use of
conformal invariance of problem~\eqref{problema0.1} under the above defined scaling and we recover the solution to the problem in the critical case.

\subsection{Method of proof and outline of the work}
The method adopted in previous works, such as Filippucci, Pucci \& Robert~\cite{filippucci2009p},
Yang \& Wu~\cite{yang2017doubly},
is not applicable to problem~\eqref{problema0.1}.
For this reason, we develop a new tool which is based on the weighted Morrey space. 
To be more precise,  we discover the embeddings~\eqref{imersoes}.

Now we give an outline of the proof of Theorem~\ref{teo:1.1chineses}. 
We already know that weak solutions to problem~\eqref{problema0.1} correspond to critical points of the energy functional $I$ defined on the homogeneous Sobolev space 
$\dot{W}_{\theta}^{s,p}(\mathbb{R}^{n})$. 
Moreover, this functional  
has the appropriate geometry to use the mountain pass theorem; see Ambrosetti \& Rabinowitz~\cite{ambrosetti},
and Willem~\cite{MR1400007}.
However, since in our problem we consider doubly critical nonlinearities, this theorem does not yield critical points but only Palais-Smale sequences. Thus, we require the mountain pass level of the Palais-Smale sequences $(PS)_{c}$ to verify the condition $c<c^{*}$ for some suitable threshold level $c^{*}$. This is crucial in ruling out the vanishing of the sequence. 

After showing that the minimizers of 
$S_\mu$ and 
$\Lambda$ are attained, 
we can prove that the mountain pass level verifies the required inequality
$c< c^{*}$, where 
\begin{align}
    \label{def:cestrela:bis}
    c^* 
    & \coloneqq \min \Biggl\{
\Bigl(\frac{1}{p}-\frac{1}{2p^{\sharp}_{s}(\delta,\theta,\mu)}\Bigr)
    S_{\mu}^{\frac{2p^{\sharp}_{s}(\delta,\theta,\mu)}{2p^{\sharp}_{s}(\delta,\theta,\mu)-p}}, \Bigl(\frac{1}{p}-\frac{1}{p^{\ast}_{s}(\beta,\theta)}\Bigr)
    \Lambda^{\frac{p^{\ast}_{s}(\beta,\theta)}{p^{\ast}_{s}(\beta,\theta)-p}}\Biggr\}.
\end{align}
Moreover, the $(PS)_{c}$ sequence 
$\{u_{k}\}_{k\in\mathbb N} \subset 
\dot{W}^{s,p}_{\theta}(\mathbb{R}^N)$
verifies the conditions
\begin{alignat}{2}
\label{lim:5.2chineses:bis}
    \lim_{k \to +\infty} I(u_{k})&=c<c^{*},
    &\qquad 
    \lim_{k \to +\infty} I'(u_{k}) &= 0 
    \textup{ strongly in } 
    \dot{W}^{s,p}_{\theta}(\mathbb{R}^N).
\end{alignat}
This sequence is bounded; so, up to the passage to a subsequence we may assume that $u_{k} \rightharpoonup u$ weakly in 
$ \dot{W}^{s,p}_{\theta}(\mathbb{R}^N)$ for some 
$u \in \dot{W}^{s,p}_{\theta}(\mathbb{R}^N)$.
But it may occur that $u \equiv 0$. 

To show that this does not occur, denote
\begin{alignat*}{2}
    d_{1} & \coloneqq
    \lim_{k \to +\infty}
\int_{\mathbb{R}^N}\frac{|u_k|^{p^*_s(\beta,\theta)}}{|x|^{\beta}}\dd x,
& \qquad 
d_{2} & \coloneqq \lim_{k \to +\infty}
\int_{\mathbb{R}^{N}}
\int_{\mathbb{R}^{N}}
\dfrac{|u_{k}(x)|^{p_{s}^{\sharp}(\delta,\theta,\mu)}
       |u_{k}(y)|^{p_{\mu}^{\sharp}(\delta,\theta,\mu)}}%
       {|x|^{\delta}
        |x-y|^{\mu}
        |y|^{\delta}}
        \dd{x}\dd{y}.
\end{alignat*}
By the definitions of 
$S_{\mu}$
and 
$\Lambda$
together with the properties of the Palais-Smale sequence, we can prove that
\begin{alignat*}{2}
    d_1^{\frac{p}{p^*_s(\beta,\theta)}} \Bigl(\Lambda  -d_1^{\frac{p^*_s(\beta,\theta) -p}{p^*_s(\beta,\theta)}}\Bigr) &\leqslant d_2,
    & \qquad 
    d_2^{\frac{1}{p^{\sharp}_{s}(\delta,\theta,\mu)}}
    \Bigl(S_{\mu}-d_2^{\frac{p^{\sharp}_{s}(\delta,\theta,\mu) -1}{p^{\sharp}_{s}(\delta,\theta,\mu)}}\Bigr) &\leqslant d_1.
\end{alignat*}
And since $0<c<c^{*}$, we can also deduce that 
\begin{alignat*}{2}
    \Lambda  -d_1^{\frac{p^*_s(\beta,\theta) -p}{p^*_s(\beta,\theta)}} 
    & >0,
    & \qquad 
    S_{\mu}-d_2^{\frac{p^{\sharp}_{s}(\delta,\theta,\mu) -1}{p^{\sharp}_{s}(\delta,\theta,\mu)}}
    & >0.
\end{alignat*}
Thus, 
$d_{1} \geqslant \varepsilon_{0} > 0$
and
$d_{2} \geqslant \varepsilon_{0} > 0$ for some $\varepsilon_{0} \in \mathbb{R}_{+}$;
indeed, if $d_{1}=0$ and $d_{2}=0$, then $c=0$, which is a contradiction. 

Using the embeddings~\eqref{imersoes} and the improved Sobolev inequality~\eqref{des:1.10chineses:bis}, we deduce that, for $k \in \mathbb{N}$ large enough,
\begin{align*}
    0<C_2\leqslant \|u_k\|_{L_{M}^{p,N-sp-\theta+pr}(\mathbb{R}^N,|y|^{-pr})}\leqslant C_1,
\end{align*}
where $r=\alpha/p^*_s(\alpha, \theta)$.
For $k \in \mathbb{N}$ large enough, we may find sequences 
$\{\lambda_{k}\}_{k \in \mathbb{N}} \subset \mathbb{R}_{+}$
and $\{x_{k}\}_{k \in \mathbb{N}} \in \mathbb{R}^{N}$ 
such that 
\begin{align*}
    \lambda_k^{(N-sp-\theta+pr)-N}\int_{B_{\lambda _k}(x_k)}\frac{|u_k(y)|^p}{|y|^{pr}} \dd y > \|u_k\|^p_{L_{M}^{p,N-sp-\theta+pr}(\mathbb{R}^N,|y|^{-pr})} -\frac{C}{2k} \geqslant \Tilde{C}>0.
\end{align*}
And with the help of these two sequences we can define another sequence 
$\{v_{k}\}_{k\in\mathbb N} \subset 
\dot{W}^{s,p}_{\theta}(\mathbb{R}^N)$
defined by the scaling
\begin{align*}
v_k(x)=\lambda_k^{(N-sp-\theta)/p}u_k(\lambda _k x).
\end{align*}
This new sequence verifies the condition $\|v_k\|=\|u_k\|\leqslant C$
and up to the passage to a subsequence we may assume that $v_{k} \rightharpoonup v$ weakly in 
$ \dot{W}^{s,p}_{\theta}(\mathbb{R}^N)$ for some 
$v \in \dot{W}^{s,p}_{\theta}(\mathbb{R}^N)$
and
$v_k \to v $ a.e.\xspace in $\mathbb{R}^N$,
up to the passage to a subsequence, as $k\to+\infty$.
Again, it may occur that $v \equiv 0$; however, the sequence 
$\{v_{k}\}_{k\in\mathbb{N}}$ is of a very structured form and we can prove that $v\not\equiv 0$.

To do this, we consider the sequence 
$\{\Tilde{x}_k\}_{k \in \mathbb{N}} \subset \mathbb{R}^{N}$
defined by $\Tilde{x}_{k}=x_{k}/\lambda_{k}$
and show that it is bounded; then, we can find 
$R \in \mathbb{R}^{N}_{+}$ 
such that the ball $B(0,R)$ contains all unitary balls centered in $\Tilde{x}_k$  for $k\in\mathbb{N}$; moreover,
$ \int_{B_R(0)} \frac{|v_k(x)|^p}{|x|^{pr+\theta}}\dd x \geqslant C_1>0$.
Additionaly, we can show that 
$|x|^{-r-\frac{\theta r}{sp}}u_k \to 
|x|^{-r-\frac{\theta r}{sp}}u|x|^{r+\frac{\theta r}{sp}} $ 
in $L^p_{\operatorname{loc}}(\mathbb{R}^N)$; 
therefore, 
$ \int_{B_R(0)} \frac{|v(x)|^p}{|x|^{pr+\theta}}\dd x \geqslant C_1>0$,
and we deduce that $v\not\equiv 0$. 

Again, the boundedness of the sequence $\{v_k\}_{k\in\mathbb{N}}$ 
in $\Dot{W}^{s,p}_{\theta}(\mathbb{R}^N)$ implies the boundedness of the sequence 
$\{|v_k|^{p^*_s(\beta, \theta)-2}v_k\}_{k\in\mathbb{N}}$ in $L^{\frac{p^*_s(\beta, \theta)}{p^*_s(\beta, \theta) -1}}(\mathbb{R}^N, |x|^{-\beta})$,
and this implies that
$|v_k|^{p^*_s(\beta, \theta)-2}v_k\ \rightharpoonup |v|^{p^*_s(\beta, \theta)-2}v$
weakly in 
$ L^{\frac{p^*_s(\beta, \theta)}{p^*_s(\beta, \theta) -1}}(\mathbb{R}^N, |x|^{-\beta})$.
For any $\phi \in \Dot{W}^{s,p}_{\theta}(\mathbb{R}^N)$, 
with the help of 
the embedding
$\Dot{W}^{s,p}_{\theta}(\mathbb{R}^N) \hookrightarrow L^{p^*_s(\alpha, \theta)}(\mathbb{R}^N,|x|^{-\alpha})$
and a variant of the Brézis-Lieb lemma (see~\cite{brezis1983relation}), we can show that 
\begin{align*}
\lim\limits_{k \to \infty} \int_{\mathbb{R}^N} \left[I_{\mu}\ast F_{\alpha}(\cdot , v_k)\right](x) f_{\alpha}(x,v_k)\phi (x) \dd x 
= \int _{\mathbb{R}^N} \left[I_{\mu}\ast F_{\alpha}(\cdot , v)\right](x) f_{\alpha}(x,v)\phi (x) \dd x. 
\end{align*}

We still need to check that the sequence 
$\{v_k\}_{k \in \mathbb{N}} \subset 
\Dot{W}^{s,p}_{\theta}(\mathbb{R}^N) \hookrightarrow L^{p^*_s(\alpha, \theta)}(\mathbb{R}^N,|x|^{-\alpha})$
is also a $(PS)_{c}$ sequence for
the energy functional $I$ at the level $c< c^{\ast}$. To do this, we notice that the norms in $\Dot{W}^{s,p}_{\theta}(\mathbb{R}^N)$ and $L^{p^*_s(\alpha, \theta)}(\mathbb{R}^N,|x|^{-\alpha})$ are invariant under the special scaling used,
$ \|v_k\|^p_{\Dot{W}^{s,p}_{\theta}(\mathbb{R}^N)} =    \|u_k\|^p_{\Dot{W}^{s,p}_{\theta}(\mathbb{R}^N)}$
and
$\|v_k\|^{p^*_s(\alpha, \theta)}_{L^{p^*_s(\alpha, \theta)}} = \|u_k\|^{p^*_s(\alpha, \theta)}_{L^{p^*_s(\alpha, \theta)}}$.
Thus, we have 
$\lim_{k \to + \infty} I(v_k) =c$.
Moreover, for all 
$\phi \in \Dot{W}^{s,p}_{\theta}(\mathbb{R}^N)$, we also have $\phi_k (x) = \lambda_k^{(N-sp-\theta)/p}\phi \left(x/\lambda _k\right) \in \Dot{W}^{s,p}_{\theta}(\mathbb{R}^N)$. 
And from the strong convergence 
$I'(u_k) \to 0$ in $\Dot{W}^{s,p}_{\theta}(\mathbb{R}^N)'$, we can deduce that 
$\langle I'(v), \phi\rangle =\lim_{k \to + \infty} \langle I'(v_k), \phi\rangle 
= \lim_{k \to + \infty} \langle I'(u_k), \phi\rangle =0$.
Hence, $v$ is a nontrivial weak solution of~\ref{problema0.1}.

To conclude the proof, it remains to show the crucial step that the quantities $S_{\mu}$ and $\Lambda$, defined in~\eqref{probminimizacao1:yang:1.16} 
and~\eqref{probminimizacao2:yang:1.17}, respectively, are attained. To this end, we need some kind of compactness.
These problems
can be solved in a direct way using the embeddings~\eqref{imersoes}
and the improved Sobolev inequality~\eqref{des:1.10chineses:bis}.

\section{Preliminary results}
In this section, we give some preliminary results that will be usefull in the proof of Theorem~\ref{teo:1.1chineses}.
We begin by stating a result about local convergence.

\begin{lemma}
\label{lemma:2.3chineses}
Let $s \in (0,1)$ and $0<r<s+ \frac{\theta}{p}<\frac{N}{p}$. If $\{u_k\}_{k\in\mathbb{N}}$ is a bounded sequence in $\Dot{W}^{s,p}_{\theta}(\mathbb{R}^N)$ and $u_k \rightharpoonup u$ in $\Dot{W}^{s,p}_{\theta}(\mathbb{R}^N)$, then as $k\to+\infty$,
\begin{align*}
    \frac{u_{k}}{|x|^{r+\frac{\theta r}{sp}}} &\to \frac{u}{|x|^{r+\frac{\theta r}{sp}}} \; \text{in $L^p_{\operatorname{loc}}(\mathbb{R}^{N})$}.
\end{align*}
\end{lemma}

\begin{proof}
    Since $u_k \rightharpoonup u$ in $\Dot{W}^{s,p}_{\theta}(\mathbb{R}^N)$, we have
    \begin{align*}
        u_k \to u \; \text{in $L^q_{\operatorname{loc}}(\mathbb{R}^N) \;(1\leqslant q\leqslant p^*_s(0,\theta))$\quad and \quad $u_k \to u$ a.e. on $\mathbb{R}^N$.}
    \end{align*}
    From Lemma~\ref{lemma:2.1chineses:bis}, we have
    \begin{align*}
          \int_{\mathbb{R}^N} \frac{|u_k|^p}{|x|^{sp+\theta}} \dd x \leqslant C_{s,N} \int_{\mathbb{R}^N} \int_{\mathbb{R}^N} \frac{|v_k(x)-v_k(y)|^{p}}{|x|^{\theta _1}|x-y|^{N+sp}|y|^{\theta _2}}\dd x \dd y.
    \end{align*}
    For any compact set $\Omega \Subset \mathbb{R}^N$, using Hölder's inequality we have
    \begin{align*}
        \int_{\Omega}\frac{|u_k-u|^p}{|x|^{pr+\frac{\theta r}{s}}}\dd x  \leqslant \biggl(\int_{\Omega} \frac{|u_k-u|^p}{|x|^{sp +\theta}}\dd x\biggr)^{\frac{r}{s}} \bigg(\int_{\Omega}|u_k-u|^p \dd x\biggr)^{\left(1-\frac{r}{s}\right)}  \leqslant C \biggl(\int_{\Omega}|u_k-u|^p \dd x\biggr)^{\left(1-\frac{r}{s}\right)}.
    \end{align*}
This means that, as $k\to+\infty$,
\begin{align*}
    \frac{u_k}{|x|^{r+\frac{\theta r}{sp}}} \to \frac{u}{|x|^{r+\frac{\theta r}{sp}}} \; \text{in $L^p_{\operatorname{loc}}(\mathbb{R}^N)$}.
\end{align*}
This concludes the proof of the lemma.
\end{proof}

Next, we state a variant of the  classic Brézis--Lieb lemma that will be useful to prove a similar result for the convolution terms.

\begin{lemma}[A variant of Brézis--Lieb lemma]
\label{lemma:2.5chineses}
Let $r>1, q \in [1,r]$ and $\delta \in [0,Nq/r)$. Assume that $\{w_k\}$ is a bounded sequence in $L^r(\mathbb{R}^N, |x|^{-\delta r/q})$ and $w_k \to w$ a.e. on $\mathbb{R}^N$. Then, 
\begin{align}
    \label{eq:lemma2.5a}
    \lim\limits_{k \to \infty} \int_{\mathbb{R}^N}\bigg|\frac{|w_k|^q}{|x|^\delta}-\frac{|w_k-w|^q}{|x|^\delta}-\frac{|w|^q}{|x|^\delta}\bigg|^{\frac{r}{q}}\dd{x} =0
\end{align}
and 
\begin{align}
    \label{eq:lemma2.5b}
    \lim\limits_{k \to \infty} \int_{\mathbb{R}^N}\bigg|\frac{|w_k|^{q-1}w_k}{|x|^\delta}-\frac{|w_k-w|^{q-1}(w_k-w)}{|x|^\delta}-\frac{|w|^{q-1}w}{|x|^\delta}\bigg|^{\frac{r}{q}}\dd{x}=0
\end{align}
\end{lemma}

\begin{proof}
    For the case  $\delta =0$, one can refer to  in~\cite[Lemma 2.3]{singh2017nonlocal}; here we focus on the case $\delta >0$. Fixing $\epsilon > 0$ small, there exists $C(\epsilon)>0$ such that for all $a,b \in \mathbb{R}$ and $q\geqslant 1$, we have
    \begin{align*}
        \left||a+b|^q-|a|^q\right| \leqslant \epsilon |a|^q+C(\epsilon) |b|^q.
    \end{align*}
    Recalling that $(a+b)^p\leqslant 2^{p-1}(a^p+b^p)$ for $a,b\geqslant 0$ and $p \geqslant 1$ and using the previous inequality, we obtain
    \begin{align}
        \label{des:2.7chineses}
        \left||a+b|^q-|a|^q\right|^{\frac{r}{q}} \leqslant \left(\epsilon |a|^q+C(\epsilon) |b|^q\right)^{\frac{r}{q}} \leqslant \Tilde{\epsilon}|a|^r+\Tilde{C}(\epsilon)|b|^r,
    \end{align}

Taking $a=(w_k-w)/|x|^{\delta/q}$
and
$b=w/|x|^{\delta/q}$ in inequality~\eqref{des:2.7chineses}, we obtain
\begin{align*}
|f_{N,\epsilon}| &\coloneqq\biggl(\left|\frac{|w_k|^q}{|x|^{\delta}}- \frac{|w_k-w|^q}{|x|^\delta}-\frac{|w|^q}{|x|^{\delta}}\right|^{\frac{r}{q}} - \Tilde{\epsilon}\biggl(\frac{|w_k-w|}{|x|^{\frac{\delta}{q}}}\biggr)^r\biggr)^+ \\
         & \leqslant \left|\frac{|w_k|^q}{|x|^{\delta}}- \frac{|w_k-w|^q}{|x|^\delta}\right|^{\frac{r}{q}} + \left|\frac{w}{|x|^{\delta/q}}\right|^r - \Tilde{\epsilon}\biggl(\frac{|w_k-w|}{|x|^{\frac{\delta}{q}}}\biggr)^r\\
         & \leqslant \Tilde{\epsilon}\biggl(\frac{|w_k-w|}{|x|^{\frac{\delta}{q}}}\biggr)^r + \Tilde{C}(\epsilon)\left|\frac{w}{|x|^{\delta/q}}\right|^r + \left|\frac{w}{|x|^{\delta/q}}\right|^r - \Tilde{\epsilon}\biggl(\frac{|w_k-w|}{|x|^{\frac{\delta}{q}}}\biggr)^r\\
         & \leqslant \bigg|\frac{w}{|x|^{\delta/q}}\bigg|^r + \Tilde{C}(\epsilon)\bigg|\frac{w}{|x|^{\delta/q}}\bigg|^r
         \\& =\left(1+\Tilde{C}(\epsilon)\right)\bigg|\frac{w}{|x|^{\delta/q}}\bigg|^r.
    \end{align*}
    Now using Lebesgue Dominated Convergence theorem, we have
    \begin{align*}
        \int_{\mathbb{R}^N}|f_{N,\epsilon}| \dd x \to 0 \quad\text{as $k \to \infty$}.
    \end{align*}
    Therefore, we get
    \begin{align*}
        \left|\frac{|w_k|^q}{|x|^{\delta}}- \frac{|w_k-w|^q}{|x|^\delta}-\frac{|w|^q}{|x|^{\delta}}\right|^{\frac{r}{q}}  \leqslant  |f_{N,\epsilon}| +\Tilde{\epsilon}\biggl(\frac{|w_k-w|}{|x|^{\frac{\delta}{q}}}\biggr)^r,
    \end{align*}
    which gives
    \begin{align*}
        \limsup\limits_{k \to \infty} \int_{\mathbb{R}^{N}}\left|\frac{|w_k|^q}{|x|^{\delta}}- \frac{|w_k-w|^q}{|x|^\delta}-\frac{|w|^q}{|x|^{\delta}}\right|^{\frac{r}{q}} \dd{x} \leqslant \Tilde{\epsilon}
        \sup\limits_{k\in \mathbb{N}} \int_{\mathbb{R}^{N}}\frac{|w_k-w|^{r}}{|x|^{\frac{\delta r}{q}}}\dd{x}<\infty.
    \end{align*}
Further, letting $\epsilon \to 0$ we conclude~\eqref{eq:lemma2.5a}. 

The limit~\eqref{eq:lemma2.5b} can be proved in the same way.
In fact, fixing $\epsilon > 0$ small, there exists $C(\epsilon)>0$ such that for all $a,b \in \mathbb{R}$ and $q\geqslant 1$, we have 
    \begin{align*}
        \left||a+b|^{q-1}(a+b)-|a|^{q-1}a\right| \leqslant \epsilon |a|^q +C(\epsilon)|b|^q.
    \end{align*}
    Using the previous inequality, we obtain
    \begin{align}
        \label{des:2.8chineses}
        \left||a+b|^{q-1}(a+b)-|a|^{q-1}a\right|^{\frac{r}{q}} \leqslant \left(\epsilon |a|^q +C(\epsilon)|b|^q\right)^{\frac{r}{q}} \leqslant \Tilde{\epsilon}|a|^r+\Tilde{C}(\epsilon)|b|^r,
    \end{align}
    where $\Tilde{\epsilon}=2^{\frac{r}{q}-1}\epsilon^{\frac{r}{q}}$ and $\Tilde{C}(\epsilon)=2^{\frac{r}{q}-1}C(\epsilon)^{\frac{r}{q}}$.
Now we can adapt the same arguments already used to conclude~\eqref{eq:lemma2.5b}
\end{proof}

Also recall that pointwise convergence of a bounded sequence implies weak converge.

\begin{lemma}
\label{lemma:2.6moroz}
Let $\Omega  \subset \mathbb{R}^N$ be a domain, $q \in (1, \infty)$ and $\{u_k\}_{k \in \mathbb{N}}$ be a bounded sequence in $L^
q(\Omega)$. If $u_k \to u$ almost everywhere on $\Omega$ as $k \to \infty$, then $u_k \rightharpoonup u$ weakly in $L^q(\Omega)$.
\end{lemma}
\begin{proof}
See Willem~\cite[Proposition~5.4.7]{MR1400007}.
\end{proof}

\begin{lemma}[Weak Young inequality]
\label{lemma:2.6chineses}
Let $n \in \mathbb{N}, \mu \in (0,N), \hat{p}, \hat{r}>1$ and $\frac{1}{\hat{p}}+\frac{\mu}{N}=1+\frac{1}{\hat{r}}$. If $v \in L^{\hat{p}}(\mathbb{R}^N)$, then $I_{\mu}\ast v \in L^{\hat{r}}(\mathbb{R}^N)$ and
\begin{align}
    \label{des:2.9chineses}
    \left(\int_{\mathbb{R}^N}|I_{\mu}\ast v|^{\hat{r}}\right)^{\frac{1}{\hat{r}}} \leqslant C(N, \mu, \hat{p})\left(\int_{\mathbb{R}^N}|v|^{\hat{p}}\right)^{\frac{1}{\hat{p}}},
\end{align}
where $I_{\mu}(x)=|x|^{-\mu}$. In particular, we can set $\hat{r}=\frac{N\hat{p}}{N-(N-\mu)\hat{p}}$ for $\hat{p}\in \left(1,\frac{N}{N-\mu}\right)$.
\end{lemma}
\begin{proof}
    See~Lieb \& Loss~\cite[Section~4.3]{MR1817225}
\end{proof}

We will use the Lemmas~\ref{lemma:2.5chineses}, ~\ref{lemma:2.6moroz} and~\ref{lemma:2.6chineses} to prove the next result, which is a generalization of Moroz \& Van Schaftingen~\cite[Lemma~2.4]{MR3056699}.

\begin{lemma}[Another variant of Brézis--Lieb lemma]
\label{lemma:2.7chineses}
Let $N \in \mathbb{N}, \mu \in (0,N), \frac{N-\delta-\mu/2}{N-\beta}\leqslant q< \infty$ and let $\{u_k\}_{k \in \mathbb{N}}$ be a bounded sequence in $L^{\frac{(N-\beta)q}{N-\delta-\mu/2}}(\mathbb{R}^N)$. If $u_k \to u$ a.e. on $\mathbb{R}^N$ as $k \to \infty$, then
\begin{align}
    \label{lim:2.10chineses}
    \lim\limits_{k \to \infty} \int_{\mathbb{R}^N} \biggl[\biggl(I_{\mu}\ast \frac{|u_k|^q}{|x|^{\delta }}\biggr)\frac{|u_k|^q}{|x|^{\delta }}-\biggl(I_{\mu}\ast \frac{|u_k-u|^q}{|x|^{\delta }}\biggr)\frac{|u_k-u|^q}{|x|^{\delta }}\biggr]\dd x=\int_{\mathbb{R}^N}\left(I_{\mu}\ast \frac{|u|^q}{|x|^{\delta }}\right)\frac{|u|^q}{|x|^{\delta }}\dd x.
\end{align}
\end{lemma}
\begin{proof}
     For every $k \in \mathbb{N}$, one has
     \begin{align*}
      \lefteqn{\int_{\mathbb{R}^N} \biggl[\biggl(I_{\mu}\ast \frac{|u_k|^q}{|x|^{\delta}}\biggr)\frac{|u_k|^q}{|x|^{\delta}}-\biggl(I_{\mu}\ast \frac{|u_k-u|^q}{|x|^{\delta}}\biggr)\frac{|u_k-u|^q}{|x|^{\delta}}\biggr]\dd x} \\
      &  = \int_{\mathbb{R}^N}\biggl(I_{\mu}\ast \frac{|u_k|^q}{|x|^{\delta}}\biggr)\frac{|u_k|^q}{|x|^{\delta}}\dd x -\int_{\mathbb{R}^N}\biggl(I_{\mu}\ast \frac{|u_k - u|^q}{|x|^{\delta}}\biggr)\frac{|u_k|^q}{|x|^{\delta}}\dd x
      \\ 
      &  \qquad - \int_{\mathbb{R}^N}\biggl(I_{\mu}\ast \frac{|u_k|^q}{|x|^{\delta}}\biggr)\frac{|u_k-u|^q}{|x|^{\delta}}\dd x
       + \int_{\mathbb{R}^N}\biggl(I_{\mu}\ast \frac{|u_k-u|^q}{|x|^{\delta}}\biggr)\frac{|u_k-u|^q}{|x|^{\delta}} \dd x
      \\ &  \qquad +2 \int_{\mathbb{R}^N}\biggl(I_{\mu}\ast \frac{|u_k|^q}{|x|^{\delta}}\biggr)\frac{|u_k-u|^q}{|x|^{\delta}}\dd x
       -2 \int_{\mathbb{R}^N}\biggl(I_{\mu}\ast \frac{|u_k-u|^q}{|x|^{\delta}}\biggr)\frac{|u_k-u|^q}{|x|^{\delta}}\dd x\\
       \qquad & = \int_{\mathbb{R}^N}\biggl(I_{\mu}\ast\biggl( \frac{|u_k|^q } {|x|^{\delta}}-\frac{|u_k-u|^q } {|x|^{\delta}}\biggr)\biggr)\frac{|u_k|^q}{|x|^{\delta}}\dd x \\
       & \qquad - \int_{\mathbb{R}^N}\biggl(I_{\mu}\ast\biggl( \frac{|u_k|^q } {|x|^{\delta}}-\frac{|u_k-u|^q } {|x|^{\delta}}\biggr)\biggr)\frac{|u_k-u|^q}{|x|^{\delta}}\dd x \\ & \qquad +2\int_{\mathbb{R}^N}\biggl(I_{\mu}\ast\biggl( \frac{|u_k|^q } {|x|^{\delta}}-\frac{|u_k-u|^q } {|x|^{\delta}}\biggr)\biggr)\frac{|u_k-u|^q}{|x|^{\delta}}\dd x\\
       \qquad & = \int_{\mathbb{R}^N}\biggl[I_{\mu}\ast\left( \frac{|u_k|^q } {|x|^{\delta }}-\frac{|u_k-u|^q } {|x|^{\delta }}\right)\biggr]\biggl(\frac{|u_k|^q}{|x|^{\delta }} - \frac{|u_k-u|^q}{|x|^{\delta }}\biggr)\dd x  
       \\ & \qquad + 2 \int_{\mathbb{R}^N}\biggl[I_{\mu}\ast\left( \frac{|u_k|^q } {|x|^{\delta }}-\frac{|u_k-u|^q } {|x|^{\delta }}\right)\biggr]\biggl(\frac{|u_k-u|^q}{|x|^{\delta}}\biggr)\dd x.
     \end{align*}
     By Lemma~\ref{lemma:2.5chineses} with $r = \frac{2(N-\beta)q}{2N-2\delta-\mu}$, as $k \to + \infty$, 
     \begin{align*}
         \frac{|u_k-u|^q } {|x|^{\delta}} - \frac{|u_k|^q } {|x|^{\delta}} \to \frac{|u|^q } {|x|^{\delta}} \qquad \text{in $L^{\frac{2(N-\beta)}{2N-2 \delta - \mu}}(\mathbb{R}^N)$},
     \end{align*}
      Using this convergence and Lemma~\ref{lemma:2.6chineses} with $\Hat{p}= \frac{2(N-\beta)}{2N -2\delta - \mu}$ and $\hat{r}=\frac{2(N-\beta)}{\mu-2\delta-2\beta}$, we have 
     \begin{align*}
         I_{\mu}\ast \biggl(\frac{|u_k-u|^q}{|x|^{\delta}} - \frac{|u_k|^q}{|x|^{\delta}}\biggr) \to I_{\mu}\ast  \frac{|u|^q}{|x|^{\delta}} \qquad \text{in $L^{\frac{2(N-\beta)}{2N-2 \delta - \mu}}(\mathbb{R}^N)$}.
     \end{align*}
    Finally, by Lemma~\ref{lemma:2.6moroz}
    we deduce that 
    \begin{align*}
    \bigg| \frac{|u_k-u|^q}{|x|^{\delta}} \bigg| \rightharpoonup 0 \qquad  \text{ in $L^{\frac{2(N-\beta)}{2N -2\delta - \mu}}(\mathbb{R}^N)$},    
    \end{align*}
     as $k \to \infty$, and we reach the conclusion.
\end{proof}

\begin{lemma}
\label{lemma:2.8chineses}
Let $s \in (0,1), 0\leqslant \alpha <sp+\theta<N, \mu \in (0,N)$ and $2\delta + \mu <N$. If $\{u_k\}_{k \in \mathbb{N}}$ is a bounded sequence in $\Dot{W}^{s,p}_{\theta}(\mathbb{R}^N)$ and $u_k \rightharpoonup u$ in $\Dot{W}^{s,p}_{\theta}(\mathbb{R}^N)$, then we have 
\begin{align*}
    \lim\limits_{k \to \infty} Q^{\sharp}(u_k,u_k) = \lim\limits_{k \to \infty}Q^{\sharp}(u_k-u,u_k-u)+Q^{\sharp}(u,u).
\end{align*}

\end{lemma}

\begin{proof}
    Consider $s \in (0,1), 0<sp+\theta<N$ and $2\delta + \mu < N$. For $p \geqslant 2$, we have 
    \begin{align*}
        p^\sharp_s(\delta,\theta,\mu)  > \frac{p(N-\delta -\mu/2)}{N} \geqslant 1.
    \end{align*}
For $1<p<2$, we use the above specified intervals and we also impose the additional condition $\delta + \mu/2 < sp + \theta < N$; therefore, in this case we also have $p^\sharp_s(\delta,\theta,\mu)$.

    Taking $q=p^\sharp_s(\delta,\theta,\mu)$ in Lemma~\ref{lemma:2.7chineses}, we obtain
    \begin{align*}
        \frac{(N-\beta )q}{N-\delta -\mu/2}=\frac{N-\beta}{N-\delta - \mu/2}p^\sharp_s(\delta,\theta,\mu)
        = p^*_s(\beta,\theta).
    \end{align*}
    Since $\{u_k \}_{k\in\mathbb{N}}\in \Dot{W}^{s,p}_{\theta}(\mathbb{R}^N)$ and $u_k \rightharpoonup u$ in $\Dot{W}^{s,p}_{\theta}(\mathbb{R}^N)$, the embedding $\Dot{W}^{s,p}_{\theta}(\mathbb{R}^N)  \hookrightarrow L^{p^*_s(\beta, \theta
    )}(\mathbb{R}^N,|x|^{-\beta})$ in the Lemma~\ref{lemma:2.2chineses:bis} implies that
    \begin{align*}
        \biggl(\int_{\mathbb{R}^N} \frac{|u_k|^{p^*_s(\beta,\theta)}}{|x|^{\beta}}\dd x\biggr)^{\frac{p}{p^*_s(\beta,\theta)}} \leqslant C \int_{\mathbb{R}^N}\int_{\mathbb{R}^N} \frac{|u_k(x)-u_k(y)|^p}{|x|^{\theta _1}|x-y|^{N+sp}|y|^{\theta _2}} \dd y \dd x \leqslant C.
    \end{align*}
    Therefore, $u_k, u \in L^{p^*_s(\beta,\theta)}(\mathbb{R}^N, |x|^{-\beta})$ and as $k\to+\infty$,
    \begin{align*}
        \frac{u_k}{|x|^{\frac{\beta}{p^*_s(\beta,\theta)}}} \to \frac{u}{|x|^{\frac{\beta}{p^*_s(\beta,\theta)}}}  \quad \text{a.e. on $\mathbb{R}^N$}.
    \end{align*}
    Consequently, Lemma~\ref{lemma:2.7chineses} gives the desired equality.
\end{proof}

\begin{lemma}
\label{lemma:2.9chineses}
Let $s \in (0,1), 0 \leqslant \alpha, \beta <sp+\theta<N$ and $\mu \in (0,N)$ and let $\{u_k\}_{k \in \mathbb{N}}$ be a bounded sequence in $L^{{p^*_s}(\alpha,\theta)}(\mathbb{R}^N, |x|^{-\alpha})$. If $u_k \to u$ a.e. on $\mathbb{R}^N$ as $k \to +\infty$, then for any $\phi \in L^{{p^*_s}(\alpha,\theta)}(\mathbb{R}^N, |x|^{-\alpha})$, we have
\begin{align}
    \label{eq:2.11chineses}
    \lim_{k \to \infty} \int_{\mathbb{R}^N}\left[I_{\mu} \ast F_{\alpha}(\cdot, u_k) \right](x)f_{\alpha}(x,u_k) \phi (x) \dd x = \int_{\mathbb{R}^N} \left[I_{\mu} \ast F_{\alpha}(\cdot, u) \right](x)f_{\alpha}(x,u) \phi (x) \dd x,
\end{align}
where $F_{\alpha}$ and $f_{\alpha}$ were introduced in~\eqref{def:f}.
\end{lemma}

\begin{proof}
Using $\phi = \phi _+ -\phi_- $, it is enough to prove our lemma for $\phi \geqslant 0$. Denote $\Tilde{u}_k = u_k - u$ and observe that
\begin{align}
    \label{eq:2.4indianos}
   {} & \int_{\mathbb{R}^N} \left[I_{\mu} \ast F_{\alpha}(\cdot, u) \right](x)f_{\alpha}(x,u) \phi (x) \dd x  \nonumber
    \\ & \quad =
    \int_{\mathbb{R}^N} \Bigl[I_{\mu} \ast \frac{|u(x)|^{p^\sharp_s(\delta,\theta,\mu)}}{|x|^{\delta}} \Bigr]\frac{|u(x)|^{p^\sharp_s(\delta,\theta,\mu) -2}\cdot u(x)}{|x|^{{\delta}}} \phi (x) \dd x \nonumber
    \\ & \quad =
    \int_{\mathbb{R}^N} \Bigl[I_{\mu} \ast \frac{|u(x)|^{p^\sharp_s(\delta,\theta,\mu)}}{|x|^{\delta}} \Bigr]\frac{|u(x)|^{p^\sharp_s(\delta,\theta,\mu) -2}\cdot u(x)}{|x|^{{\delta}}} \phi (x) \dd x \nonumber
       \\ &\qquad + \int_{\mathbb{R}^N} \Bigl[I_{\mu} \ast \frac{|\tilde{u}(x)|^{p^\sharp_s(\delta,\theta,\mu)}}{|x|^{\delta}} \Bigr]\frac{|\tilde{u}(x)|^{p^\sharp_s(\delta,\theta,\mu) -2}\cdot \tilde{u}(x)}{|x|^{{\delta}}} \phi (x) \dd x \nonumber
    \\  &\qquad - \int_{\mathbb{R}^N} \Bigl[I_{\mu} \ast \frac{|\tilde{u}(x)|^{p^\sharp_s(\delta,\theta,\mu)}}{|x|^{\delta}} \Bigr]\frac{|\tilde{u}(x)|^{p^\sharp_s(\delta,\theta,\mu) -2}\cdot \tilde{u}(x)}{|x|^{{\delta}}} \phi (x) \dd x \nonumber
    \\  &\qquad + \int_{\mathbb{R}^N} \Bigl[I_{\mu} \ast \frac{|u(x)|^{p^\sharp_s(\delta,\theta,\mu) -2}\cdot u(x)}{|x|^{{\delta}}} \Bigr] \frac{|\tilde{u}(x)|^{p^\sharp_s(\delta,\theta,\mu)}}{|x|^{\delta}}\phi (x) \dd x \nonumber
    \\ &\qquad - \int_{\mathbb{R}^N} \Bigl[I_{\mu} \ast \frac{|\tilde{u}(x)|^{p^\sharp_s(\delta,\theta,\mu)}}{|x|^{\delta}} \Bigr]\frac{|u(x)|^{p^\sharp_s(\delta,\theta,\mu) -2}\cdot u(x)}{|x|^{{\delta}}} \phi (x) \dd x \nonumber
     \\  & =   \int_{\mathbb{R}^N} \Bigl[I_{\mu} \ast \Bigl (\frac{|{u}(x)|^{p^\sharp_s(\delta,\theta,\mu)}}{|x|^{\delta}} - \frac{|\tilde{u}(x)|^{p^\sharp_s(\delta,\theta,\mu)}}{|x|^{\delta}} \Bigr)\Bigr]\frac{|u(x)|^{p^\sharp_s(\delta,\theta,\mu) -2}\cdot u(x)}{|x|^{{\delta}}} \phi (x) \dd x \nonumber
     \\  &\qquad+ \int_{\mathbb{R}^N} \Bigl[I_{\mu} \ast \Bigl(\frac{|u(x)|^{p^\sharp_s(\delta,\theta,\mu) -2}\cdot u(x)}{|x|^{{\delta}}} - \frac{|\tilde{u}(x)|^{p^\sharp_s(\delta,\theta,\mu) -2}\cdot \tilde{u}(x)}{|x|^{{\delta}}} \Bigr) \Bigr] \frac{|\tilde{u}(x)|^{p^\sharp_s(\delta,\theta,\mu)}}{|x|^{\delta}}\phi (x) \dd x \nonumber
     \\ &\qquad + \int_{\mathbb{R}^N} \Bigl[I_{\mu} \ast \frac{|\tilde{u}(x)|^{p^\sharp_s(\delta,\theta,\mu)}}{|x|^{\delta}} \Bigr]\frac{|\tilde{u}(x)|^{p^\sharp_s(\delta,\theta,\mu) -2}\cdot \tilde{u}(x)}{|x|^{{\delta}}} \phi (x) \dd x. 
\end{align}

Now we apply Lemma~\ref{lemma:2.5chineses} with $q= p^\sharp_s(\delta,\theta,\mu)$ and $r = \alpha/ p^*(\beta, \theta)$, by taking $(w_k, w) = (u_k, u)$. We find that, as $k\to+\infty$,
\begin{align*}   
\bigg|\frac{|u_k|^{p^\sharp_s(\delta,\theta,\mu)}}{|x|^{\delta}} - \frac{|u_k - u|^{p^\sharp_s(\delta,\theta,\mu)}}{|x|^{\delta}} - \frac{|u|^{p^\sharp_s(\delta,\theta,\mu)}}{|x|^{\delta}}\bigg| \to 0 \quad \text{in $L^{\frac{N-\beta}{N - \delta-\mu/2}}(\mathbb{R}^{N})$},
\end{align*}
i.e., as $k\to+\infty$,
\begin{align}  
\label{eq:a}
    {} &\bigg|\frac{|u_k|^{p^\sharp_s(\delta,\theta,\mu)}}{|x|^{\delta}} - \frac{|u_k - u|^{p^\sharp_s(\delta,\theta,\mu)}}{|x|^{\delta}} \bigg| \to  \frac{|u|^{p^\sharp_s(\delta,\theta,\mu)}}{|x|^{\delta}}
    \quad \text{strongly in $L^{\frac{N-\beta}{N-\delta -\mu/2}}(\mathbb{R}^N)$.}
    \end{align}

Analogously applying the same reasoning to $(w_k, w) = (u_k\phi^{1/p^\sharp_s(\delta,\theta,\mu)}, u\phi^{1/p^\sharp_s(\delta,\theta,\mu)})$, as $k\to+\infty$ we obtain 
\begin{align*}  
\lefteqn{\bigg|\frac{|u_k|^{p^\sharp_s(\delta,\theta,\mu) -2}\cdot u_k \phi}{|x|^{\delta}} - \frac{{|u_k- u|}^{p^\sharp_s(\delta,\theta,\mu)-2}(u_k - u)\phi}{|x|^{\delta}} \bigg|} \\
&\to \frac{|u|^{p^\sharp_s(\delta,\theta,\mu)-2}u \phi}{|x|^{\delta}}
\quad \text{strongly in $L^{\frac{N-\beta}{N-\delta -\mu/2}}(\mathbb{R}^N)$.}
\end{align*}

Now, we apply Lemma~\ref{lemma:2.6chineses} with the choices $\hat{p} = \frac{p_{s}^{\ast}(\beta,\theta)}{p_{s}^{\sharp}(\delta,\theta,\mu)}$ and 
$\hat{r}=\frac{p_{s}^{\sharp}(\delta,\theta,\mu)N-p_{s}^{\ast}(\beta,\theta)(N-\mu)}{p_{s}^{\ast}(\beta,\theta)N}$, together with limits~\eqref{eq:a}. We obtain, as $k\to+\infty$
\begin{align*}
    {} & \biggl(\int_{\mathbb{R}^N} \bigg| I_{\mu} \ast \biggl(\frac{|u_k|^{p^\sharp_s(\delta,\theta,\mu)}}{|x|^{\delta}} - \frac{|u_k - u|^{p^\sharp_s(\delta,\theta,\mu)}}{|x|^{\delta}} - \frac{|u|^{p^\sharp_s(\delta,\theta,\mu)}}{|x|^{\delta}}\biggr)\bigg|^{\frac{N-\beta}{\mu/2 -\delta -\mu \beta N + \beta}}\dd x\biggr)^{\frac{\mu/2 - \delta -\mu \beta N + \beta}{N-\beta}} \nonumber \\
    & \qquad \leqslant C(N, \mu, \hat{p})  \biggl(\int_{\mathbb{R}^N}  \bigg|\frac{|u_k|^{p^{\sharp}_{\mu}(\alpha,\theta
    )}}{|x|^{\delta}} - \frac{|u_k - u|^{p^\sharp_s(\delta,\theta,\mu)}}{|x|^{\delta}} - \frac{|u|^{p^\sharp_s(\delta,\theta,\mu)}}{|x|^{\delta
    }}\bigg|^{\frac{N-\beta}{N-\delta -\mu/2}}\dd x\biggr)^{\frac{N-\delta -\mu/2}{N-\beta}} \nonumber \\
    & \qquad \to 0 \quad\text{strongly in $L^{\frac{N-\beta}{\mu/2-\delta-\mu\beta N + \beta}}(\mathbb{R}^N)$}.
\end{align*}
Therefore, as $k\to+\infty$
\begin{align}
\label{parte1eq2.5indianos}
    I_{\mu} \ast \biggl(\frac{|u_k|^{p^\sharp_s(\delta,\theta,\mu)}}{|x|^{\delta}} - \frac{|\Tilde{u}_k|^{p^\sharp_s(\delta,\theta,\mu)}}{|x|^{\delta}} \biggr) \to I_{\mu} \ast \frac{|u|^{p^\sharp_s(\delta,\theta,\mu)}}{|x|^{\delta}} \;\; \text{strongly in $L^{\frac{N-\beta}{\mu/2-\delta-\mu\beta N + \beta}}(\mathbb{R}^N)$}.
\end{align}

In the same way, as $k\to+\infty$ we can obtain
\begin{align}
\label{parte2eq2.5indianos}
    \lefteqn{I_{\mu} \ast \biggl(\frac{|u_k|^{p^{\sharp}_{\mu}(\delta, \theta, \mu)-2}u_k\phi}{|x|^{\delta}} - \frac{|\Tilde{u}_k|^{p^\sharp_s(\delta,\theta,\mu)-2}\Tilde{u}_k\phi }{|x|^{\delta}} \biggr)} \nonumber \\
    &\to I_{\mu} \ast \frac{|u|^{p^\sharp_s(\delta,\theta,\mu)-2}u\phi}{|x|^{\delta}} \quad\text{strongly in $L^{\frac{N-\beta}{\mu/ 2-\delta-\mu\beta N + \beta}}(\mathbb{R}^N)$}.
\end{align}

Since $u_k \rightharpoonup u $ weakly in $L^{p^*_s(\beta, \theta)}(\mathbb{R}^N, |x|^{-\delta})$ as $k\to+\infty$, we also have
\begin{align}
\label{eq:2.6indianos}
    \begin{cases}
        |u_k|^{p^\sharp_s(\delta,\theta,\mu)-2} u_k \phi \rightharpoonup |u|^{p^\sharp_s(\delta,\theta,\mu)-2}u\phi
        \\ |u_k- u|^{p^\sharp_s(\delta,\theta,\mu)} \rightharpoonup 0 \Rightarrow|\Tilde{u}_k|^{p^\sharp_s(\delta,\theta,\mu)} \rightharpoonup 0 \qquad\text{in $L^{\frac{N-\beta }{\mu/2-\delta-\mu\beta N + \beta}}(\mathbb{R}^N, |x|^{-\delta})$}
        \\ |\Tilde{u}_k|^{p^\sharp_s(\delta,\theta,\mu)-2}\Tilde{u}_k\phi \rightharpoonup 0.
    \end{cases}
\end{align}

Combining~\eqref{parte1eq2.5indianos},~\eqref{parte2eq2.5indianos} and~\eqref{eq:2.6indianos} we deduce that
\begin{align}
\label{eq:2.7indianos}
    \begin{cases}
            {} &\lim\limits_{k\to \infty} \displaystyle\int_{\mathbb{R}^N} \biggl[I_{\mu} \ast \biggl(\frac{|u_k|^{p^\sharp_s(\delta,\theta,\mu)}}{|x|^{\delta}}-\frac{|\Tilde{u}_k|^{p^\sharp_s(\delta,\theta,\mu)}}{|x|^{\delta}}\biggr)\biggr] \biggl( |\frac{|u_k|^{p^\sharp_s(\delta,\theta,\mu)-2}u_k\phi}{|x|^{\delta}}  - \frac{|\Tilde{u}_k|^{p^\sharp_s(\delta,\theta,\mu)-2}\Tilde{u}_k\phi}{|x|^{\delta}}\biggr)\dd{x} 
        \\ & \qquad = \displaystyle\int_{\mathbb{R}^N}\biggl(I_{\mu} \ast \frac{|u|^{p^\sharp_s(\delta,\theta,\mu)}}{|x|^{\delta}}\biggr)\frac{|u|^{p^\sharp_s(\delta,\theta,\mu)-2 }u\phi}{|x|^{\delta}} \dd{x},
        \\       &  \lim\limits_{k \to \infty} \displaystyle\int_{\mathbb{R}^N}\bigg[I_{\mu} \ast \biggl(\frac{|u_k|^{p^\sharp_s(\delta,\theta,\mu)}}{|x|^{\delta}} - \frac{|\Tilde{u}_k|^{p^\sharp_s(\delta,\theta,\mu)}}{|x|^{\delta}}\biggr) \biggr]\frac{|\Tilde{u}_k|^{p^\sharp_s(\delta,\theta,\mu)-2}\Tilde{u}_k\phi}{|x|^{\delta}}\dd{x} =0,
        \\   &  \lim\limits_{k \to \infty} \displaystyle\int_{\mathbb{R}^N}\biggl[I_{\mu}\ast \biggl(\frac{|u_k|^{p^\sharp_s(\delta,\theta,\mu)-2}u_k\phi}{|x|^{\delta}} - \frac{|\Tilde{u}_k|^{p^\sharp_s(\delta,\theta,\mu)-2}\Tilde{u}_k\phi}{|x|^{\delta}}\biggr)\biggr]\frac{|\Tilde{u}_k|^{p^\sharp_s(\delta,\theta,\mu)}}{|x|^{\delta}}\dd{x} = 0.
    \end{cases}
\end{align}

By Hölder’s inequality together with Lemma~\ref{lemma:2.6chineses} we obtain
\begin{align}
\label{des:2.8indianos}
   {} & \bigg|\int_{\mathbb{R}^N} \biggl(I_{\mu} \ast \frac{|\Tilde{u}_k|^{p^\sharp_s(\delta,\theta,\mu)}}{|x|^{\delta}}\biggr) \frac{|\Tilde{u}_k|^{p^\sharp_s(\delta,\theta,\mu)-2}\Tilde{u}_k\phi}{|x|^{\delta}}\bigg| \nonumber
    \\ 
    & \leqslant
\biggl(\int_{\mathbb{R}^N}\biggl(I_{\mu} \ast \frac{|\Tilde{u}_k|^{p^\sharp_s(\delta,\theta,\mu)}}{|x|^{\delta}}\dd{x}\biggr)^{\frac{N- \beta}{\mu/2-\delta -\mu/ \beta N + \beta}}\dd{x}\biggr)^{\frac{\mu/2-\delta -\mu \beta N + \beta}{N-\beta}} \nonumber \\ 
&\quad \times \biggl(\int_{\mathbb{R}^N}\biggl(\frac{|\Tilde{u}_k|^{p^\sharp_s(\delta,\theta,\mu)-1}\phi}{|x|^{\delta}}\biggr)^{\frac{p^*_s(\beta, \theta)}{p^{\sharp}_s(\delta , \theta , \mu)}}\dd{x}\biggr)^{\frac{p^{\sharp}_s(\delta , \theta , \mu)}{p^*_s(\beta, \theta)}} \nonumber \\   
        &  \leqslant C\left \| |\Tilde{u}_k|\right \|_{L^{p^*_s(\alpha,\theta)}(\mathbb{R}^N, |x|^{-\delta})}^{\frac{1}{p^\sharp_s(\delta,\theta,\mu)}}   \left\||\Tilde{u}_k|^{p^\sharp_s(\delta,\theta,\mu)-1}\phi\right\|_{L^{\frac{p^*_s(\beta, \theta)}{p^{\sharp}_s(\delta, \theta, \mu)}}(\mathbb{R}^N, |x|^{-\delta})}   \nonumber
        \\   & \leqslant C \left\||\Tilde{u}_k|^{p^\sharp_s(\delta,\theta,\mu)-1}\phi\right\|_{L^{\frac{p^*_s(\beta, \theta)}{p^{\sharp}_s(\delta, \theta, \mu)}}(\mathbb{R}^N, |x|^{-\delta})}.
\end{align}
In the last inequality we used the assumption that $\{\Tilde{u}_k\}_{k \in \mathbb{N}}$ is a bounded sequence in $L^{p^*_s}(\mathbb{R}^N, |x|^{-\alpha})$ and the fact that the parameters in Lemma~\ref{lemma:2.5chineses},
the variant of Br\'{e}zis-Lieb lemma, are in the admissible range.

On the other hand, using $q = p^\sharp_s(\delta,\theta,\mu)$ in Lemma~\ref{lemma:2.6moroz}, we have
\begin{align*}
    |\Tilde{u}_k|^{p^\sharp_s(\delta,\theta,\mu)} \rightharpoonup 0 \quad\text{weakly in $L^{\frac{p^*_s(\beta, \theta)}{p^{\sharp}_s(\delta, \theta, \mu)}}(\mathbb{R}^N, |x|^{-\delta})$} 
    \end{align*}
    as $k\to+\infty$; consequently,
\begin{align*}
    \big\| |\Tilde{u}_k|^{p^\sharp_s(\delta,\theta,\mu)-1}\phi\big\|_{L^{\frac{p^*_s(\beta, \theta)}{p^{\sharp}_s(\delta, \theta, \mu)}}(\mathbb{R}^N, |x|^{-\delta})} 
   \to 0
\end{align*}
as $k\to+\infty$. Thus, from~\eqref{des:2.8indianos} we obtain
\begin{align}
    \label{des:2.9indianos}
    \lim\limits_{k \to \infty}\int_{\mathbb{R}^N} \biggl(I_{\mu} \ast \frac{|\Tilde{u}_k|^{p^\sharp_s(\delta,\theta,\mu)}}{|x|^{\delta}}\biggr) \frac{|\Tilde{u}_k|^{p^\sharp_s(\delta,\theta,\mu)-2}\Tilde{u}_k\phi}{|x|^{\delta}} =0.
\end{align}

Passing to the limit in~\eqref{eq:2.4indianos}, from~\eqref{eq:2.7indianos} and~\eqref{des:2.9indianos} we reach
\begin{align*}
     \lim\limits_{k \to \infty}\int_{\mathbb{R}^N} \left[I_{\mu} \ast F_{\alpha}(\cdot, u_k) \right](x)f_{\alpha}(x,u_k) \phi (x) \dd x = \int_{\mathbb{R}^N} \left[I_{\mu} \ast F_{\alpha}(\cdot, u) \right](x)f_{\alpha}(x,u) \phi (x) \dd x .
\end{align*}
The lemma is proved.
\end{proof}

\section{Proof of Lemma~\ref{prop:1.3chineses:bis}}
\label{proof:prop:1.3chineses:bis}
To prove the Caffarelli-Kohn-Nirenberg's inequality stated in Lemma~\ref{prop:1.3chineses:bis}, first we have to deal with the generalization of the extension problem. The main goal here is to write a formula that extends, to the nonlinear setting, the extension obtained by Caffarelli \& Silvestre in the linear setting. More precisely, for $0 < s < 1$ and $1 < p < +\infty$, and $0 < sp + \theta < N$, consider the extension problem
\begin{align}
\label{prob:del_teso:3.1}
    \begin{cases}
        [(-(\Delta_{x})_{p,\theta}^{s}]u(x,y) + \dfrac{1-sp-\theta}{y} u_{y}(x,y)+u_{yy}(x,y) = 0
        & x\in\mathbb{R}^{N}, 
        \, y \in \mathbb{R}_{+}\\
        u(x,0) = g(x) & x \in \mathbb{R}^{N}
        \end{cases}
\end{align}
The solution of this problem can be obtained by the convolution
\begin{align*}
    u(x,y)&=\int_{\mathbb{R}^{N}}
    P(x-\xi,y)g(\xi)\dd{\xi}
\end{align*}
where the Poisson kernel $P$ is given, up to a multiplicative constant, by
\begin{align*}
    P(x,y) &\coloneqq \dfrac{y^{sp+\theta}}{(|x|^{2}+y^{2})^{\frac{N+sp+\theta}{2}}}.
\end{align*}
By means of these formulas, we define $E_{s,p,\theta}[g](x,y) \coloneqq u(x,y)$, called
the extension operator.
This operator allows one to give a representation formula for the fractional $p$-Laplacian;  see del~Teso, Castro-G\'{o}mez \& V\'{a}zquez~\cite[Theorem~3.1]{MR4303657}.
\begin{proposition}
    Let $0<s<1$, $1<p<+\infty$, $0< sp+\theta < N$, $x_{0} \in \mathbb{R}^{N}$. 
    Suppose that $u \in C^{2}(\mathbb{R}^{N})$ is a continuously bounded function. If $1 < p < 2/(2-s-\theta/p)$, assume additionally that $\nabla u(x_{0})\neq 0$. Then the fractional $p$-Laplacian operator $(-\Delta)_{p,\theta}^{s}$ can be represented by 
    \begin{align}
        (-\Delta)_{p,\theta}^{s}u(x_{0}) & = \lim_{y\to 0} \dfrac{E_{s,p,\theta}[|u(x_{0})-u(\cdot)|^{p-2}(u(x_{0})-u(\cdot))](x_{0},y)}{y^{sp+\theta}}.
    \end{align}
\end{proposition}

Now we state an estimate by
Sawyer \& Wheeden~\cite[Theorem~1]{MR1175693};
see also Muckenhoupt \& Wheeden~\cite[Theorem~D]{MR340523}.

\begin{lemma}
    \label{lemma:3.1chineses}
    Suppose that $0< \Tilde{s}<N, 1<  \Tilde{p} \leqslant \Tilde{q}< + \infty, 0< \Tilde{s}\Tilde{p}+\Tilde{\theta}<N, \Tilde{p}' = \frac{\Tilde{p}}{\Tilde{p}-1}$ and that $V$ and $W$ are nonnegative measurable functions on $\mathbb{R}^N$, $N \geqslant 1$. If, for some $\sigma >1$,
    \begin{align}
        \label{des:3.1chineses}
        |Q|^{\frac{\Tilde{s}}{N}+\frac{1}{\Tilde{q}}-\frac{1}{\Tilde{p}}}\left(\frac{1}{|Q|}\int_{Q}V^{\sigma}\dd y\right)^{\frac{1}{\Tilde{q}\sigma}}\left(\frac{1}{|Q|}\int_Q W^{(1-\tilde{p}')\sigma} \dd y \right)^{\frac{1}{\Tilde{p}'\sigma}}\leqslant C_{\sigma}
    \end{align}
    for all cubes $Q \subset \mathbb{R}^N$, then for any functions $f \in L^{\Tilde{p}}(\mathbb{R}^N, W(y))$, we have
    \begin{align}
        \label{des:3.2chineses}\left(\int_{\mathbb{R}^N}|E_{\Tilde{s},\Tilde{p},\Tilde{\theta}}[f](y)|^{\Tilde{q}}V(y)\dd y\right)^{\frac{1}{\Tilde{q}}} \leqslant CC_{\sigma} \left(\int_{\mathbb{R}^N} |f(y)|^{\Tilde{p}}W(y)\dd y\right)^{\frac{1}{\Tilde{p}}},
    \end{align}
    where $C=C(\Tilde{p},\Tilde{q},N)$ and $E_{s,p,\theta}$ is the extension operator denotes the Riesz potential of order $\Tilde{s}$, namely
    \begin{align}
    \label{des:3.3chineses}
        E_{\Tilde{s},\Tilde{p},\Tilde{\theta}}f(x) = \int_{\mathbb{R}^N} \frac{f(y)}{|x-y|^{N-\Tilde{s}}}\dd y.
    \end{align}
\end{lemma}
\begin{proof}[Proof of Lemma~\ref{prop:1.3chineses:bis}]
    
For $g \in L^{p}(\mathbb{R},|x|^{-\alpha})$, we define the operator
\begin{align}
\label{def:ls}
    E_{s,p,\theta} [g](x) 
    &\coloneqq \int_{\mathbb{R}^{N}} \dfrac{|g(y)|^{p-2}g(y)}{|x|^{\theta_{1}}|x-y|^{N-sp}|y|^{\theta_{2}}}\dd{y}.
\end{align}
If
\begin{align*}
g(x) \coloneqq 
(-\Delta)_{p,\theta}^{s} u(x)  \coloneqq    2
\lim_{\varepsilon \to 0}
\int_{\mathbb{R}^N\backslash 
B_{\varepsilon}(x)}
\dfrac{|u(x)-u(y)|^{p-2}(u(x)-u(y))}{|x|^{\theta_{1}}
|x-y|^{N+sp}|y|^{\theta_{2}}}\dd{y}
    \dd{x} 
\end{align*}
then
\begin{align}
\label{def:u}
    u(x) & = E_{s,p,\theta}[g](x)
    = E_{s,p,\theta}[(-\Delta)_{p,\theta}^{s} u(x)].
\end{align}

    First, we take $\Tilde{s}=s$, $\Tilde{p}=p$,  $\max \{p, p^*_s(0,\theta) -1\}<\Tilde{q}< p^*_s(\alpha,\theta)$,  $\sigma = \frac{1}{p^*_s(0,\theta) - \Tilde{q}}>1$ and
    \begin{alignat*}{2}
    W(y)&\equiv 1, 
   &\qquad V(y)&\coloneqq\frac{|u(y)|^{p^*_s(\alpha,\theta)-\Tilde{q}}}{|y|^{\alpha}}    
    \end{alignat*}
    in Lemma~\ref{lemma:3.1chineses}.
   Then, the left side of inequality~\eqref{des:3.1chineses} becomes
    \begin{align}
        \label{des:3.4chineses}
        |Q|^{\frac{s}{N}+\frac{1}{\Tilde{q}}-\frac{1}{p}}\left(\frac{1}{|Q|}\int_{Q}V^{\sigma}\dd y\right)^{\frac{1}{\Tilde{q}\sigma}}
        & = |Q|^{\frac{s}{N}+\frac{1}{\Tilde{q}}-\frac{1}{p}}\left(\frac{1}{|Q|}\int_{Q} \Big|\frac{|u(y)|^{p^*_s(\alpha,\theta)-\Tilde{q}}}{|y|^{\alpha}}\Big|^{\frac{1}{p^*_s(0,\theta) - \Tilde{q}}}\dd y\right)^{\frac{p_{s}^{\ast}(0,\theta)-\Tilde{q}}{\Tilde{q}}}.
    \end{align}

    Secondly, we verify that this expression is bounded by a constant dependent only on the parameter $\sigma$. To do this, we consider an arbitrary fixed $x \in \mathbb{R}^N$ and, without loss of generality, we can substitute the cube $Q \subset \mathbb{R}^{N}$ with an open ball $B_{R}(x)$. For the chosen parameters, we define $t\coloneqq \frac{\Tilde{q}}{p^*_s(\alpha,\theta)}$; then,  $0<[p^*_s(\alpha,\theta)-\Tilde{q}]\sigma < 1 $ and $\frac{t\sigma \alpha}{1- [p^*_s(\alpha,\theta)-\Tilde{q}]\sigma}<N$.
    Using Hölder's inequality, we deduce that
    \begin{align*}
       \lefteqn{R^{-N} \int_{B_R(x)}V^{\sigma} \dd y }\\
        & = R^{-N}   \int_{B_R(x)}\frac{|u(y)|^{(p^*_s(\alpha,\theta)-\Tilde{q})\sigma}}{|y|^{\alpha \sigma}} \dd y \\
        & \leqslant R^{-N} \biggl(\int_{B_R(x)}\frac{1}{|y|^{\frac{t \alpha \sigma}{1-[p^*_s(\alpha,\theta)-\Tilde{q}]\sigma}}}\dd y\biggr)^{1-[p^*_s(\alpha,\theta)-\Tilde{q}]\sigma} \bigg(\int_{B_R(x)}\frac{|u|}{|y|^r}\dd y\biggr)^{[p^*_s(\alpha,\theta) -\Tilde{q}]\sigma},
    \end{align*}
    where $r \coloneqq \frac{(1-t)\alpha}{p^*_s(\alpha,\theta)-\Tilde{q}}=\frac{\alpha}{p^*_s(\alpha,\theta)}$. To evaluate the first integral on the right-hand side of the previous inequality we use the integration in polar coordinates formula and we obtain
    \begin{align*}
         {} & R^{-N} \int_{B_R(x)}V^{\sigma} \dd y 
         \\ 
        &  \leqslant R^{-N} \left(\int_{0}^R \Tilde{w}^{N-1 -\frac{t \alpha \sigma}{1-[p^*_s(\alpha,\theta)-\Tilde{q}]\sigma}}\dd \Tilde{w}\right)^{1-[p^*_s(\alpha,\theta)-\Tilde{q}]\sigma} \cdot\left(\int_{B_R(x)}\frac{|u|}{|y|^r}\dd y\right)^{[p^*_s(\alpha,\theta) -\Tilde{q}]\sigma} \\ 
        &  = C R^{-t\alpha \sigma - N [p^*_s(\alpha,\theta)-\Tilde{q}]\sigma}\left(\int_{B_R(x)}\frac{|u|}{|y|^r}\dd y\right)^{[p^*_s(\alpha,\theta) -\Tilde{q}]\sigma}. 
    \end{align*}
    This implies that
    \begin{align*}
        {} & R^{s+\frac{N}{\Tilde{q}} -\frac{N}{p}}  \biggl\{R^{-N}\int_{B_R(x)}V^{\sigma} \dd y\biggr\}^{\frac{1}{\Tilde{q}\sigma}} \\ 
        \qquad  & \leqslant R^{s+\frac{N}{\Tilde{q}} -\frac{N}{p}} 
        \biggl\{C R^{-t\alpha \sigma - N [p^*_s(\alpha,\theta)-\Tilde{q}]\sigma}\biggl(\int_{B_R(x)}\frac{|u|}{|y|^r}\dd y\biggr)^{[p^*_s(\alpha,\theta) -\Tilde{q}]\sigma} \biggr\}^{\frac{1}{\Tilde{q}\sigma}} \\
        & = C\left\{R^{\left({s+\frac{N-t\alpha}{\Tilde{q}} -\frac{N}{p}}\right)\frac{\Tilde{q}}{{p^*_s(\alpha,\theta) -\Tilde{q}}}} \cdot R^{-N}\int_{B_R(x)}\frac{|u|}{|y|^r}\dd y \right\}^{\frac{p^*_s(\alpha,\theta) -\Tilde{q}}{\Tilde{q}}}\\
        & = C\|u\|_{L_M^{1, \frac{N-sp}{p}+r}(\mathbb{R}^N, |y|^{-r})}^{\frac{p^*_s(\alpha,\theta)-\Tilde{q}}{\Tilde{q}}} \coloneqq C_\sigma.
    \end{align*}
Using~\eqref{def:ls} and~\eqref{def:u} and inequality~\eqref{des:3.2chineses} in Lemma~\ref{lemma:3.1chineses}, we obtain
\begin{align*}
    \int_{\mathbb{R}^N} \frac{|u(y)|^{p^*_s(\alpha,\theta)}}{|y|^{\alpha}}\dd y 
    & = \int_{\mathbb{R}^N}|u(y)|^{\Tilde{q}}\frac{|u(y)|^{p^*_s(\alpha,\theta)-\Tilde{q}}}{|y|^{\alpha}} \dd y \\
    & = \int_{\mathbb{R}^N}|E_{s,p,\theta}[g](y)|^{\Tilde{q}} V(y) \dd y \\
    & \leqslant (CC_\sigma)^{\Tilde{q}}\|g\|_{L^p(\mathbb{R}^N)}^{\Tilde{q}} \\
    & = C  \|u\|_{L_M^{1, \frac{N-sp}{p}+r}(\mathbb{R}^N, |y|^{-r})}^{p^*_s(\alpha,\theta)-\Tilde{q}} \|g\|_{L^p(\mathbb{R}^N)}^{\Tilde{q}}\\
     & \leqslant C\|u\|_{L_M^{1, \frac{N-sp}{p}+r}(\mathbb{R}^N, |y|^{-r})}^{p^*_s(\alpha,\theta)-\Tilde{q}}\|u\|_{\Dot{W}^{s,p}_{\theta}(\mathbb{R})}^{\Tilde{q}}  .
\end{align*}
Finally, we choose $p \in [1, p^*_s(\alpha,\theta) )$ and define $\zeta \coloneqq \frac{\Tilde{q}}{p^*_s(\alpha,\theta)}$; for the parameters in the specified intervals we deduce that $\max \left\{ \frac{p}{p^*_s(\alpha,\theta)}, \frac{p^*_s(\alpha,\theta)-1}{p^*_s(\alpha,\theta)}\right\}< \zeta < 1$. Hence, for every function $u \in \Dot{W}^{s,p}_{\theta}(\mathbb{R}^N)$, it holds the inequality
 \begin{align*}
\left(\int_{\mathbb{R}^N}\frac{|u(y)|^{p^*_s(\alpha,\theta)}}{|y|^{\alpha}}\dd y\right)^{\frac{1}{p^*_s(\alpha,\theta)}} \leqslant \|u\|^{1-\zeta}_{L^{\Tilde{q},\frac{N-sp}{p}\Tilde{q}+\Tilde{q}r}(\mathbb{R}^N, |y|^{-\Tilde{q}r})}\|u\|_{\Dot{W}^{s,p}_{\theta}}^{\zeta}.
    \end{align*}
This concludes the proof of Lemma~\ref{prop:1.3chineses:bis}.
\end{proof}

\section{Proof of Proposition~\ref{prop:4.1chineses}}
\label{proof:prop:4.1chineses}
In this section, we deal with a crucial step in the proof of the Theorem~\ref{teo:1.1chineses}. More precisely, we solve the minimization problems~\eqref{probminimizacao1:yang:1.16} and~\eqref{probminimizacao2:yang:1.17}. Using the embeddings of the fractional Sobolev space into the weighted Lebesgue space and the Morrey space in Lemma~\ref{lemma:imersoes}  together with the Caffarelli-Kohn-Nirenberg's inequality in Lemma~\ref{prop:1.3chineses:bis}, we can prove the existence of minimizers for $S_\mu (N, s, \gamma, \alpha)$ and $\Lambda (N, s, \gamma, \beta)$.

\begin{proof}[Proof of Proposition~\ref{prop:4.1chineses}]
\begin{enumerate}[wide]
  \item If $0< \alpha <sp+\theta<N$ and $\gamma < \gamma _H$, let $\{u_k\}_{k\in\mathbb{N}} \subset \Dot{W}^{s,p}_{\theta}(\mathbb{R}^N)$ be a minimizing sequence of $S_\mu (N, s, \gamma , \alpha)$ such that
  \begin{align}
        \label{q:igual1}
            Q^{\sharp}(u_k, u_k) = 1, \qquad \|u_k\|^p \to S_\mu (N, s, \gamma, \alpha)
        \end{align}
        as $k\to +\infty$.
    Recall that $r=\frac{\alpha}{p^*_s(\alpha,\theta)}$.    
The embeddings~\eqref{imersoes} and the Caffarelli-Kohn-Nirenberg's inequality~\eqref{des:1.10chineses:bis} imply that 
\begin{align*}
    \|u_k\|_{L_{M}^{q,\frac{N-sp-\theta}{p}q+qr}(\mathbb{R}^N, |y|^{-pr})} & \leqslant C \|u_k\|_{L^{p^*_s(\alpha,\theta)}(\mathbb{R}^N, |y|^{-\alpha})} 
    \\ & \leqslant C_{1}\|u_k\|_{\Dot{W}^{s,p}_{\theta}(\mathbb{R}^N)}^\zeta \|u_k\|^{1-\zeta}_{L_{M}^{q,\frac{N-sp-\theta}{p}q+qr}(\mathbb{R}^N, |y|^{-pr})}. 
\end{align*}
Therefore,
\begin{align*}
      \|u_k\|_{L_{M}^{q,\frac{N-sp-\theta}{p}q+qr}(\mathbb{R}^N, |y|^{-pr})} \leqslant C_{1} \|u_k\|_{\Dot{W}^{s,p}_{\theta}(\mathbb{R}^N)}.
\end{align*}
On the other hand,  using Caffarelli-Kohn-Nirenberg's inequality once again, together with the inequality~\eqref{des:2.6chineses:bis} and the properties~\eqref{q:igual1}, we get
\begin{align*}
    \|u_k\|^{1-\zeta}_{L_{M}^{q,\frac{N-sp-\theta}{p}q+qr}(\mathbb{R}^N, |y|^{-pr})} & \geqslant C \frac{\|u_k\|_{L^{p^*_s(\alpha,\theta)}(\mathbb{R}^N, |y|^{-\alpha})} }{\|u_k\|_{\Dot{W}^{s,p}_{\theta}(\mathbb{R}^N)}^\zeta}  \\
& \geqslant C\frac{(Q^{\sharp}(u_k,u_k))^{\frac{N}{(2N - \mu)p^*_s(\alpha, \theta)}}}{\|u_k\|_{\Dot{W}^{s,p}_{\theta}(\mathbb{R}^N)}^\zeta} \\
& = C\frac{1}{\|u_k\|_{\Dot{W}^{s,p}_{\theta}(\mathbb{R}^N)}^\zeta}.
\end{align*}
We can also deduce that 
$\|u_k\|_{L_{M}^{q,\frac{N-sp-\theta}{p}q+qr}(\mathbb{R}^N, |y|^{-pr})} \geqslant C_{2}$ by the boundedness of the sequence 
$\{u_k\}_{k\in\mathbb{N}}$ in $\Dot{W}^{s,p}_{\theta}(\mathbb{R}^N)$. Putting these results together, we
have
\begin{align}
\label{limitacaomorey}
   C_2  \leqslant \|u_k\|_{L_{M}^{q,\frac{N-sp-\theta}{p}q+qr}(\mathbb{R}^N, |y|^{-pr})} \leqslant C_1.
\end{align}

For any $k \in \mathbb{N}$ large enough, we may find $\lambda _k >0$ and $x_k \in \mathbb{R}^N$ such that
\begin{align*}
    \lambda _k ^{-sp-\theta +pr} \int_{B_{\lambda _{k} (x_k)}} \frac{|u_k(y)|^p}{|y|^{pr}}\dd y > \|u_k\|^p_{L_{M}^{p,N-sp-\theta+pr}(\mathbb{R}^N, |y|^{-pr})} -\frac{C_3}{2k} \geqslant C_4 >0
\end{align*}
for constants $C_{3}, C_{4} \in\mathbb{R}_{+}$.

Our goal is to pass to the limit as $k \to +\infty$ in the minimizing sequence. To do this, we create another sequence that will help us to control the radius and the centers of these balls. 
Let 
\begin{align*}
v_k (x) =\lambda _k ^{\frac{N-sp-\theta}{p}}u_k(\lambda _k x)
\end{align*} 
be the appropriate scaling for the class of problems that we consider and define $\Tilde{x}_k \coloneqq \frac{x_k}{\lambda _k}$. Then, using the change of variables $y=\lambda_{k}x$ with $\dd{y}=\lambda_{k}^{N}\dd{x}$, we have
\begin{align}
    \label{des:4.1chineses}
    \int_{B_1(\Tilde{x}_k)} \frac{|v_k(x)|^p}{|x|^{pr}}\dd x &  = \int_{B_1\left( \frac{x_k}{\lambda _k}\right)} \frac{\lambda _k ^{N-sp-\theta}|u_k(\lambda _k x)|^p}{|x|^{pr}} \dd x \nonumber \\
    & = \int _{B_{\lambda _k}(x_k)} \lambda _k ^{-sp-\theta+pr}\frac{|u_k(y)|^p}{|y|^{pr}}\dd y\geqslant C>0.    
\end{align}

Now we claim that $S_{\mu}(N,s,\gamma, \alpha)$ is invariant under the previously defined dilation.

In fact, $ Q^{\sharp}(v_k,v_k) = 1$. To show this property, we use the change of variables $\Bar{x}=\lambda _k x$ and $\Bar{y}=\lambda _k y$, we have 
\begin{align*}
    Q^{\sharp}(v_k,v_k) & = \int_{\mathbb{R}^N}\int_{\mathbb{R}^N} \frac{|\lambda _k ^{\frac{N-sp-\theta}{p}}u_k(\lambda _k x)|^{p^\sharp_s(\delta,\theta,\mu)}|\lambda _k ^{\frac{N-sp-\theta}{p}}u_k(\lambda _k y)|^{p^\sharp_s(\delta,\theta,\mu)}}{|x|^{\delta }|x-y|^{\mu}|y|^{\delta }} \dd x \dd y \\
             & = \int_{\mathbb{R}^N}\int_{\mathbb{R}^N} \frac{|u_k(\Bar{x})|^{p^\sharp_s(\delta,\theta,\mu)}|u_k(\Bar{y})|^{p^\sharp_s(\delta,\theta,\mu)}}{|\Bar{x}|^{\delta}|\Bar{x}-\Bar{y}|^{\mu}|\Bar{y}|^{\delta}} \dd \Bar{x} \dd \Bar{y} \\
& =  Q^{\sharp}(u_k,u_k) = 1.
        \end{align*}

Furthermore, $\|v_k\|^p\to S_{\mu}(N,s,\gamma,\alpha) $. In fact, we know that $\{u_k\}_{k \in \mathbb{N}}$ is a minimizing sequence for $S_{\mu}(N,s,\gamma,\alpha) $.
Using the same change of variables $\Bar{x}=\lambda _k x$ and $\Bar{y}=\lambda _k y$, we obtain
\begin{align*}
    \|v_k\|^p &  = \int_{\mathbb{R}^N} \int_{\mathbb{R}^N} \frac{\lambda _k ^{N-sp-\theta}|u_k(\lambda _kx)-u_k(\lambda _k y)|^{p}}{|x|^{\theta _1}|x-y|^{N+sp}|y|^{\theta _2}}\dd x \dd y - \gamma \int_{\mathbb{R}^N} \frac{\lambda _k ^{N-sp-\theta}|u_k(\lambda _k x)|^p}{|x|^{sp+\theta}}\dd x\\
     & = \int_{\mathbb{R}^N} \int_{\mathbb{R}^N}  \frac{|u_k(\Bar{x}) - u_k(\Bar{y})|^p}{|\Bar{x}|^{\theta _1}|\Bar{x}-\Bar{y}|^{N+sp}|\Bar{y}|^{\theta _2}} \dd \Bar{x}\dd \Bar{y} - \gamma \int_{\mathbb{R}^N} \frac{|u_k(\Bar{x})|^p}{|\Bar{x}|^{sp+\theta}} \dd \Bar{x}\\
     & = \|u_k\|^p.
\end{align*}
And since $\|u_k\|^p\to S_{\mu}(N,s,\gamma,\alpha)$ as $k\to+\infty$, we deduce that
$\|v_k\|^p\to S_{\mu}(N,s,\gamma,\alpha)$ as $k\to+\infty$.

In this way, the sequence
$\{v_k\}_{k\in\mathbb{N}} \subset \Dot{W}^{s,p}_{\theta}(\mathbb{R}^N)$ is also a minimizing sequence for $S_\mu (N, s, \gamma , \alpha)$ such that
we have
\begin{align}
    \label{4.1achineses}
    Q^{\sharp}(v_k,v_k) = 1, \qquad \|v_k\|^p \to S_{\mu}(N,s,\gamma, \alpha).
\end{align}

From inequality~\eqref{des:4.1chineses} together with Hölder's inequality,
\begin{align}
\label{des:4.2chineses}
    0< C \leqslant \int_{B_1(\Tilde{x}_k)} \frac{|v_k(x)|^p}{|x|^{pr}} \dd x  \leqslant C \left(\int_{B_1(\Tilde{x}_k)} \frac{|v_k(x)|^{p
    ^*_s(\alpha,\theta)}}{|x|^{\alpha}}\dd x\right)^{\frac{p}{p^*_s(\alpha,\theta)}}.
\end{align}

We claim that the sequence $\{\Tilde{x}_k\}\subset \mathbb{R}^{N}$ of the centers of the balls is bounded. We argue by contradiction and suppose that $|\Tilde{x}_k| \to + \infty$ as $k \to +\infty$; then for any $x \in B_1(\Tilde{x}_k)$,  we have $ |x| \geqslant |\Tilde{x}_k|-1$ for $k \in \mathbb{N}$ large enough. By Hölder's inequality, we obtain
\begin{align*}
    \int_{B_1(\Tilde{x}_k)} \frac{|v_k(x)|^{p^*_s(\alpha, \theta)}}{|x|^{\alpha}} \dd x & \leqslant 
    \frac{1}{(|\Tilde{x}_k|-1)^{\alpha}}\int_{B_1(\Tilde{x}_k)}|v_k(x)|^{p^*_s(\alpha, \theta)}\dd x\\
     \qquad & \leqslant \frac{C}{(|\Tilde{x}_k|-1)^{\alpha}}\left(\int_{B_1(\Tilde{x}_k)} |v_k(x)|^{p^*_s(0,\theta)}\dd x\right)^{\frac{p^*_s(\alpha,\theta)}{p^*_s(0,\theta)}}\\
     \qquad & \leqslant \frac{C}{(|\Tilde{x}_k|-1)^{\alpha}}\left(\int_{\mathbb{R}^N} |v_k(x)|^{p^*_s(0,\theta)}\dd x\right)^{\frac{p^*_s(\alpha,\theta)}{p^*_s(0,\theta)}}\\
     &  = \frac{C}{(|\Tilde{x}_k|-1)^{\alpha}}\|v_k(x)\|_{L^{p^*_s(0,\theta)}}^{p^*_s(\alpha,\theta)}.
\end{align*}
From this inequality, together with the embeddings in Lemma~\ref{lemma:imersoes}--(4), we deduce that
\begin{align*}
\int_{B_1(\Tilde{x}_k)} \frac{|v_k(x)|^{p^*_s(\alpha, \theta)}}{|x|^{\alpha}} \dd x 
     & \leqslant  \frac{C}{(|\Tilde{x}_k|-1)^{\alpha}}\|v_k(x)\|_{L^{p^*_s(\alpha,\theta)}}^{p^*_s(0,\theta)}  \\
     & \leqslant  \frac{C}{(|\Tilde{x}_k|-1)^{\alpha}}\|v_k(x)\|_{\Dot{W}^{s,p}_{\theta}(\mathbb{R}^N)}^{p^*_s(\alpha, \theta)} \\
      & \leqslant \frac{C}{(|\Tilde{x}_k|-1)^{\alpha}}
      =0 \quad (k\to+\infty),
      \end{align*}
where we used the boundedness of the minimizing sequence 
$\{v_k\}_{k\in\mathbb{N}}\subset \Dot{W}^{s,p}_{\theta}(\mathbb{R}^N)$. 
This is a contradiction with inequality~\eqref{des:4.2chineses} and this implies that the sequence $\{\Tilde{x}_k\}\subset \mathbb{R}^{N}$ is bounded. 

From inequality~\eqref{des:4.1chineses} and the boundedness of
the sequence $\{\Tilde{x}_k\}\subset \mathbb{R}^{N}$ of the centers of the balls, we may find $R>0$ such that $B_{R}(0)$ contains all balls of center $\Tilde{x}_k$ and radius $1$; moreover, with
\begin{align}
    \label{des:4.3chineses}
    \int_{B_R(0)} \frac{|v_k(x)|^p}{|x|^{pr}}\dd x \geqslant C_1>0.
\end{align}
Since $\|v_k\|=\|u_k\|\leqslant C$ for $k \in \mathbb{N}$ large enough, there exists a function $v \in \Dot{W}^{s,p}_{\theta}(\mathbb{R}^N)$ such that 
\begin{align}
    \label{4.4chineses}
    v_k &\rightharpoonup v\quad \text{in $\Dot{W}^{s,p}_{\theta}(\mathbb{R}^N)$},
    &\quad 
    v_k &\to v \; \text{a.e. \quad on $\mathbb{R}^N$},
\end{align}
as $k \to +\infty$, up to subsequences. According to Lemma~\ref{lemma:2.3chineses}, we have 
\begin{align*}
     \frac{u_k}{|x|^{r}} \to \frac{u}{|x|^{r}} \; \text{in $L^p_{\operatorname{loc}}(\mathbb{R}^N)$};
\end{align*}
hence, 
\begin{align*}
    \int_{B_R(0)} \frac{|v(x)|^p}{|x|^{pr}}\dd x \geqslant C_1>0,
\end{align*}
and we deduce that $v\not\equiv 0$. 

We may verify by Lemma~\ref{lemma:2.8chineses} that
\begin{align}
\label{limiteqsus}
    1 = Q^{\sharp}(v_k,v_k) = Q^{\sharp}(v_k-v,v_k-v)+ Q^{\sharp}(v,v)+ o(1).
\end{align}
By definition~\eqref{probminimizacao1:yang:1.16}, by weak convergence $v_k \rightharpoonup v$ in $\Dot{W}^{s,p}_{\theta}(\mathbb{R}^N)$ together with Br\'{e}zis-Lieb's lemma and by the estimate~\eqref{limiteqsus}, we have
\begin{align*}
    S_{\mu}(N,s,\gamma,\alpha) & = \lim\limits_{k \to \infty}\|v_k\|^p 
    = \|v\|^p +\lim\limits_{k \to \infty}\|v_k-v\|^p\\
    & \geqslant S_{\mu}(N, s, \gamma, \alpha) \bigr(Q^{\sharp}(v,v)+ \lim\limits_{k \to \infty}Q^{\sharp}(v_k-v,v_k-v)\bigl)^{\frac{p}{p^{\sharp}(\delta, \theta, \mu)}}\\
    & =  S_{\mu}(N,s,\gamma,\alpha),
\end{align*}
where in the last but one passage above we used the inequality 
$(a+b)^{q} \leqslant a^{q}+b^{q}$,
valid for all $a, b  \in \mathbb{R}_{+}^{*}$ and $q>1$. So we have equality in all passages, that is, 
\begin{align}
    \label{4.4achineses}
    Q^{\sharp}(v,v)&=1, 
    &\qquad \lim\limits_{k\to \infty}Q^{\sharp}(v_k-v,v_k-v)&=0,
\end{align}
since $v \not\equiv 0$. It turns out that, since
\begin{align*}
    S_{\mu}(N,s,\gamma,\alpha) = \|v\|^p+\lim\limits_{k \to \infty}\|v_k-v\|^p,
\end{align*}
then
\begin{align*}
    S_{\mu}(N,s,\gamma,\alpha) &=\|v\|^p, 
    &\qquad 
    \lim\limits_{k \to \infty}\|v_k-v\|^p&=0.
\end{align*}

Finally, by inequality
\begin{align*}
\int_{\mathbb{R}^N} \int_{\mathbb{R}^N} 
\dfrac{||u(x)|-|u(y)||^{p}}{|x|^{\theta _1}|x-y|^{N+sp}|y|^{\theta _2}}
\dd x \dd y \leqslant \int_{\mathbb{R}^N} \int_{\mathbb{R}^N} 
\dfrac{|u(x)-u(y)|^{p}}{|x|^{\theta _1}|x-y|^{N+sp}|y|^{\theta _2}}
\dd x \dd y.
\end{align*} 
we deduce that $|v| \in \Dot{W}^{s,p}_{\theta}(\mathbb{R}^N)$ is also a minimizer for $S_{\mu}(N,s,\gamma,\alpha)$; so we can assume that $v\geqslant 0$. Thus, $S_{\mu}(N,s,\gamma,\alpha)$ is achieved by a non-negative function in the case $0<\alpha<sp+\theta$ and $\gamma < \gamma _H$.

\item For $0< \beta <sp+\theta<N$ and $\gamma < \gamma _H$, let $\{u_k\}_{k\in\mathbb{N}}\subset \Dot{W}^{s,p}_{\theta}(\mathbb{R}^N)$ be a minimizing sequence for $\Lambda (N, s, \gamma , \beta)$ such that, as $k \to +\infty$,
        \begin{align*}
\int_{\mathbb{R}^N}\frac{|u_k|^{p^{\ast}_s(\alpha,\theta)}}{|x|^{\alpha}}\dd x &= 1, 
&\quad 
\|u_k\|^p &\to \Lambda (N, s, \gamma, \beta)
        \end{align*}
         
Now we claim that $\Lambda(N,s,\gamma, \beta)$ is invariant under the previously defined dilation.

Let $v_k (x) =\lambda _k ^{\frac{N-sp-\theta}{p}}u_k(\lambda _k x)$ and $\Tilde{x}_k =\frac{x_k}{\lambda _k}$ as in the previous case. 
In this way, the sequence $\{v_k\}_{k\in\mathbb{N}} \subset \Dot{W}^{s,p}_{\theta}(\mathbb{R}^N)$ is also a minimizing sequence for $\Lambda (N, s, \gamma , \beta)$ such that
we have
\begin{align}
    \label{4.1bchineses}
    \int_{\mathbb{R}^N}\frac{|v_k|^{p^{*}_s(\beta,\theta)}}{|x|^{\beta}}\dd x &= 1, &\quad 
    \|v_k\|^p &\to \Lambda(N,s,\gamma, \beta) 
\end{align}

        In fact,  using variable change $\Bar{x}=\lambda _k x$, for every $k \in \mathbb{N}$ we have
        \begin{align*}
\int_{\mathbb{R}^N}\frac{|v_k|^{p^{*}_s(\beta,\theta)}}{|x|^{\beta}}\dd x 
         = \int_{\mathbb{R}^N}\frac{\lambda _k^{N-\beta} u_k(\Bar{x})^{p^*_s(\beta,\theta)}}{\lambda_k^{-\beta}|\Bar{x}|^{\beta}} \frac{\dd \Bar{x}}{\lambda _k^N} = \int_{\mathbb{R}^N}\frac{|u_k|^{p^{*}_s(\beta,\theta)}}{|\Bar{x}|^{\beta}}\dd \Bar{x}=1.
        \end{align*}
We have already shown that
$\| v_{k} \| = \| u_{k} \|$ for every $k \in \mathbb{N}$.
Hence,  $\|v_k\|^p\to \Lambda(N,s,\gamma,\beta)$.

We claim that the sequence $\{\Tilde{x}_k\} \subset \mathbb{R}^{N}$ is bounded and the proof follows the same steps already presented. 
 From this boundedness and inequality~\eqref{des:4.1chineses}, we may find $R>0$ such that $B_{R}(0)$ contains all the unitary balls $B_{1}(\Tilde{x}_{k})$ centered in $\Tilde{x}_k$ and
\begin{align}
    \label{des:4.3achineses}
    \int_{B_R(0)} \frac{|v_k(x)|^p}{|x|^{pr}}\dd x \geqslant C_1>0.
\end{align}
Since $\|v_k\|=\|u_k\|\leqslant C$, there exists a $v \in \Dot{W}^{s,p}_{\theta}(\mathbb{R}^N)$ such that 
\begin{align}
    \label{4.4bchineses}
    v_k &\rightharpoonup v\quad \text{in $\Dot{W}^{s,p}_{\theta}(\mathbb{R}^N)$},
    &\quad 
    v_k &\to v \; \text{a.e. on $\mathbb{R}^N$},
\end{align}
as $k \to +\infty$, up to subsequences. According to Lemma~\ref{lemma:2.3chineses}, we have 
\begin{align*}
    \frac{v_k}{|x|^r} \to \frac{v}{|x|^r} \qquad \text{in $L^p_{\operatorname{loc}}(\mathbb{R}^N)$},
\end{align*}
as $k \to +\infty$, where $r=\frac{\beta}{p^*_s(\beta
,\theta)}$. Therefore, 
\begin{align*}
    \int_{B_R(0)} \frac{|v(x)|^p}{|x|^{pr}}\dd x \geqslant C_1>0,
\end{align*}
and we deduce that $v\not\equiv 0$. 

We may verify by Lemma~\ref{lemma:2.5chineses} that, if $q=p^*_s(\beta,\theta)$ and $\delta = \beta$, then
\begin{align*}
    1 = \int_{\mathbb{R}^N} \frac{|v_k|^{p^*_s(\beta,\theta)}}{|x|^{\beta}}\dd x = \int_{\mathbb{R}^N} \frac{|v_k-v|^{p^*_s(\beta,\theta)}}{|x|^{\beta}}\dd x +\int_{\mathbb{R}^N} \frac{|v|^{p^*_s(\beta,\theta)}}{|x|^{\beta}}\dd x+ o(1).
\end{align*}
By definition~\eqref{probminimizacao2:yang:1.17} and by weak convergence $v_k \rightharpoonup v$ in $\Dot{W}^{s,p}_{\theta}(\mathbb{R}^N)$, we deduce that
\begin{align*}
    \Lambda(N,s,\gamma,\beta) & = \lim\limits_{k \to \infty}\|v_k\|^p 
     = \|v\|^p +\lim\limits_{k \to \infty}\|v_k-v\|^p\\
    & \geqslant \Lambda(N,s,\gamma,\beta)\left(\int_{\mathbb{R}^N} \frac{|v|^{p^*_s(\beta, \theta)}}{|x|^{\beta}}\dd x + \lim\limits_{k \to \infty}\int_{\mathbb{R}^N} \frac{|v_k-v|^{p^*_s(\beta,\theta)}}{|x|^{\beta}}\dd x\right)^{\frac{p}{p^*_s(\beta, \theta)}}\\
    & =  \Lambda(N,s,\gamma,\beta).
\end{align*}
where we used the inequality
$(a+b)^{q} \leqslant a^{q}+b^{q}$,
valid for all $a, b  \in \mathbb{R}_{+}^{*}$ and $q>1$.
So we have equality in all
passages, that is, 
\begin{align}
    \label{4.4cchineses}
    \int_{\mathbb{R}^N} \frac{|v|^{p^*_s(\alpha,\theta)}}{|x|^{\beta}}\dd x &=1, &\quad 
    \lim\limits_{k\to \infty}\int_{\mathbb{R}^N} \frac{|v_k-v|^{p^*_s(\alpha,\theta)}}{|x|^{\beta}}\dd x &=0,
\end{align}
since $v \not\equiv 0$. It turns out that, since
\begin{align*}
    \Lambda(N,s,\gamma,\beta) = \|v\|^p+\lim\limits_{k \to \infty}\|v_k-v\|^p,
\end{align*}
then
\begin{align*}
    \Lambda(N,s,\gamma,\beta) &=\|v\|^p 
    &\quad  
    \lim\limits_{k \to \infty}\|v_k-v\|^p&=0.
\end{align*}
        
         As in the previous case, we deduce that $|v| \in \Dot{W}^{s,p}_{\theta}(\mathbb{R}^N)$ is also a minimizer for $\Lambda(N,s,\gamma,\beta)$ is achieved by a non-negative function in the case $0<\beta<sp+\theta$ and $\gamma < \gamma _H$.

\item In the case $\alpha = 0$ and $0\leqslant \gamma <\gamma _H$, we were inspired by the method introduced by Filippucci et. al~\cite{filippucci2009p} and Dipierro et. al.~\cite{dipierro2016qualitative}. 
Let $\{u_k\}_{k\in\mathbb{N}} \subset \Dot{W}^{s,p}_{\theta}(\mathbb{R}^N)$ be a minimizing sequence for $S_{\mu}(N,s,\gamma,0)$.
Without loss of generality, we can choose this sequence such that 
\begin{align}
    \label{3.1prop4.1}
    Q^{\sharp}(u_k,u_k) 
    &= 1,
    &\quad 
    S_{\mu}(N,s,\gamma,0)
    &\leqslant \|u_k\|^p 
    < S_{\mu}(N,s,\gamma,0)+\frac{1}{k}.
\end{align}
Indeed, by definition~\eqref{probminimizacao1:yang:1.16}, if we normalize $Q^{\sharp}(u_k,u_k) = 1$, then 
        \begin{align*}
            S_{\mu}(N,s,\gamma,0)\leqslant \frac{\|u_k\|^p}{(Q^{\sharp}(u_k,u_k))^{\frac{p}{2p^{\sharp}_{\mu}(0,\theta)}}}\leqslant\|u_k\|^p < S_{\mu}(N,s,\gamma,0)+\frac{1}{k}
        \end{align*}
for $k \in \mathbb{N}$ large enough.
 By inequality
\begin{align}
\label{polya}
\int_{\mathbb{R}^N} \int_{\mathbb{R}^N} 
\dfrac{||u_k(x)|^*-|u_k(y)|^*|^{p}}{|x|^{\theta _1}|x-y|^{N+sp}|y|^{\theta _2}}
\dd x \dd y \leqslant \int_{\mathbb{R}^N} \int_{\mathbb{R}^N} 
\dfrac{|u_k(x)-u_k(y)|^{p}}{|x|^{\theta _1}|x-y|^{N+sp}|y|^{\theta _2}}
\dd x \dd y,
\end{align} 
    where $|u_k|^*$ is the symmetric decreasing rearrangement of $|u_k|$, we deduce that $|u_k|^* \in \Dot{W}^{s,p}_{\theta}(\mathbb{R}^N)$ is also a minimizer for $S_{\mu}(N,s,\gamma,\alpha)$; so we can assume that $u_k\geqslant 0$.

    Furthermore, 
\begin{align}
    \label{3.2prop4.1}
    1=Q^{\sharp}(|u_k|,|u_k|) \leqslant Q^{\sharp}(|u_k|^*,|u_k|^*) 
\end{align}
and
\begin{align}
    \label{3.3prop4.1}
    \int_{\mathbb{R}^N}\frac{|u_k|^p}{|x|^{sp+\theta}}\dd x\leqslant  \int_{\mathbb{R}^N}\frac{||u_k|^*|^p}{|x|^{sp+\theta}}\dd x
\end{align}
Denoting $w_k \coloneqq |u_k|^*$, we have that $v_k$ is radially symmetric and decreasing. Since $0\leqslant \gamma< \gamma _H$, by the definition of $S_{\mu}$ and by inequalities~\eqref{polya}, ~\eqref{3.2prop4.1} and~\eqref{3.3prop4.1}, we deduce that
\begin{align*}
    S_{\mu} \leqslant S_{\mu,\,\text{rad}} & \coloneqq \inf\limits_{w_k \in \Dot{W}^{s,p}_{\theta}(\mathbb{R}^N)\setminus \{0\}} \frac{\displaystyle\int_{\mathbb{R}^N} \int_{\mathbb{R}^N} \frac{|w_k(x)-w_k(y)|^{p}}{|x|^{\theta _1}|x-y|^{N+sp}|y|^{\theta _2}}\dd x \dd y  -\gamma \int_{\mathbb{R}^N}\frac{|w_k|^p}{|x|^{sp+\theta}}\dd x}{Q^{\sharp}(w_k,w_k)^{\frac{p}{2p^{\sharp}_{\mu}(0,\theta)}}} \\
    & \leqslant \frac{\displaystyle\int_{\mathbb{R}^N} \int_{\mathbb{R}^N} \frac{|w_k(x)-w_k(y)|^{p}}{|x|^{\theta _1}|x-y|^{N+sp}|y|^{\theta _2}}\dd x \dd y  -\gamma \int_{\mathbb{R}^N}\frac{|w_k|^p}{|x|^{sp+\theta}}\dd x}{Q^{\sharp}(w_k,w_k)^{\frac{p}{2p^{\sharp}_{\mu}(0,\theta)}}} \\
    & \leqslant \int_{\mathbb{R}^N} \int_{\mathbb{R}^N} \frac{|u_k(x)-u_k(y)|^{p}}{|x|^{\theta _1}|x-y|^{N+sp}|y|^{\theta _2}}\dd x \dd y  -\gamma \int_{\mathbb{R}^N}\frac{|u_k|^p}{|x|^{sp+\theta}}\dd x\\
    & = \|u_k\|^p < S_{\mu}(N,s,\gamma,0)+\frac{1}{k},
\end{align*} 
for $k \in \mathbb{N}$ large enough,  where in the last passage we used inequality~\eqref{3.1prop4.1}. Therefore, $\{w_k\}_{k\in\mathbb{N}}\subset \Dot{W}^{s,p}_{\theta}(\mathbb{R}^N)$ is a minimizing sequence of $S_{\mu}(N,s,\gamma,0)$ and $\{\|w_k\|\}_{k\in\mathbb{N}}\subset \mathbb{R}$ is a  uniformly bounded sequence. Noticing that $Q^{\sharp}(w_k,w_k)\geqslant 1$, the embeddings
    \begin{align*}
        \Dot{W}^{s,p}_{\theta}(\mathbb{R}^N) \hookrightarrow L^{p^*_s(0,\theta)} (\mathbb{R}^N) \hookrightarrow L^{p,N-sp}_{M}(\mathbb{R}^N),
    \end{align*}
    together Caffarelli-Kohn-Nirenberg's inequality~\eqref{des:1.10chineses:bis} imply that 
\begin{align*}
    \|w_k\|_{L_{M}^{q,\frac{N-sp-\theta}{p}q}(\mathbb{R}^N)} & \leqslant C \|w_k\|_{L^{p^*_s(0,\theta)}(\mathbb{R}^N, |y|^{-\alpha})} 
    \leqslant C\|w_k\|_{\Dot{W}^{s,p}_{\theta}(\mathbb{R}^N)}^\zeta \|w_k\|^{1-\zeta}_{L_{M}^{q,\frac{N-sp-\theta}{p}q}(\mathbb{R}^N)}. 
\end{align*}
Therefore,
\begin{align*}
      \|w_k\|_{L_{M}^{q,\frac{N-sp-\theta}{p}q}(\mathbb{R}^N)} \leqslant C_{1} \|w_k\|_{\Dot{W}^{s,p}_{\theta}(\mathbb{R}^N)}.
\end{align*}

On the other hand,  using Caffarelli-Kohn-Nirenberg's inequality once again, together with the inequality~\eqref{des:2.6chineses:bis} and the properties~\eqref{q:igual1}, we get
\begin{align*}
    \|w_k\|^{1-\zeta}_{L_{M}^{q,\frac{N-sp-\theta}{p}q}(\mathbb{R}^N)} & \geqslant C \frac{\|w_k\|_{L^{p^*_s(0,\theta)}(\mathbb{R}^N)} }{\|w_k\|_{\Dot{W}^{s,p}_{\theta}(\mathbb{R}^N)}^\zeta} 
 \geqslant C\frac{(Q^{\sharp}(w_k,w_k))^{\frac{N}{(2N - \mu)p^*_s(0, \theta)}}}{\|w_k\|_{\Dot{W}^{s,p}_{\theta}(\mathbb{R}^N)}^\zeta} 
 = C\frac{1}{\|w_k\|_{\Dot{W}^{s,p}_{\theta}(\mathbb{R}^N)}^\zeta}.
\end{align*}
By the boundedness of the sequence 
$\{w_k\}_{k\in\mathbb{N}}$ in $\Dot{W}^{s,p}_{\theta}(\mathbb{R}^N)$ and the previous inequality, we deduce that there exists a positive constant $C_{2} >0$ such that 
\begin{align*}
\|w_k\|_{L_{M}^{q,\frac{N-sp-\theta}{p}q}(\mathbb{R}^N)}\geqslant C_{2}.    
\end{align*}
Putting these results together, we
have
\begin{align*}
   C_2  \leqslant \|u_k\|_{L_{M}^{q,\frac{N-sp-\theta}{p}q}(\mathbb{R}^N)} \leqslant C_1.
\end{align*}

Using this inequality, we may find $\lambda _k > 0$ and $x_k \in \mathbb{R}^N$ such that
\begin{align*}
    \lambda _k ^{-sp- \theta} \int_{B_{\lambda _k}(x_k)} |w_k(y)|^p \dd y > \|w_k\|_{L_{M}^{p,N-sp-\theta}(\mathbb{R}^N)}^p - \frac{C}{2k}\geqslant C >0.
\end{align*}
Letting $v_k(x) = \lambda _k^{\frac{N-sp-\theta}{p}}w_k(\lambda _k x)$ and $\Tilde{x}_k=\frac{x_k}{\lambda _k}$, we see that $\{v_k\}_{k\in\mathbb{N}}\subset \Dot{W}^{s,p}_{\theta}(\mathbb{R}^N)$ is also a minimizing sequence of $S_{\mu}(N,s,\gamma,0)$ and satisfies
\begin{align}
    \label{des:4.5chineses}
    \int_{B_1(\Tilde{x}_k)}|v_k(x)|^p\dd x \geqslant C>0.
\end{align}
In fact, using the change of variables $y=\lambda _k x$, we obtain
\begin{align*}
\int_{B_1(\Tilde{x}_k)}|v_k(x)|^p\dd x    = \int_{B_1\left(\frac{y_k}{\lambda _k}\right)} \lambda_k^{-sp-\theta}|w_k(y)|^p \dd y 
 = \int_{B_1(\Tilde{y}_k)} \lambda_k^{-sp-\theta}|w_k(y)|^p \dd y\geqslant C >0.
\end{align*}
Furthermore,
\begin{align*}
    \|v_k\|^p & = \int_{\mathbb{R}^N} \int_{\mathbb{R}^N} \frac{|v_k(x)-v_k(y)|^{p}}{|x|^{\theta _1}|x-y|^{N+sp}|y|^{\theta _2}}\dd x \dd y - \gamma \int_{\mathbb{R}^N} \frac{|v_k(x)|^p}{|x|^{sp+\theta}}\dd x\\
    & = \int_{\mathbb{R}^N} \int_{\mathbb{R}^N} \frac{\left(\lambda _k ^{\frac{N-sp-\theta}{p}}|w_k(\lambda _kx)-w_k(\lambda _k y)|\right)^{p}}{|x|^{\theta _1}|x-y|^{N+sp}|y|^{\theta _2}}\dd x \dd y - \gamma \int_{\mathbb{R}^N} \frac{\lambda _k ^{N-sp-\theta}|w_k(\lambda _k x)|^p}{|x|^{sp+\theta}}\dd x.
\end{align*}
Using the change of variables  $\Bar{x}=\lambda _k x$ and $\Bar{y}=\lambda _k y$, we obtain
\begin{align*}
     {} \|v_k\|^p =  \int_{\mathbb{R}^N} \int_{\mathbb{R}^N}  \frac{|w_k(\Bar{x}) - w_k(\Bar{y})|^p}{|\Bar{x}|^{\theta _1}|\Bar{x}-\Bar{y}|^{N+sp}|\Bar{y}|^{\theta _2}} \dd \Bar{x}\dd \Bar{y} - \gamma \int_{\mathbb{R}^N} \frac{|w_k(\Bar{x})|^p}{|\Bar{x}|^{sp+\theta}} \dd \Bar{x} = \|w_k\|^p.
\end{align*}
Since $\|v_k\|=\|w_k\|\leqslant C$, there exists $v\in \Dot{W}^{s,p}_{\theta}(\mathbb{R}^N)$ such that $v_k \rightharpoonup v$ in $\Dot{W}^{s,p}_{\theta}(\mathbb{R}^N)$ weakly as $k \to +\infty$ up to subsequences. 

Now, we need to prove $v\not\equiv 0$. For this purpose, we will consider separatelly the cases where  $\{\Tilde{x}_k\}_{k\in\mathbb{N}}$ unbounded and  $\{\Tilde{x}_k\}_{k\in\mathbb{N}}$ bounded.

\begin{itemize}[wide]
    \item [Case $(1)$.] If $\{\Tilde{x}_k\}_{k\in\mathbb{N}}\subset \mathbb{R}^{N}$ is an unbounded sequence, we assume that $|\Tilde{x}_k| \to +\infty$ up to a subsequence. Since the sequence $\{v_k(x)\}_{\Dot{W}^{s,p}_{\theta}(\mathbb{R}^N)}$ is radially symmetric and decreasing, from inequality~\eqref{des:4.5chineses}, for all $k\in\mathbb{N}$ we have that
    \begin{align*}
        \int_{B_2(0)}|v_k(x)|^p \dd x \geqslant \int_{B_1(0)}|v_k(x+\Tilde{x}_k)|^p \dd x = \int_{B_1(\Tilde{x}_k)}|v_k(x)|^p \dd x \geqslant C >0.
    \end{align*}
Since the embedding $\Dot{W}^{s,p}_{\theta}(\mathbb{R}^N)  \hookrightarrow L_{\operatorname{loc}}^p(\mathbb{R}^N)$ is compact, we have
\begin{align*}
    \int_{B_2(0)}|v(x)|^p \dd x \geqslant C>0.
\end{align*}
So, in the unbounded case we have $v \not\equiv 0$.
     \item [Case $(2)$.] If $\{\Tilde{x}_k\}_{k\in\mathbb{N}}\subset \mathbb{R}^{N}$ is a bounded sequence, from~\eqref{des:4.5chineses} we may find $R>0$ such that 
     \begin{align*}
         \int_{B_R(0)}|v_k(x)|^p \dd x \geqslant C >0,
     \end{align*}
     and from this inequality we deduce that
     \begin{align*}
         \int_{B_R(0)}|v(x)|^p \dd x \geqslant C >0.
     \end{align*}
       Thus,
        in the bounded case we also have $v \not\equiv 0$.
       
\end{itemize}
\item The proof of this item is similar to the previous one and is omitted. \qedhere
    \end{enumerate}
\end{proof}
\section{Existence of Palais-Smale sequences}

We shall now use the minimizers of $S_{\mu}$ and $\Lambda$ obtained in Proposition~\ref{prop:4.1chineses} to prove the existence of a nontrivial weak solution for equation~\eqref{problema0.1}. Recall that the energy functional associated to~\eqref{problema0.1} is
\begin{equation}
    \label{func:5.1chineses}
    \begin{split}
    I(u) & \coloneqq 
    \frac{1}{p}\int_{\mathbb{R}^N} \int_{\mathbb{R}^N} \frac{|u(x)-u(y)|^p}{|x|^{\theta _1}|x-y|^{N+sp}|y|^{\theta _2}}\dd x \dd y  - \frac{\gamma}{p}\int_{\mathbb{R}^N} \frac{|u|^p}{|x|^{sp+\theta}} \dd x
    \\ &-\frac{1}{p^{*}_{s}(\beta,\theta)}\int_{\mathbb{R}^N} \frac{|u|^{p^*_s{(\beta, \theta)}}}{|x|^{\beta}}\dd x 
    -\frac{1}{2p^\sharp_s(\delta,\theta,\mu)}\int_{\mathbb{R}^n}\int_{\mathbb{R}^N}\frac{|u(x)|^{p^\sharp_s(\delta,\theta,\mu)}|u(y)|^{p^\sharp_s(\delta,\theta,\mu)}}{|x|^{\delta }|x-y|^{\mu}|y|^{\delta }}\dd x \dd y, 
    \end{split}
\end{equation}
for all $u \in \Dot{W}^{s,p}_{\theta}(\mathbb{R}^N)$. The fractional Sobolev and fractional Hardy-Sobolev inequalities imply that $I \in C^1(\Dot{W}^{s,p}_{\theta}(\mathbb{R}^N), \mathbb{R})$ and that
\begin{align*}
    \langle I'(u),\phi\rangle & =\int_{\mathbb{R}^N} \int_{\mathbb{R}^N} \frac{|u(x)-u(y)|^{p-2}(u(x)-u(y))(\phi (x)-\phi (y))}{|x|^{\theta _1}|x-y|^{N+sp}|y|^{\theta _2}}\dd x \dd y  - \gamma \int_{\mathbb{R}^N} \frac{|u|^{p-2}u\phi}{|x|^{sp+\theta}} \dd x\\
    & \quad
    -\int_{\mathbb{R}^N} \frac{|u|^{p^*_s{(\beta, \theta)}-2}u(x)\phi }{|x|^{\beta}}\dd x +\int_{\mathbb{R}^n}\int_{\mathbb{R}^N}\frac{|u(x)|^{p^\sharp_s(\delta,\theta,\mu)}|u(y)|^{p^\sharp_s(\delta,\theta,\mu)}}{|x|^{\delta }|x-y|^{\mu}|y|^{\delta }}\dd x\dd y.
\end{align*}
Note that a nontrivial critical point of $I$ is a nontrivial weak solution to equation~\eqref{problema0.1}.

Recall that a Palais-Smale sequence for the energy functional $I$ at the level $c \in \mathbb{R}$ is a sequence 
$\{u_{k}\}_{k\in\mathbb{N}} \subset \Dot{W}^{s,p}_{\theta}(\mathbb{R}^N)$ such that 
    \begin{align}
        \label{lim:5.2chineses}
        \lim\limits_{k \to +\infty} I(u_k) = c \quad \text{and} \quad \lim\limits_{k \to +\infty} I'(u_k)=0 \quad \text{strongly in $\Dot{W}^{s,p}_{\theta}(\mathbb{R}^N)'$}.
    \end{align}
This sequence is refered to as a $(PS)_{c}$ sequence.

Now we state a result that ensures the existence of a Palais-Smale sequence for the energy functional.
\begin{proposition}
    \label{prop:5.2chineses}
    Let $ s \in (0,1), 0<\alpha,\beta<sp+\theta<N, \mu \in(0,N)$ and $\gamma < \gamma _H$. Consider the functional $I \colon \Dot{W}^{s,p}_{\theta}(\mathbb{R}^N) \to \mathbb{R}$ defined in~\eqref{func:5.1chineses} on the Banach space $\Dot{W}^{s,p}_{\theta}(\mathbb{R}^N)$. Then there exists a $(PS)_{c}$ sequence $\{u_k\}\subset \Dot{W}^{s,p}_{\theta}(\mathbb{R}^N)$  for $I$ at some level $c \in (0, c^*)$, 
where
\begin{align}
    \label{def:cestrela}
    c^* 
    & \coloneqq \min \biggl\{
\Bigl(\frac{1}{p}-\frac{1}{2p^{\sharp}_{s}(\delta,\theta,\mu)}\Bigr)
    S_{\mu}^{\frac{2p^{\sharp}_{s}(\delta,\theta,\mu)}{2p^{\sharp}_{s}(\delta,\theta,\mu)-p}}, \Bigl(\frac{1}{p}-\frac{1}{p^{\ast}_{s}(\beta,\theta)}\Bigr)
    \Lambda^{\frac{p^{\ast}_{s}(\beta,\theta)}{p^{\ast}_{s}(\beta,\theta)-p}}\biggr\}.
\end{align}

\end{proposition}

To prove Proposition~\ref{prop:5.2chineses} we need the following version of the mountain pass theorem by Ambrosetti and Rabinowitz~\cite{ambrosetti}. 

\begin{lemma}(Montain Pass Lemma) 
\label{lemma:5.1chineses}
Let $(E, \| \cdot \|)$ be a Banach space and let $I \in C^1(E,\mathbb{R})$ a functional such that the following conditions are satisfied:
\begin{itemize}[wide]
    \item [$(1)$] $I(0)=0$;
    \item [$(2)$] There exist $\rho >0$ and $ r>0$ such that $I(u)\geqslant \rho$ for all $u \in E$ with $\|u\|=r;$
    \item [$(3)$] There exist $v_0 \in E$ such that $\lim\limits_{t \to +\infty} \sup I(tv_0)<0$.
    Let $t_0>0$ be such that $\|t_0v_0\|>r$ and $I(t_0v_0)<0$; define
    \begin{align*}
        c &\coloneqq \inf\limits_{g \in \Gamma} \sup\limits_{t \in [0,1]}I(g(t)),
        &\quad
        \Gamma &\coloneqq \left\{g \in C^0([0,1],E) \colon g(0)=0, g(1)=t_0v_0\right\}.
    \end{align*}
    \end{itemize}
    Then $c\geqslant \rho >0$, and there exists a $(PS)_{c}$ sequence $\{u_k\}\subset E$ for $I$ at level $c $, i.e.,
    \begin{align*}
        \lim\limits_{k \to +\infty} I(u_k) &= c 
        &\quad  
        \lim\limits_{k \to +\infty} I'(u_k)&=0 \quad \text{strongly in $E'$}.
    \end{align*}
\end{lemma}

The proof of Proposition~\ref{prop:5.2chineses} follows from the next two lemmas.

\begin{lemma}
    \label{lema1passodamontanha}
    The functional $I$ verifies the assumptions of Lemma~\ref{lemma:5.1chineses}.
\end{lemma}

\begin{proof}
     Clearly, we have $I(0)=0$. 
     We now verify the second assumption of Lemma~\ref{lemma:5.1chineses}.  Recalling the definition~\eqref{def:qsharp} of the quadratic form $Q^{\sharp}$ and using  inequality~\eqref{des:2.6chineses:bis}, for any $ u \in \Dot{W}^{s,p}_{\theta}(\mathbb{R}^N)$ we obtain
    \begin{align*}
        I(u) & = \frac{1}{p}\|u\|^p -\frac{1}{p^*_s(\beta, \theta)} \int_{\mathbb{R}^N} \frac{|u|^{p^*_s(\beta,\theta)}}{|x|^{\beta}}\dd x -\frac{1}{2p^{\sharp}_{\mu}(\alpha, \theta)}Q^{\sharp}(u,u)\\ 
        & \geqslant \frac{1}{p}\|u\|^p - C_1 \|u\|^{p^*_s(\beta,\theta)} -  C_2 \|u\|^{2p^\sharp_s(\delta,\theta,\mu)}.
    \end{align*}

    Since $s \in (0,1), 0 < \alpha, \beta <sp +\theta <N$ and $\mu \in (0,N)$, we have that $p^*_s(\beta,\theta)>p$ and $2p^\sharp_s(\delta,\theta,\mu)>p^*_s(\alpha,\theta)>p$. Therefore, there exists $r>0$ small enough such that
    \begin{align*}
        \inf\limits_{\|u\|=r}I(u)>0 = I(0),
    \end{align*}
    so item $(2)$ of Lemma~\ref{lemma:5.1chineses} are satisfied.

For $u \in \Dot{W}^{s,p}_{\theta}(\mathbb{R}^N)$ and $t \in \mathbb{R}_{+}$, we have 
\begin{align*}
    I(tu)=\frac{t^p}{p}\|u\|^p-\frac{t^{p^*_s(\beta, \theta)}}{p^*_s(\beta, \theta)} \int_{\mathbb{R}^N}\frac{|u|^{p^*_s(\beta, \theta)}}{|x|^{\beta}} \dd x - \frac{t^{2p^{\sharp}_{\mu}(\alpha, \theta)}}{2p^{\sharp}_{s}(\delta, \theta,\mu)}Q^{\sharp}(u,u);
\end{align*}
since $2p^\sharp_s(\delta,\theta,\mu)>p^*_s(\alpha,\theta)>p$, we deduce that
\begin{align*}
    \lim\limits_{t \to + \infty} I(tu)=-\infty \quad \text{for any $u \in \Dot{W}^{s,p}_{\theta}(\mathbb{R}^N)$}.
\end{align*}
Consequently, for any fixed $v_0 \in \Dot{W}^{s,p}_{\theta}(\mathbb{R}^N)$, there exists $t_{v_{0}}>0$ such that $\|t_{v_0}v_0\|>r$ and $I(t_{v_0}v_0)<0$. Thus, item $(3)$ of Lemma~\ref{lemma:5.1chineses} is satisfied.
\end{proof}

From Lemma~\ref{lema1passodamontanha} above, we guarante by Lemma~\ref{lemma:5.1chineses} the existence of a Palais-Smale sequence $\{u_k\} \subset \Dot{W}^{s,p}_{\theta}(\mathbb{R}^N) $ such that
\begin{align*}
        \lim\limits_{k \to +\infty} I(u_k) &= c 
        &\quad  
        \lim\limits_{k \to +\infty} I'(u_k)&=0 \quad \text{strongly in $\Dot{W}^{s,p}_{\theta}(\mathbb{R}^N)'$}.
    \end{align*}
Moreover, by the definition of $c$ we deduce that $c \geqslant \rho > 0 $. Therefore $c>0$ for all function $u \in \Dot{W}^{s,p}_{\theta}(\mathbb{R}^N)\setminus\{0\}$.

\begin{lemma}
  \label{lema2passodamontanha}
  Suppose that $\mu \in (0,N)$ and that $0< \alpha< sp+ \theta$. Then there exists $u \in \Dot{W}^{s,p}_{\theta}(\mathbb{R}^N)\setminus \{0\}$ such that $c \in (0, c^*)$, where $c^*$ is defined in~\eqref{def:cestrela}.
\end{lemma}

\begin{proof}

Using $(1)$ and $(2)$ in Proposition~\ref{prop:4.1chineses}, we obtain the minimizers $U_{\gamma, \alpha}\in \Dot{W}^{s,p}_{\theta}(\mathbb{R}^N)$ for $S_{\mu}(N,s,\gamma,\alpha)$ and $V_{\gamma, \beta} \in \Dot{W}^{s,p}_{\theta}(\mathbb{R}^N)$ for $\Lambda (N,s,\gamma , \beta)$, respectively. Thus, there exist a function $v_{0}\in \Dot{W}^{s,p}_{\theta}(\mathbb{R}^N)$ defined by
\begin{align*}
    v_0(x) = 
    \begin{cases}
        U_{\gamma, \alpha}(x), \quad \text{if} \; \dfrac{2p^\sharp_s(\delta,\theta,\mu)-p}{2pp^\sharp_s(\delta,\theta,\mu)}S_{\mu}(N,s,\gamma,\alpha)^{\frac{2p^{\sharp}_{\mu}(\alpha, \theta)}{2p^{\sharp}_{\mu}(\alpha, \theta)-p}}\leqslant \dfrac{sp-\beta}{p(N-\beta)}\Lambda (N, s, \gamma,\beta)^{\frac{N-\beta}{sp-\beta}} \\
        V_{\gamma, \beta}(x), \quad \text{if} \; \dfrac{2p^\sharp_s(\delta,\theta,\mu)-p}{2pp^\sharp_s(\delta,\theta,\mu)}S_{\mu}(N,s,\gamma,\alpha)^{\frac{2p^{\sharp}_{\mu}(\alpha, \theta)}{2p^{\sharp}_{\mu}(\alpha, \theta)-p}} > \dfrac{sp-\beta}{p(N-\beta)}\Lambda (N, s, \gamma,\beta)^{\frac{N-\beta}{sp-\beta}}
    \end{cases}
\end{align*}
and a positive number $t_0>0$ such that $\|t_0v_0\|>r$ and $I(t_0v_0)<0$. We can define
\begin{align*}
    c &\coloneqq \inf\limits_{g \in \Gamma}\sup\limits_{t \in [0,1]} I(g(t)),
    &\quad
    \Gamma 
    &\coloneqq \left\{g \in C^0([0,1], \Dot{W}^{s,p}_{\theta}(\mathbb{R}^N)) \colon g(0)=0,g(1)=t_0v_0\right\}.
\end{align*}
Clearly, we have that $c>0$. For the case where $v_0(x)=U_{\gamma, \alpha}(x)$, we can deduce that 
\begin{align*}
    0<c< \frac{2p^\sharp_s(\delta,\theta,\mu)-p}{2pp^\sharp_s(\delta,\theta,\mu)}S_{\mu}(N,s,\gamma,\alpha)^{\frac{p^\sharp_s(\delta,\theta,\mu)}{p^*_{\mu}(\alpha,\theta)-1}}. 
\end{align*}
In fact, for all $ t \geqslant 0$, by the definition of the functional $I$, we have 
\begin{align*}
    I(tv_0)=I(tU_{\gamma, \alpha}) \leqslant \frac{t^p}{p}\|U_{\gamma, \alpha}\|^p - \frac{t^{2p^\sharp_s(\delta,\theta,\mu)}}{2p^\sharp_s(\delta,\theta,\mu)} Q^{\sharp}(U_{\gamma, \alpha},U_{\gamma, \alpha}) \eqqcolon f_1(t).
\end{align*}
It is easy to see that $
    f_1'(t) = t^{p-1}[\|U_{\gamma, \alpha}\|^p - t^{2p^\sharp_s(\delta,\theta,\mu) -p}Q^{\sharp}(U_{\gamma, \alpha},U_{\gamma, \alpha})]$;
so, $f'_{1}(\Tilde{t})=0$ where
$\Tilde{t}=\left(\|U_{\gamma, \alpha}\|^p/Q^{\sharp}(U_{\gamma, \alpha},U_{\gamma, \alpha})\right)^{\frac{1}{2p^\sharp_s(\delta,\theta,\mu)-p}}$
and this is a point of maximum for $f_{1}$. Besides of that, this maximum value is
\begin{align*}
    f_1(\Tilde{t}) = \left[\frac{2p^\sharp_s(\delta,\theta,\mu)-p}{2pp^\sharp_s(\delta,\theta,\mu)}\right]S_{\mu}(N, s, \gamma , \alpha)^{\frac{2p^\sharp_s(\delta,\theta,\mu)}{2p^\sharp_s(\delta,\theta,\mu)-p}}.
\end{align*}
Therefore,
\begin{align}
\label{des:5.3chineses}
    \sup\limits_{t \geqslant 0}I(tU_{\gamma, \alpha}) \leqslant \sup\limits_{t \geqslant 0} f_1(t)  = \frac{2p^\sharp_s(\delta,\theta,\mu)-p}{2pp^\sharp_s(\delta,\theta,\mu)} S_{\mu}(N,s,\gamma,\alpha)^{\frac{2p^\sharp_s(\delta,\theta,\mu)}{2p^\sharp_s(\delta,\theta,\mu)-p}}.
\end{align}
The equality does not hold in~\eqref{des:5.3chineses}; otherwise, we would have that $ \sup\limits_{t \geqslant 0}I(tU_{\gamma, \alpha}) = \sup\limits_{t \geqslant 0} f_1(t)$. Let $t_1>0$ be the point where $\sup\limits_{t \geqslant 0}I(tU_{\gamma, \alpha})$ is attained. We have
\begin{align*}
    f_1(t_1) -\frac{t_1^{p^*_s(\beta,\theta)}}{p^*_s(\beta,\theta)}\int_{\mathbb{R}^N}\frac{|U_{\gamma,\alpha}|^{p^*_s(\beta,\theta)}}{|x|^{\beta}} \dd x = f_1(\Tilde{t})
\end{align*}
which means that $f_1(t_1)>f_1(\Tilde{t})$, since $t_1>0$. This contradicts the fact that $\Tilde{t}$ is the unique maximum point for $f_1$. Thus, we have strict inequality in~\eqref{des:5.3chineses}, that is,
\begin{align}
    \label{des:5.4chineses}
     \sup\limits_{t \geqslant 0}I(tU_{\gamma, \alpha}) < \sup\limits_{t \geqslant 0} f_1(t)  = \frac{2p^{\sharp}_{\mu}(\alpha, \theta)-p}{2pp^{\sharp}_{\mu}(\alpha, \theta)} S_{\mu}(N,s,\gamma,\alpha)^{\frac{2p^{\sharp}_{\mu}(\alpha, \theta)}{2p^{\sharp}_{\mu}(\alpha, \theta)-p}}.
\end{align}
Therefore, $0<c<\frac{2p^{\sharp}_{\mu}(\alpha, \theta)-p}{2pp^{\sharp}_{\mu}(\alpha, \theta)} S_{\mu}(N,s,\gamma,\alpha)^{\frac{2p^{\sharp}_{\mu}(\alpha, \theta)}{2p^{\sharp}_{\mu}(\alpha, \theta)-p}}$.

Similarly, for the case of $v_0(x)=V_{\gamma, \beta}(x)$, we can verify that
\begin{align}
    \label{des:5.5chineses}
    \sup\limits_{t\geqslant 0} I(tV_{\gamma,\beta}) < \frac{sp-\beta}{p(N-\beta)}\Lambda (N,s,\gamma,\beta)^{\frac{N-\beta}{sp-\beta}}.
\end{align}
In fact, for all $ t \geqslant 0$, by the definition of the functional $I$, we have
\begin{align*}
    I(tv_0)=I(tV_{\gamma, \beta}) \leqslant \frac{t^p}{p}\|V_{\gamma, \beta}\|^p - \frac{t^{p^*_s(\beta,\theta)}}{p^*_s(\beta,\theta)} \|u\|^{p^*_s(\beta,\theta)}_{L^{p^*_s(\beta,\theta)}(\mathbb{R}^N,|x|^{-\beta})} \coloneqq g_1(t).
\end{align*}
It is easy to see that $g_1'(t) = t^{p-1}\left(\|V_{\gamma, \beta}\|^p - t^{p^*_s(\beta,\theta) -p}\|u\|^{p^*_s(\beta,\theta)}_{L^{p^*_s(\beta,\theta)}(\mathbb{R}^N,|x|^{-\beta})}\right)$; so, $g'_1(\Tilde{t}) =0$ where 
$\Tilde{t}=\left(
\|V_{\gamma, \beta}\|^p/
\|u\|^{p^*_s(\beta,\theta)}_{L^{p^*_s(\beta,\theta)}(\mathbb{R}^N,|x|^{-\beta})}
\right)^{\frac{1}{p^*_s(\beta,\theta)-p}}$
and this is a point of maximum for $g_1$.
Besides of that, this maximum value is
\begin{align*}
    g_1(\Tilde{t})  = \frac{sp+\theta-\beta}{p(N-\beta)} \Lambda(N,s,\gamma, \beta)^{\frac{N-\beta}{sp+\theta-\beta}}.
\end{align*}
Therefore,
\begin{align}
\label{des:5.3achineses}
    \sup\limits_{t \geqslant 0}I(tV_{\gamma, \beta}) &\leqslant \sup\limits_{t \geqslant 0} g_1(t) = \frac{sp+\theta-\beta}{p(N-\beta)} \Lambda (N,s,\gamma, \beta)^{\frac{N-\beta}{sp+\theta-\beta}}
\end{align}
The equality does not hold in~\eqref{des:5.3achineses}, otherwise, we would have that $ \sup\limits_{t \geqslant 0}I(tV_{\gamma, \beta}) = \sup\limits_{t \geqslant 0} g_1(t)$. Let $t_1>0$, where $\sup\limits_{t \geqslant 0}I(tV_{\gamma, \beta})$ is attained. We have
\begin{align*}
    g_1(t_1) -\frac{t_1^{2p^\sharp_s(\delta,\theta,\mu)}}{2p^\sharp_s(\delta,\theta,\mu)}Q^{\sharp}(V_{\gamma, \alpha}, V_{\gamma, \alpha}) = g_1(\Tilde{t})
\end{align*}
which means that $g_1(t_1)>g_1(\Tilde{t})$, since $t_1>0$. This contradicts the fact that $\Tilde{t}$ is the unique maximum point for $g_1(t)$. Thus
\begin{align}
    \label{des:5.4achineses}
     \sup\limits_{t \geqslant 0}I(tV_{\gamma, \beta}) &< \sup\limits_{t \geqslant 0} g_1(t) = \frac{sp+\theta
     -\beta}{p(N-\beta)}\Lambda (N,s,\gamma,\beta)^{\frac{N-\beta}{sp+\theta-\beta}}.
\end{align}
Therefore, $0<c<\frac{sp+\theta-\beta}{p(N-\beta)}\Lambda(N,s,\gamma,\beta)^{\frac{N-\beta}{sp+\theta-\beta}}$.

From the definition~\eqref{def:cestrela} of $c^{\ast}$ and from inequalities~\eqref{des:5.4chineses} and~\eqref{des:5.4achineses}, we have
$ 0<c < c^* $. The lemma is proved.
\end{proof}

\begin{proof}[Proof of Proposition~\ref{prop:5.2chineses}]
Follows immediately from Lemmas~\ref{lema1passodamontanha} and~\ref{lema2passodamontanha}.
\end{proof}

\begin{proposition}
    \label{prop:5.3chineses}
    Let $s \in (0,1), N>sp+\theta, \alpha =0<\beta <sp+\theta$ or $\beta=0<\alpha<sp+\theta, \mu \in (0,N)$ and $0\leqslant \gamma < \gamma _H$. Consider the functional $I$ defined in~\eqref{func:5.1chineses} on the Banach space $\Dot{W}^{s,p}_{\theta}(\mathbb{R}^N)$. Then there exists a $(PS)$ sequence $\{u_k\} \subset \Dot{W}^{s,p}_{\theta}(\mathbb{R}^N)$ for $I$ at some $c \in (0,c^*)$ where $c^*$ is defined in~\eqref{def:cestrela}, 
    i.e.,
    \begin{align*}
    \lim\limits_{k \to + \infty}I(u_k)&=c,
    &\quad  
    \lim\limits_{k \to + \infty} I'(u_k)&=0 \quad \text{strongly in $\Dot{W}^{s,p}_{\theta}(\mathbb{R}^N)'$}.
\end{align*} 

\end{proposition}

\begin{proof}
    The proof is similar to that of Proposition~\ref{prop:5.2chineses}. Since $0\leqslant \gamma<\gamma _H$, using items $(3)$ and $(4)$ in Proposition~\ref{prop:4.1chineses}, we obtain a minimizer $U_{\gamma, 0} \in \Dot{W}^{s,p}_{\theta}(\mathbb{R}^N) $ for $S_{\mu}(N,s,\gamma,0)$ and $V_{\gamma, 0} \in \Dot{W}^{s,p}_{\theta}(\mathbb{R}^N)$ for $\Lambda(N,s,\gamma,0)$. The rest is standard.
\end{proof}

\section{Proof of Theorem~\ref{teo:1.1chineses}}
The existence of a solution will follow from the proof of the Theorem~\ref{teo:1.1chineses}.

\begin{proof}[Proof of Theorem~\ref{teo:1.1chineses}]

Suppose that $s \in (0,1), 0<\alpha, \beta< sp+\theta, \mu \in (0,N)$ and $\gamma < \gamma _H$.

    Let $\{u_k\}_{k \in \mathbb{N}}\subset \Dot{W}^{s,p}_{\theta}(\mathbb{R}^N)$ be a Palais-Smale sequence $(PS)_{c}$ as in Proposition~\ref{prop:5.2chineses}, i.e.,
    \begin{align*}
        I(u_k) \to c, \; I'(u_k) \to 0 \qquad \text{strongly in $\Dot{W}^{s,p}_{\theta}(\mathbb{R}^N)'$ as $k \to + \infty$}.
    \end{align*}
    Then 
    \begin{align}
        \label{eq:5.6chineses}
        I(u_k)= \frac{1}{p}\|u_k\|^p-\frac{1}{p^*_s(\beta,\theta)}\int_{\mathbb{R}^N}\frac{|u_k|^{p^*_s(\beta,\theta)}}{|x|^{\beta}}\dd x - \frac{1}{2p^\sharp_s(\delta,\theta,\mu)}Q^{\sharp}(u_k,u_k) =c+o(1)
    \end{align}
    and 
    \begin{align}
        \label{eq:5.7chineses}
         \langle I'(u_k),u_k \rangle =\|u_k\|^p-\int_{\mathbb{R}^N}\frac{|u_k|^{p^*_s(\beta,\theta)}}{|x|^{\beta}}\dd x -Q^{\sharp}(u_k,u_k) =o(1).
    \end{align}

    From~\eqref{eq:5.6chineses} and~\eqref{eq:5.7chineses}, if $2p^\sharp_s(\delta,\theta,\mu)\geqslant p^*_s(\beta, \theta)>p$, we have
\begin{align*}
    c+o(1)\|u_k\| & = I(u_k)-\frac{1}{p^*_s(\beta, \theta)}\langle I'(u_k),u_k\rangle \\
    & = \frac{p^*_s(\beta,\theta) -p}{p\cdot p^*_s(\beta,\theta)} \|u_k\|^p + \left(\frac{1}{p^*_s(\beta,\theta)}-\frac{1}{2p^\sharp_s(\delta,\theta,\mu)}\right)Q^{\sharp}(u_k,u_k)\\
    & \geqslant  \frac{p^*_s(\beta,\theta) -p}{p\cdot p^*_s(\beta,\theta)} \|u_k\|^p.
\end{align*}

From~\eqref{eq:5.6chineses} and~\eqref{eq:5.7chineses}, if $ p^*_s(\beta, \theta)> 2p^\sharp_s(\delta,\theta,\mu)>p$, we have
\begin{align*}
    c+o(1)\|u_k\| & = I(u_k)-\frac{1}{2p^\sharp_s(\delta,\theta,\mu)
 }\langle I'(u_k),u_k\rangle \\
    &=\frac{2p^\sharp_s(\delta,\theta,\mu)
  -p}{p\cdot 2p^\sharp_s(\delta,\theta,\mu)
 } \|u_k\|^p + \left(\frac{1}{2p^\sharp_s(\delta,\theta,\mu)
 }-\frac{1}{p^*_s(\beta, \theta)}\right)\int_{\mathbb{R}^N}\frac{|u_k|^{p^*_s(\beta,\theta)}}{|x|^{\beta}}\dd x\\
    & \geqslant \frac{2p^\sharp_s(\delta,\theta,\mu)
  -p}{p\cdot 2p^\sharp_s(\delta,\theta,\mu)
 } \|u_k\|^p.
\end{align*}
Thus, $\{u_k\}_{k \in \mathbb{N}}$ is bounded $\Dot{W}^{s,p}_{\theta}(\mathbb{R}^N)$, so from the estimate~\eqref{eq:5.7chineses} there exists a subsequence, still denoted by $\{u_k\}_{k\in\mathbb{N}}\subset \Dot{W}^{s,p}_{\theta}(\mathbb{R}^N)$, such that 
\begin{align*}
{}& \|u_k\|^p \to b, 
&\qquad
{}&\int_{\mathbb{R}^N}\frac{|u_k|^{p^*_s(\beta,\theta)}}{|x|^{\beta}}\dd x \to d_1,
&\qquad
{}&Q^{\sharp}(u_k,u_k)\to d_2,
\end{align*}
as $k\to+\infty$;
additionally, $b=d_1+d_2.$
By the definitions of $\Lambda$ and $S_{\mu}$, we get
\begin{align*}
    d_1^{\frac{p}{p^*_s(\beta,\theta)}} \Lambda 
    &\leqslant b = d_1+d_2, 
    &\quad 
    d_2^{\frac{1}{p^\sharp_s(\delta,\theta,\mu)}}S_{\mu}
    &\leqslant b = d_1+d_2.
\end{align*}
From the first inequality we have 
$d_1^{\frac{p}{p^*_s(\beta,\theta)}} \Lambda  -d_1 \leqslant d_2 $, that is
\begin{align}
    \label{des:5.8achineses}
     d_1^{\frac{p}{p^*_s(\beta,\theta)}} \Biggl(\Lambda  -d_1^{\frac{p^*_s(\beta,\theta) -p}{p^*_s(\beta,\theta)}}\Biggr) \leqslant d_2.
\end{align}
And from the second inequality we have 
$    d_2^{\frac{1}{p^\sharp_s(\delta,\theta,\mu)}}S_{\mu}-d_2\leqslant d_1 $, that is,
\begin{align}
    \label{des:5.8bchineses}
    d_2^{\frac{1}{p^\sharp_s(\delta,\theta,\mu)}}\biggl(S_{\mu}-d_2^{\frac{p^\sharp_s(\delta,\theta,\mu) -1}{p^\sharp_s(\delta,\theta,\mu)}}\biggr) \leqslant d_1.
\end{align}

\begin{claim}
\label{claim1}
We have $\Lambda  -d_1^{\frac{p^*_s(\beta,\theta) -p}{p^*_s(\beta,\theta)}} >0$ and $S_{\mu} -d_2^{\frac{p^\sharp_s(\delta,\theta,\mu) -1}{p^\sharp_s(\delta,\theta,\mu)}} >0$.
\end{claim}
\begin{proof}[Proof of the Claim~\ref{claim1}]
Since $c+o(1)\|u_k\|=I(u_k)-\frac{1}{p}\langle I'(u_k),u_k\rangle$, we have
\begin{align*}
    \lefteqn{I(u_k)-\frac{1}{p}\langle I'(u_k),u_k\rangle}\\
    &= \frac{1}{p}\|u_k\|^p-\frac{1}{p^*_s(\beta,\theta)} \int_{\mathbb{R}^N}\frac{|u_k|^{p^*_s(\beta,\theta)}}{|x|^{\beta}}\dd x  -\frac{1}{2p^{\sharp}_{\mu}(\alpha, \theta)}Q^{\sharp}(u_k,u_k)\\
    & \quad 
    - \frac{1}{p}\|u_k\|^p+\frac{1}{p}\int_{\mathbb{R}^N}\frac{|u_k|^{p^*_s(\beta,\theta)}}{|x|^{\beta}}\dd x  +\frac{1}{p}Q^{\sharp}(u_k,u_k)\\
    & = \left(\frac{1}{p}-\frac{1}{p^*_s(\beta,\theta)}\right) \int_{\mathbb{R}^N}\frac{|u_k|^{p^*_s(\beta,\theta)}}{|x|^{\beta}}\dd x + \left(\frac{1}{p}-\frac{1}{2p^{\sharp}_{\mu}(\alpha, \theta)}\right)Q^{\sharp}(u_k,u_k)\\
    &=c+o(1)\|u_k\|.
\end{align*}
Passing to the limit as $k\to+\infty$, we get
\begin{align}
    \label{eq:5.9chineses}
    \left(\frac{1}{p}-\frac{1}{p^*_s(\beta,\theta)}\right)  d_1 + \left(\frac{1}{p}-\frac{1}{2p^{\sharp}_{\mu}(\alpha, \theta)}\right) d_2 =c;
\end{align}
so,
\begin{align*}
    d_1 &\leqslant 
    \left(\frac{1}{p}-\frac{1}{p^*_s(\beta,\theta)}\right)^{-1}c 
    = \frac{p(N-\beta)}{sp+\theta -\beta}\, c
\end{align*}
and
\begin{align*}
    d_2 &\leqslant 
   \left(\frac{1}{p}-\frac{1}{2p^{\sharp}_{\mu}(\alpha, \theta)}\right)^{-1} c =\frac{2pp^\sharp_s(\delta,\theta,\mu)}{2p^\sharp_s(\delta,\theta,\mu) -p}\, c.
\end{align*}
Using these upper bounds for $d_1$, $d_2$ and the fact $0<c<c^*$, we have
\begin{align*}
    \Lambda  - d_1 ^{\frac{p^*_s(\beta,\theta)-p}{p^*_s(\beta,\theta)}} &\geqslant  \Lambda 
    -\left(\frac{p(N-\beta)}{sp+\theta-\beta}c\right)^{\frac{p^*_s(\beta,\theta)-p}{p^*_s(\beta,\theta)}} \\
    &> \Lambda 
    -\left(\frac{p(N-\beta)}{sp+\theta-\beta}c^*\right)^{\frac{p^*_s(\beta,\theta)-p}{p^*_s(\beta,\theta)}} \geqslant 0.
\end{align*}
Similarly,
\begin{align*}
    S_{\mu} -d_2^{\frac{p^\sharp_s(\delta,\theta,\mu)-1}{p^\sharp_s(\delta,\theta,\mu)}}    
    &\geqslant 
    S_{\mu} - \left(\frac{2pp^\sharp_s(\delta,\theta,\mu)}{2p^\sharp_s(\delta,\theta,\mu)-p}c\right)^{\frac{p^\sharp_s(\delta,\theta,\mu)-1}{p^\sharp_s(\delta,\theta,\mu)}}\\
    &> 
    S_{\mu} - \left(\frac{2pp^\sharp_s(\delta,\theta,\mu)}{2p^\sharp_s(\delta,\theta,\mu)-p}c^*\right)^{\frac{p^\sharp_s(\delta,\theta,\mu)-1}{p^\sharp_s(\delta,\theta,\mu)}}  \geqslant 0.
\end{align*}    
This concludes the proof of the claim.
\end{proof}

Following up,
inequalities~\eqref{des:5.8achineses} and~\eqref{des:5.8bchineses} imply, respectively, that
\begin{align*}
    \Biggl(\Lambda  
    -\left(\frac{p(N-\beta)}{sp+\theta-\beta}c\right)^{\frac{p^*_s(\beta,\theta)-p}{p^*_s(\beta,\theta)}}\Biggr) d_1^{\frac{p}{p^*_s(\beta,\theta)}} \leqslant \Biggl(\Lambda  - d_1 ^{\frac{p^*_s(\beta,\theta)-p}{p^*_s(\beta,\theta)}}\Biggr)
    d_1^{\frac{p}{p^*_s(\beta,\theta)}} \leqslant d_2
\end{align*}
and
\begin{align*}
    \Biggl(S_{\mu}- \left(\frac{2pp^\sharp_s(\delta,\theta,\mu)}{2p^\sharp_s(\delta,\theta,\mu)-p}c\right)^{\frac{p^\sharp_s(\delta,\theta,\mu)-1}{p^\sharp_s(\delta,\theta,\mu)}}\Biggr)d_2^{\frac{1}{p^\sharp_s(\delta,\theta,\mu)}}
    \leqslant \Biggl(S_{\mu} -d_2^{\frac{p^\sharp_s(\delta,\theta,\mu)-1}{p^\sharp_s(\delta,\theta,\mu)}} \Biggr)
    d_2^{\frac{1}{p^\sharp_s(\delta,\theta,\mu)}} \leqslant d_1.
\end{align*}

If $d_1=0$ and $d_2=0$, then~\eqref{eq:5.9chineses} implies that $c=0$, which is in contradiction with $c>0$. Therefore, $d_1>0$ and $d_2>0$ and we can choose $\epsilon _0>0$ such that $d_1\geqslant \epsilon _0 >0$ and $d_2\geqslant \epsilon _0 >0$; moreover, there exists a $K>0$ such that 
\begin{align*}
    \int_{\mathbb{R}^N}\frac{|u_k|^{p^*_s(\beta, \theta)}}{|x|^{\beta} } \dd x > \frac{\epsilon _0}{2}, \qquad Q^{\sharp} (u_k,u_k)>\frac{\epsilon _0}{2}
\end{align*}
for every $k > K$.
Then the inequality~\eqref{des:2.6chineses:bis}, the embeddings~\eqref{imersoes}, and the improved Sobolev inequality~\eqref{des:1.10chineses:bis} imply that there exists $C_1,C_2>0$ such that 
\begin{align*}
    0<C_2\leqslant \|u_k\|_{L_M^{p,N-sp-\theta+pr}(\mathbb{R}^N,|y|^{-pr})}\leqslant C_1,
\end{align*}
where $r=\frac{\alpha}{p^*_s(\alpha, \theta)}$. 
For any $k>K$, we may find 
$\lambda _k>0$ and 
$x_k \in \mathbb{R}^N$ such that  
\begin{align*}
    \lambda_k^{(N-sp-\theta+pr)-N}\int_{B_{\lambda _k}(x_k)}\frac{|u_k(y)|^p}{|y|^{pr}} \dd y > \|u_k\|^p_{L_M^{p,N-sp-\theta+pr}(\mathbb{R}^N,|y|^{-pr})} -\frac{C}{2k} \geqslant \Tilde{C}>0.
\end{align*}
Now we define the sequence 
$\{v_{k}\}_{k\in\mathbb{N}}\subset \Dot{W}^{s,p}_{\theta}(\mathbb{R}^N)$ by
$v_k(x)=\lambda_k^{\frac{N-sp-\theta}{p}}u_k(\lambda _k x)$. As we have already shown,  $\|v_k\|=\|u_k\|\leqslant C$ for every $k\in\mathbb{N}$; so, there exists a $v \in \Dot{W}^{s,p}_{\theta}(\mathbb{R}^N)$ such that, after passage to subsequence, still denoted in the same way, as $k\to+\infty$, 
\begin{align*}
    v_k \rightharpoonup v \quad \text{in $\Dot{W}^{s,p}_{\theta}(\mathbb{R}^N)$}.
\end{align*}
In a fashion similar to the proof of Proposition~\ref{prop:4.1chineses}-($1$), we can prove that $v\not\equiv 0$.

In addition, the boundedness of the sequence $\{v_{k}\}_{k\in\mathbb{N}}\subset \Dot{W}^{s,p}_{\theta}(\mathbb{R}^N)$ implies that the sequence $\{|v_k|^{p^*_s(\beta, \theta)-2}v_k\}_{k\in\mathbb{N}}\subset L^{\frac{p^*_s(\beta, \theta)}{p^*_s(\beta, \theta) -1}}(\mathbb{R}^N, |x|^{-\beta})$ is bounded also. In fact, by embeddings~\eqref{imersoes}, we obtain 
\begin{align*}
    \int_{\mathbb{R}^N} \frac{\left||v_k|^{p^*_s(\beta, \theta)-2}\cdot v_k\right|^{\frac{p^*_s(\beta, \theta)}{p^*_s(\beta, \theta) -1}}}{|x|^{\beta}} \dd x   = \int_{\mathbb{R}^N}\frac{\left||v_k|^{p^*_s(\beta, \theta)-1}\right|^{\frac{p^*_s(\beta, \theta)}{p^*_s(\beta, \theta) -1}}}{|x|^{\beta}} \dd x  =\int_{\mathbb{R}^N} \frac{|v_k|^{p^*_s(\beta, \theta)}}{|x|^{\beta}} \dd x <C.
\end{align*}

Then, after passage to a subsequence, still denoted in the same way, we deduce that
\begin{align}
    \label{conv:5.10chineses}
    |v_k|^{p^*_s(\beta, \theta)-2}v_k\ \rightharpoonup |v|^{p^*_s(\beta, \theta)-2}v\ \quad \text{in $L^{\frac{p^*_s(\beta, \theta)}{p^*_s(\beta, \theta) -1}}(\mathbb{R}^N, |x|^{-\beta})$}
\end{align}
as $k\to+\infty$.

For any $\phi \in L^{p^*_s(\alpha, \theta)}(\mathbb{R}^N, |x|^{-\alpha})$, Lemma~\ref{lemma:2.9chineses} implies that
\begin{align}
    \label{lim:5.11chineses}
    \lim\limits_{k \to \infty} \int_{\mathbb{R}^N} \left[I_{\mu}\ast F_{\alpha}(\cdot , v_k)\right](x) f_{\alpha}(x,v_k)\phi (x) \dd x = \int _{\mathbb{R}^N} \left[I_{\mu}\ast F_{\alpha}(\cdot , v)\right](x) f_{\alpha}(x,v)\phi (x). \dd x. 
\end{align}
Since $\Dot{W}^{s,p}_{\theta}(\mathbb{R}^N) \hookrightarrow L^{p^*_s(\alpha, \theta)}(\mathbb{R}^N,|x|^{-\alpha})$, the limit \eqref{lim:5.11chineses} holds for any $\phi \in \Dot{W}^{s,p}_{\theta}(\mathbb{R}^N)$.

Finally, we need to check that $\{v_k\}_{k \in \mathbb{N}}\Dot{W}^{s,p}_{\theta}(\mathbb{R}^N)$ is also a $(PS)_{c}$ sequence for the functional $I$ at energy level $c$. For this, the norms in $\Dot{W}^{s,p}_{\theta}(\mathbb{R}^N)$ and $L^{p^*_s(\alpha, \theta)}(\mathbb{R}^N,|x|^{-\alpha})$ are invariant under the special dilatation $v_k(x)\coloneqq \lambda _k^{\frac{N-sp-\theta}{p}}u_k(\lambda _k x)$. In fact,
\begin{align*}
    \|v_k\|^p_{\Dot{W}^{s,p}_{\theta}(\mathbb{R}^N)} & = \int_{\mathbb{R}^N}\int_{\mathbb{R}^N} \frac{|v_k(x)-v_k(y)|^p}{|x|^{\theta_1}|x-y|^{N+sp}|y|^{\theta_2}} \dd x \dd y\\
    &  = \int_{\mathbb{R}^N}\int_{\mathbb{R}^N} \frac{|u_k(\Bar{x})-u_k(\Bar{y})|^p}{|\Bar{x}|^{\theta_1}|x-y|^{N+sp}|\Bar{y}|^{\theta_2}} \dd \Bar{x} \dd \Bar{y} = \|u_k\|^p_{\Dot{W}^{s,p}_{\theta}(\mathbb{R}^N)}
\end{align*}
and
\begin{align*}
    \|v_k\|^{p^*_s(\alpha, \theta)}_{L^{p^*_s(\alpha, \theta)}} & = \int_{\mathbb{R}^N} \frac {|v_k(x)|^{p^*_s(\alpha, \theta)}}{|x|^{\alpha}}\dd x\\
    & =  \int_{\mathbb{R}^N} \frac {|u_k(\Bar{x})|^{p^*_s(\alpha, \theta)}}{|\Bar{x}|^{\alpha}} \dd \Bar{x}  = \|u_k\|^{p^*_s(\alpha, \theta)}_{L^{p^*_s(\alpha, \theta)}}.
\end{align*}
Thus, we have
\begin{align*}
    \lim\limits_{k \to + \infty} I(v_k) =c.
\end{align*}
Moreover, for all $ \phi \in \Dot{W}^{s,p}_{\theta}(\mathbb{R}^N)$, we have $\phi_k (x) \coloneqq \lambda_k^{\frac{N-sp-\theta}{p}}\phi \left(x/\lambda _k\right) \in \Dot{W}^{s,p}_{\theta}(\mathbb{R}^N)$. From $I'(u_k) \to 0$ in $\Dot{W}^{s,p}_{\theta}(\mathbb{R}^N)'$ as $k\to+\infty$, we can deduce that 
\begin{align*}
    \lim\limits_{k \to + \infty} \langle I'(v_k), \phi\rangle = \lim \limits_{k \to + \infty} \langle I'(u_k), \phi\rangle =0.
\end{align*}
Thus, limits~\eqref{conv:5.10chineses} and~\eqref{lim:5.11chineses} lead to
\begin{align*}
    \langle I'(v), \phi\rangle = \lim \limits_{k \to + \infty} \langle I'(v_k), \phi\rangle =0.
\end{align*}
Hence $v$ is a nontrivial weak solution to problem~\ref{problema0.1}.
\end{proof}

\begin{proof}[Sketch of proof of Theorem~\ref{teo:1.2chineses}]
    The proof follows the same steps of the proof of Theorem~\ref{teo:1.1chineses}. Here we only remark that for problem~\eqref{problema0.2} with a Hardy potential and double Sobolev type nonlinearities we have to define the value below which we can recover the compactness of the Palais-Smale sequences by
    \begin{align*}
    c^* \coloneqq \min_{k\in\{1,2\}} \left\{ \frac{sp+\theta-\beta_{k}}{p(N-\beta_{k})}\Lambda (N, s, \gamma,\beta_{k})^{\frac{N-\beta_{k}}{sp+\theta-\beta_{k}}}\right\}.
\end{align*}
Similarly, for problem~\eqref{problema0.3} with a Hardy potential and double Choquard type nonlinearities we have to define the corresponding number by
\begin{align*}
    c^* \coloneqq \min_{k\in\{1,2\}} \left\{\frac{2p^\sharp_s(\delta_{k},\theta,\mu_{k})-p}{2pp^\sharp_s(\delta_{k},\theta,\mu_{k})}S_{\mu_{k}}(N,s,\gamma,\alpha)^{\frac{2p^\sharp_s(\delta_{k},\theta,\mu_{k})}{2p^\sharp_s(\delta_{k},\theta,\mu_{k})-p}}\right\}.
\end{align*}
The details are omitted.
\end{proof}

\subsubsection*{Acknowledgements}
Ol\'{\i}mpio H. Miyagaki was
supported by National Council for Scientific and Technological Development-(CNPq).
Rafaella F. S. Siqueira was
supported by Coordination of Superior Level Staff Improvement-(CAPES).

\subsubsection*{Conflict of interest}
On behalf of all authors, the corresponding author states that there is no
conflict of interest.

\subsubsection*{Data availability statement}
Data sharing is not applicable to this article as no new data were created or analyzed in this study.

\subsubsection*{Orcid}

\begin{tabbing}
\hspace{4.5cm}\=\kill
Ronaldo B. Assun\c{c}\~{a}o 
\> \url{https://orcid.org/0000-0003-3159-6815}\\ 
Olímpio H. Miyagaki
\> \url{https://orcid.org/0000-0002-5608-3760} \\ 
Rafaella F. S. Siqueira
\> \url{https://orcid.org/0009-0007-2271-7327}
\end{tabbing}

\bibliographystyle{IEEEtranS}
\bibliography{bibliografia}

\appendix
\section{Properties of Morrey spaces}\label{appendix:properties:morrey}

Let $\Omega\subset\mathbb{R}^{n}$ be a bounded domain (i.e.,\xspace an open and connected set); let $1\leqslant p \leqslant +\infty$ and $\gamma \geqslant 0$.
The Morrey spaces, denoted by $L^{p,\gamma}_{M}(\Omega)$, are the collection of all functions $u \in L^{p}(\Omega)$ such that
\begin{align*}
\|u\|_{L_{M}^{p,\gamma}(\Omega)} 
&\coloneqq 
\sup_{x\in\Omega,\: 0 < R < \operatorname{diam}(\Omega)}
\biggl\{
\biggl( R^{\gamma-N}
\int_{\Omega \cap B_{R}(x)}
|u|^{p}\dd{x}\biggr)^{1/p}
\biggr\}
< +\infty,
\end{align*}
where
$\operatorname{diam}(\Omega)$ is the diameter of the subset $\Omega \subset \mathbb{R}^{N}$.

\begin{lemma}
\label{gantumur:prop:0.0.4}
\begin{enumerate}[wide]
\item The map $u \mapsto \| u \|_{L_{M}^{p,\gamma}(\Omega)}$ defines a norm on the Morrey space $L_{M}^{p,\gamma}(\Omega)$, making it into a normed vector space.
\item The Morrey space
$L_{M}^{p,\gamma}(\Omega)$
is a Banach space.
\end{enumerate}
\end{lemma}

\begin{lemma}
\label{gantumur:thm:0.0.12}
\begin{enumerate}[wide]
\item For $1 \leqslant p < +\infty$ we have $L^{p,N}_{M} (\Omega) = L^{p}(\Omega)$, i.e.,\xspace
$L^{p,N}_{M} (\Omega)$ and 
$L^{p}(\Omega)$ are continuously embedded in each other. 

\item For $1 \leqslant p < +\infty$ we have $ L^{\infty}(\Omega)
\hookrightarrow L^{p,0}_{M} (\Omega)$. 

\item For $1 \leqslant p < +\infty$ and $\lambda < 0$ we get $L
^{p,\lambda}_{M} (\Omega) = \{0\}$.

\item For $1 \leqslant p \leqslant q < +\infty$ and 
$\lambda,\, \mu \geqslant 0$ with
$\gamma/p \leqslant \mu/q$ it holds 
$L^{q,\mu}_{M} (\Omega) \hookrightarrow L^{p,\gamma}_{M} (\Omega)$.
\end{enumerate}
\end{lemma}

\begin{remark}
\label{gantumur:rem:0.0.13}
Lemma~\ref{gantumur:thm:0.0.12} suggests that for fixed $1 \leqslant p < +\infty$ the Morrey space
$L_{M}^{p,\gamma} (\Omega)$ with $0 \leqslant \gamma \leqslant N$ provides a certain scaling of the spaces between 
$L^{p}(\Omega)$ and $L^{\infty}(\Omega)$. 
Also, taking $p = q$ in 
Lemma~\ref{gantumur:thm:0.0.12}--$4$, we have 
$L_{M}^{p,\gamma_{2}} (\Omega) \hookrightarrow 
 L_{M}^{p,\gamma_{1}} (\Omega)$
whenever $\gamma_{1} \leqslant \gamma_{2}$, 
just like for finite 
$L^{p}$ spaces.
\end{remark}

In general, the Morrey space 
$L_{M}^{p,\gamma+\lambda}(\mathbb{R}^{N},|x|^{-\lambda})$  
is the collection of all measurable functions $u\in L^{p}(\mathbb{R}^{N},|y|^{-\lambda})$
such that 
\begin{align*}
\|u\|_{L_{M}^{p,\gamma+\lambda}(\mathbb{R}^{n},|x|^{-\lambda})}
& \coloneqq
\sup_{x\in\mathbb{R}^{N},\: R\in \mathbb{R}_{+}}
\Bigl\{
\Bigl(
R^{\gamma+\lambda-N}
\int_{B_{R}(x)} \dfrac{|u|^{p}}{|x|^{\lambda}}\dd{x}
\Bigr)^{\frac{1}{p}} \Bigr\} 
< +\infty,
\end{align*}
where $1 \leqslant p < +\infty$;
$\gamma, \, \lambda \in \mathbb{R}_{+}$,
and
$0 < \gamma+\lambda < N$.

\begin{lemma}
\label{lemma:imersoes}
The following fundamental properties are true.
\begin{enumerate}
    \item $L^{p\rho}(\mathbb{R}^N, |y|^{-\rho \lambda}) \hookrightarrow L^{p,\gamma +\lambda}(\mathbb{R}^N, |y|^{-\lambda})$ for $\rho = \frac{N}{\gamma +\lambda}>1$.

    \item For any $p \in (1, + \infty)$, we have $L^{p, \gamma + \lambda}(\mathbb{R}^N, |y|^{- \lambda}) \hookrightarrow L^{1,\frac{\gamma}{p} +\frac{\lambda}{p}}(\mathbb{R}^N, |y|^{-\frac{\lambda}{p}})$. 

    \item For $1 \leqslant p < +\infty$ and $\gamma + \lambda = N$, we have 
    \begin{align*}
       L_{M}^{p,N}(\mathbb{R}^{N}, |y|^{-\lambda}) =L^p(\mathbb{R}^N, |y|^{-\lambda}), 
    \end{align*}
    i.e.,\xspace   $L_{M}^{p,N}(\mathbb{R}^{N}, |y|^{-\lambda})$ and $L^p(\mathbb{R}^N, |y|^{-\lambda})$ are continuously embedded in each other. 
\end{enumerate}
    Moreover, if we assume that $s \in (0,1)$ and $0< \alpha < sp+\theta < N $, then we have
\begin{enumerate}[resume]    
    \item For $1 \leqslant q <  p^*_s(\alpha,\theta)$ and
    $r= \frac{\alpha}{p^*_s(\alpha,\theta)}$, it holds
    \begin{align}
    \label{imersoes}
        \dot{W}^{s,p}_{\theta}(\mathbb{R}^N) \hookrightarrow L^{p^*_s(\alpha,\theta)}(\mathbb{R}^N, |y|^{-\alpha}) \hookrightarrow L_{M}^{q, \frac{(N-sp-\theta)q}{p}+qr}(\mathbb{R}^N, |y|^{-pr})
    \end{align}
    and the norms in these spaces share the same dilation invariance.

    \item For any $q \in [1, p^*_s(0,\theta)), \dot{W}^{s,p}_{\theta}(\mathbb{R}^N) \hookrightarrow L^{p^*_s(0,\theta)}(\mathbb{R}^N) \hookrightarrow L^{q, \frac{(N-sp-\theta)q}{p}}(\mathbb{R}^N).$
\end{enumerate}

\end{lemma}

For more properties of Lebesgue spaces, integral inequalities and boundedness properties of the operators in generalized Morrey spaces, see Sawano~\cite{MR4038542}.

\end{document}